\RequirePackage{fix-cm}
\documentclass{svjour3}                     % onecolumn (standard format)
\smartqed  % flush right qed marks, e.g. at end of proof

\input{preamble.tex}
\begin{document}

\title{Stochastic Dual Dynamic Programming for Multistage Stochastic Mixed-Integer Nonlinear Optimization}% : Generalized Conjugacy Cuts and Iteration Complexity Analysis}
%\thanks{Grants or other notes
%about the article that should go on the front page should be
%placed here. General acknowledgments should be placed at the end of the article.}

\titlerunning{SDDP for MS-MINLP}        % if too long for running head

\author{Shixuan Zhang         \and
        Xu Andy Sun %etc.
}

\institute{Shixuan Zhang \at
				H. Milton Stewart School of Industrial and Systems Engineering, Georgia Institute of Technology\\
				\email{szhang483@gatech.edu}           %  \\
%             \emph{Present address:} of F. Author  %  if needed
			\and
			Xu Andy Sun \at
				H. Milton Stewart School of Industrial and Systems Engineering, Georgia Institute of Technology\\
				% Tel.: +123-45-678910\\
				% Fax: +123-45-678910\\
				\email{andy.sun@isye.gatech.edu}
}

\date{Received: date / Accepted: date}

\maketitle

\begin{abstract}

In this paper, we study multistage stochastic mixed-integer nonlinear programs (MS-MINLP).
This general class of problems encompasses, as important special cases, multistage stochastic convex optimization with \emph{non-Lipschitzian} value functions and multistage stochastic mixed-integer linear optimization. 
We develop stochastic dual dynamic programming (SDDP) type algorithms with nested decomposition, deterministic sampling, and stochastic sampling. 
The key ingredient is a new type of cuts based on generalized conjugacy. 
Several interesting classes of MS-MINLP are identified, where the new algorithms are guaranteed to obtain the global optimum without the assumption of complete recourse. 
This significantly generalizes the classic SDDP algorithms. 
We also characterize the iteration complexity of the proposed algorithms.
In particular, for a $(T+1)$-stage stochastic MINLP with $d$-dimensional state spaces, to obtain an $\epsilon$-optimal root node solution, we prove that the number of iterations of the proposed deterministic sampling algorithm is upper bounded by $\mathcal{O}((\frac{2T}{\epsilon})^d)$, and is lower bounded by $\mathcal{O}((\frac{T}{4\epsilon})^d)$ for the general case or by $\mathcal{O}((\frac{T}{8\epsilon})^{d/2-1})$ for the convex case. 
This shows that the obtained complexity bounds are rather sharp. 
It also reveals that the iteration complexity depends \emph{polynomially} on the number of stages. We further show that the iteration complexity depends \emph{linearly} on $T$,  if all the state spaces are finite sets, or if we seek a $(T\epsilon)$-optimal solution when the state spaces are infinite sets, i.e. allowing the optimality gap to scale with $T$. 
To the best of our knowledge, this is the first work that reports global optimization algorithms as well as iteration complexity results for solving such a large class of multistage stochastic programs. The iteration complexity study resolves a conjecture by the late Prof.\,Shabbir Ahmed in the general setting of multistage stochastic mixed-integer optimization.

% \PACS{PACS code1 \and PACS code2 \and more}
% \subclass{MSC code1 \and MSC code2 \and more}
\end{abstract}

\section{Introduction.}
\label{sec:Introduction}

A multistage stochastic mixed-integer nonlinear program (MS-MINLP) is a sequential decision making problem under uncertainty with both continuous and integer decisions and nonconvex nonlinear objective function and constraints. This provides an extremely powerful modeling framework. Special classes of MS-MINLP, such as multistage stochastic linear programming (MS-LP) and mixed-integer linear programming (MS-MILP), have already found a wide range of applications in diverse fields such as
electric power system scheduling and expansion planning \cite{takriti_incorporating_2000,baringo_risk-constrained_2013,zou_multistage_2019,zou_partially_2018}, 
portfolio optimization under risk \cite{bradley_dynamic_1972,kusy_bank_1986,mulvey_stochastic_1992}, and production and capacity planning problems \cite{escudero_production_nodate,chen_scenario-based_2002,ahmed_multi-stage_2000,basciftci_adaptive_2019}, just to name a few. 

Significant progress has been made in the classic nested Benders decomposition (NBD) algorithms for solving MS-LP with  general scenario trees, and an efficient random sampling variation of NBD, the stochastic dual dynamic programming (SDDP) algorithm, is developed for MS-LP with scenario trees having stagewise independent structures. In the past few years, these algorithms are extended to solve MS-MILP \cite{philpott_midas_2020}. 
For example, SDDP is generalized to Stochastic Dual Dynamic integer Programming (SDDiP) algorithm for global optimization of MS-MILP with binary state variables \cite{zou_stochastic_2018,zou_multistage_2019}. Despite the rapid development, key challenges remain in further extending SDDP to the most general problems in MS-MINLP: 
1) There is no general cutting plane mechanism for generating exact under-approximation of nonconvex, discontinuous, or non-Lipschitzian value functions;
2) The computational complexity of SDDP-type algorithms is not well understood even for the most basic MS-LP setting, especially the interplay between iteration complexity of SDDP, optimality gap of obtained solution, number of stages, and dimension of the state spaces of the MS-MINLP. 

This paper aims at developing new methodologies for the solution of these challenges. In particular, we develop a unified cutting plane mechanism in the SDDP framework for generating exact under-approximation of value functions of a large class of MS-MINLP, and develop sharp characterization of the iteration complexity of the proposed algorithms. In the remaining of this section, we first give an overview of the literature, then summarize more details of our contributions.

\subsection{Literature Review}
Benders decomposition \cite{benders_partitioning_1962}, Dantzig-Wolfe decomposition \cite{dantzig_decomposition_1960}, and the L-shaped method \cite{slyke_l-shaped_1969} are standard algorithms for solving two-stage stochastic LPs. Nested decomposition procedures for deterministic models are developed in \cite{ho_nested_1974,glassey_nested_1973}. Louveaux \cite{louveaux_solution_1980} first generalized the two-stage L-shaped method to multistage quadratic problems. Nested Benders decomposition for MS-LP was first proposed in Birge \cite{birge_decomposition_1985} and Pereira and Pinto \cite{pereira_stochastic_1985}. SDDP, the sampling variation of NBD, was first proposed in \cite{pereira_multi-stage_1991}. The largest consumer of SDDP by far is in the energy sector, see e.g. \cite{pereira_multi-stage_1991,shapiro_risk_2013,flach_long-term_2010,lara_deterministic_2018}.

Recently, SDDP has been extended to SDDiP \cite{zou_stochastic_2018}. 
It is observed that the cuts generated from Lagrangian relaxation of the nodal problems in an MS-MILP are always tight at the given parent node's state, as long as all the state variables only take binary values and have complete recourse. 
From this fact, the SDDiP algorithm is proved to find an exact optimal solution in finitely many iterations with probability one. In this way, the SDDiP algorithm makes it possible to solve nonconvex problems through binarization of the state variables \cite{zou_multistage_2019,hjelmeland_nonconvex_2019}.
In addition, when the value functions of MS-MILP with general integer state variables are assumed to be Lipschitz continuous, which is a critical assumption, 
augmented Lagrangian cuts with an additional reverse norm term to the linear part obtained via augmented Lagrangian duality are proposed in \cite{ahmed2019stochastic}.

The convergence analysis of the SDDP-type algorithms begins with the linear cases \cite{philpott_convergence_2008,shapiro_analysis_2011,chen_convergent_1999,linowsky_convergence_2005,guigues2017dual}, where almost sure finite convergence is established based on the polyhedral nodal problem structures.
For convex problems, if the value functions are Lipschitz continuous and the state space is compact, asymptotic convergence of the under-approximation of the value functions leads to asymptotic convergence of the optimal value and optimal solutions \cite{girardeau_convergence_2015,guigues2016convergence}.
By constructing over-approximations of value functions, an SDDP with a deterministic sampling method with asymptotic convergence is proposed for the convex case in \cite{Baucke_Downward_Zakeri}. 
Upon completion of this paper, we became aware of the recent work~\cite{lan2020complexity}, which proves iteration complexity upper bounds for multistage convex programs under the assumption that all the value functions and their under-approximations are all Lipschitz continuous.
It is shown that for discounted problems, the iteration complexity depends linearly on the number of stages.
However, the following conjecture (suggested to us by the late Prof. Shabbir Ahmed) remains to be resolved, especially for the problems without convexity, Lipschitz continuity, or discounts:
\begin{conjecture}\label{conj:ComplexityMotivation}
The number of iterations needed for SDDP/SDDiP to find an optimal first-stage solution grows linearly in terms of the number of stages $T$, while it may depend nonlinearly on other parameters such as the diameter $D$ and the dimension $d$ of the state space.
\end{conjecture}
Our study resolves this conjecture by giving a full picture of the iteration complexity of 
SDDP-type algorithms in a general setting of MS-MINLP problems that allow exact Lipschitz regularization (defined in Section~\ref{subsec:PenaltyReformulation}).
In the following, we summarize our contributions.

\subsection{Contributions.}
\begin{enumerate}
    \item 
    To tackle the MS-MINLP problems without Lipschitz continuous value functions, which the existing SDDP algorithms and complexity analyses cannot handle,
    we propose \textit{a regularization approach} to provide a surrogate of the original problem such that the value functions become Lipschitz continuous.
    In many cases, the regularized problem preserves the set of optimal solutions.
    \item We use the theory of generalized conjugacy to develop a cut generation scheme, referred to as \textit{generalized conjugacy cuts}, that are valid for value functions of MS-MINLP.
    Moreover, generalized conjugacy cuts are shown to be tight to the regularized value functions. 
    The generalized conjugacy cuts can be replaced by linear cuts without compromising such tightness when the problem is convex.

    \item With the regularization and the generalized conjugacy cuts, we propose three algorithms for MS-MINLP, including nested decomposition for general scenario trees, SDDP algorithms with random sampling as well as deterministic sampling similar to \cite{Baucke_Downward_Zakeri} for stagewise independent scenario trees. 

    \item We obtain upper and lower bounds on the iteration complexity for the proposed SDDP with both sampling methods for MS-MINLP problems. 
    The complexity bounds show that in general, Conjecture~\ref{conj:ComplexityMotivation} holds if only we seek a $(T\epsilon)$-optimal solution, instead of an $\epsilon$-optimal first-stage solution for a $(T+1)$-stage problem, or when all the state spaces are finite sets.
\end{enumerate}

In addition, this paper contains the following contributions compared with the recent independent work~\cite{lan2020complexity}:
(1) We consider a much more general class of problems which are not necessarily convex.
As a result, all the iteration complexity upper bounds of the algorithms are also valid for these nonconvex problems.
(2) We use the technique of regularization to make the iteration complexity bounds independent of the subproblem oracles.
This is particularly important for the conjecture, since the Lipschitz constants of the under-approximation of value functions may exceed those of the original value functions.
(3) We propose matching lower bounds on the iteration complexity of the algorithms and characterize important cases for the conjecture to hold.

This paper is organized as follows. 
In Section~\ref{sec:ProblemFormulations} we introduce the problem formulation, regularization of the value functions, and the approximation scheme using generalized conjugacy.
Section~\ref{sec:DDPAlgorithms} proposes SDDP algorithms. 
Section~\ref{sec:ComplexityUpperBounds} investigates upper bounds on the iteration complexity of the proposed algorithm, while Section \ref{sec:ComplexityLowerBounds} focuses on lower bounds, therefore completes the picture of iteration complexity analysis.  We finally provide some concluding remarks in Section~\ref{sec:Concluding}.

\section{Problem Formulations.}
\label{sec:ProblemFormulations}

In this section, we first present the extensive and recursive formulations of multistage optimization.
Then we characterize the properties of the value functions, with examples to show that they may fail to be Lipschitz continuous.
With this motivation in mind, we propose a penalty reformulation of the multistage problem through regularization of value functions and show that it is equivalent to the original formulation for a broad class of problems. 
Finally, we propose generalized conjugacy cuts for under-approximation of value functions.

\subsection{Extensive and Recursive Formulation.}
\label{subsec:OriginalFormulations}
For a multistage stochastic program, let $\calT=(\calN,\calE)$ be the scenario tree, where $\calN$ is the set of nodes and $\calE$ is the set of edges. 
For each node $n\in\calN$, let $a(n)$ denote the parent node of $n$, $\calC(n)$ denote the set of child nodes of $n$, and $\calT(n)$ denote the subtree starting from the node $n$. 
Given a node $n\in\calN$, let $t(n)$ denote the stage that the node $n$ is in and let $T\coloneqq\max_{n\in\calN}t(n)$ denote the last stage of the tree $\calT$.
A node in the last stage is called a leaf node, otherwise a non-leaf node. 
The set of nodes in stage $t$ is denoted as $\calN(t)\coloneqq\{n\in\calN:t(n)=t\}$. We use the convention that the root node of the tree is denoted as $r\in\calN$ with $t(r)=0$ so the total number of stages is $T+1$. 
The parent node of the root node is denoted as $a(r)$, which is a dummy node for ease of notation. 

For every node $n\in\calN$, let $\calF_n$ denote the feasibility set in some Euclidean space of decision variables $(x_n,y_n)$ of the nodal problem at node $n$. 
We refer to $x_n$ as the state variable and $y_n$ as the internal variable of node $n$.
Denote the image of the projection of $\calF_n$ onto the subspace of the variable $x_n$ as $\calX_n$, which is referred to as the state space.
Let $x_{a(r)}=0$ serve as a dummy parameter and thus 
$\calX_{a(r)}=\{0\}$.	
The nonnegative nodal cost function of the problem at node $n$ is denoted as $f_n(x_{a(n)},y_n,x_n)$ and is defined on the set $\{(z,y,x):z\in\calX_{a(n)},(x,y)\in\calF_n\}$.
We allow $f_n$ to take the value $+\infty$ so indicator functions can be modeled as part of the cost.
Let $p_n>0$ for all $n\in\calN$ denote the probability that node $n$ on the scenario tree is realized. For the root node, $p_r=1$. 
The  transition probability that node $m$ is realized conditional on its parent node $n$ being realized is given by $p_{nm}:=p_m/p_n$ for all edges $(n,m)\in\calE$.

The multistage stochastic program considered in this paper is defined in the following extensive form: 
\begin{equation}\label{eq:ExtensiveForm}
	v^{\primal}\coloneqq\min_{\substack{(x_n,y_n)\in\calF_n,\\\forall\,n\in\calN}}\sum_{n\in\calN}p_n f_n(x_{a(n)},y_n,x_n).
\end{equation}

The recursive formulation of the problem~\eqref{eq:ExtensiveForm} is defined as
\begin{align}\label{eq:NodalProblem}
Q_{n}(x_{a(n)})\coloneqq\min_{(x_n,y_n)\in\calF_n} \biggl\{f_n(x_{a(n)},y_n,x_n)+\sum_{m\in\calC(n)}p_{nm}Q_{m}(x_n)\biggr\},
\end{align}
where $n\in\calT$ is a non-leaf node and $Q_n(x_{a(n)})$ is the \emph{value function} of node $n$. 
At a leaf node, the sum in \eqref{eq:NodalProblem} reduces to zero, as there are no child nodes $\calC(n)=\varnothing$. 
The problem on the right-hand side of \eqref{eq:NodalProblem} is called the \emph{nodal problem} of node $n$. 
Its objective function consists of the nodal cost function $f_n$ and the \emph{expected cost-to-go function}, which is denoted as $\calQ_n$ for future reference, i.e.
\begin{equation}\label{eq:DefinitionCostToGo}
\calQ_n(x_n)\coloneqq\sum_{m\in\calC(n)}p_{nm}Q_{m}(x_n).
\end{equation}

To ensure that the minimum in problem~\eqref{eq:ExtensiveForm} is well defined and finite, we make the following very general assumption on $f_n$ and $\calF_n$ throughout the paper. 
\begin{assumption}\label{assum:MinCondition}
	For every node $n\in\calN$,
	the feasibility set $\calF_n$ is compact, and
	the nodal cost function $f_n$ is nonnegative and lower semicontinuous (l.s.c.). 
	The sum $\sum_{n\in\calN}p_nf_n$ is a proper function, i.e., there exists $(x_n,y_n)\in\calF_n$ for all nodes $n\in\calN$ such that the sum $\sum_{n\in\calN}p_nf_n(x_{a(n)},y_n,x_n)<+\infty$.
\end{assumption}

Note that the state variable $x_{a(n)}$ only appears in the objective function $f_n$ of node $n$, not in the constraints. 
Perhaps the more common way is to allow $x_{a(n)}$ to appear in the constraints of node $n$. 
It is easy to see that any such constraint can be modeled by an indicator function of $(x_{a(n)},x_n,y_n)$ in the objective $f_n$.

\subsection{Continuity and Convexity of Value Functions.}
\label{sec:PropertiesValueFunction}
The following proposition presents some basic properties of the value function $Q_n$ under Assumption \ref{assum:MinCondition}.
\begin{proposition}\label{prop:OptimalValueFunctionProperties}
Under Assumption \ref{assum:MinCondition}, the value function $Q_n$ is lower semicontinuous (l.s.c.) for all $n\in\calN$. Moreover, for any node $n\in\calN$,
	\begin{enumerate}
            \item if $f_n(z,y,x)$ is Lipschitz continuous in the first variable $z$ 
            with constant $l_n$, i.e. $|f_n(z,y,x)-f_n(z',y,x)|\le l_n\|z-z'\|$ for any $z,z'\in\calX_{a(n)}$ and any $(x,y)\in\calF_n$, then $Q_n$ is also Lipschitz continuous with constant $l_n$;
            \item if $\calX_{a(n)}$ and $\calF_n$ are convex sets, and $f_n$ and $\calQ_n$ are convex functions, then $Q_n$ is also convex.
	\end{enumerate}
\end{proposition}
The proof is given in Section~\ref{app:proof:prop:OptimalValueFunctionProperties}.
When $Q_m$ is l.s.c.\ for all $m\in\calC(n)$, the sum $\sum_{m\in\calC(n)}p_{nm}Q_m$ is l.s.c..
Therefore, the minimum in the definition~\eqref{eq:NodalProblem} is well define, since $\calF_n$ is assumed to be compact. 

If the objective function $f_n(x_{a(n)},y_n,x_n)$ is not Lipschitz, e.g., when it involves an indicator function of $x_{a(n)}$, or equivalently when $x_{a(n)}$ appears in the constraint of the nodal problem of $Q_n(x_{a(n)})$, then the value function $Q_n$ may not be Lipschitz continuous, as is shown by the following examples.
\begin{example}
	\label{ex:ConvexNonLipschitz}
	Consider the convex nonlinear two-stage problem
	\begin{align*}
		v^*\coloneqq\min_{x,z,w}\biggl\{x + z \; : \ (z-1)^2 + w^2\le 1,\ w=x,\ x\in[0,1]\biggr\}.
	\end{align*}
	The objective function and all constraints are Lipschitz continuous. 
	The optimal objective value $v^*=0$, and the unique optimal solution is $(x^*,z^*,w^*)=(0,0,0)$. 
	At the optimal solution, the inequality constraint is active. 
	Note that the problem can be equivalently written as $v^*=\min_{0\le x\le 1}\ x + Q(x)$,
	where $Q(x)$ is defined on $[0,1]$ as $Q(x)\coloneqq\min\bigl\{z : \exists\, w \in \R, \, \st \, (z-1)^2 + w^2 \le 1, \, w=x\bigr\}=1-\sqrt{1-x^2}$, which  is not locally Lipschitz continuous at the boundary point $x=1$. Therefore, $Q(x)$ is not Lipschitz continuous on $[0,1]$.
	% We claim that $Q(y)$ is not Lipschitz continuous over $[0,1]$.
	% Otherwise suppose that $Q(y)$ is Lipschitz. 
	% Using $Q(1)=0$, $\vert{Q(1)-Q(y)}\vert=1+\sqrt{1-y^2}$. 
	% Then there exists $l>0$ such that $\vert Q(1)-Q(y)\vert\le l\cdot(1-y)$ for all $y\in[0,1]$. 
	% This implies that $1+\sqrt{1-y^2}\le l(1-y)$, and hence
	% $$(l-ly-1)^2-(1-y^2)=(l^2+1)y^2-2l(l-1)y+(l-1)^2-1\ge0,\ \forall\,y\in[0,1].$$
	% We find this is not true since when $y=\frac{l^2-l}{l^2+1}\in[0,1]$ (assuming $l>1$), we have the left hand side being $\frac{-2l}{l^2+1}<0$. Hence $Q(y)$ is not Lipschitz continuous over $[0,1]$.
\end{example}

\begin{example}
	\label{ex:MixedIntegerNonLipschitz}
	Consider the mixed-integer linear two-stage problem
	\begin{align*}
		v^*\coloneqq\min\biggl\{1-2x+z \; : \;  z\ge x,\ x\in[0,1],\ z\in\{0,1\}\biggr\}.
	\end{align*}
	The optimal objective value is $v^*=0$, and the unique optimal solution is $(x^*,z^*)=(1,1)$.
	Note that the problem can be equivalently written as $v^*=\min\{1-2x+Q(x) \; : \; 0\le x\le 1\}$,
	where the function $Q(x)$ is defined on $[0,1]$ as $Q(x)\coloneqq\min\{z\in\{0,1\}:z\ge x\}$, which equals $0$ if $x=0$, and $1$ for all $0<x\le1$, i.e. $Q(x)$ is discontinuous at $x=0$, therefore, it is not Lipschitz continuous on $[0,1]$.
\end{example}

These examples show a major issue with the introduction of value functions $Q_n$, namely $Q_n$ may fail to be Lipschitz continuous even when the original problem only has constraints defined by Lipschitz continuous functions. 
This could lead to failure of algorithms based on approximation of the value functions. 
% Fortunately, in many of these problems (such as the above listed examples), the value function need not be Lipschitz continuous on the entire domain for a cutting plane algorithm to find an optimal solution.
In the next section, we will discuss how to circumvent this issue without compromise of feasibility or optimality for a wide range of problems.

\subsection{Regularization and Penalty Reformulation.}
\label{subsec:PenaltyReformulation}
The main idea of avoiding failure of cutting plane algorithms in multistage dynamic programming is to use some Lipschitz continuous envelope functions to replace the original value functions, which we refer to as \emph{regularized value functions}. 

To begin with, we say a function $\psi:\R^d\to\R_+$ is a \emph{penalty function}, if $\psi(x)=0$ if and only if $x=0$, and the diameter of its level set $\lev_a(\psi):=\{x\in\R^d : \psi(x)\le\alpha\}$ approaches 0 when $a\to0$. 
In this paper, we focus on penalty functions that are locally Lipschitz continuous, the reason for which will be clear from Proposition~\ref{prop:InfConvolutionLipschitz}.

For each node $n$, we introduce a new variable $z_n$ as a local variable of node $n$ and impose the duplicating constraint $x_{a(n)}=z_n$. 
This is a standard approach for obtaining dual variables through relaxation (e.g.~\cite{zou_stochastic_2018}).
The objective function can then be written as $f_n(z_n,y_n,x_n)$. Let $\psi_n$ be a penalty function for node $n\in\calN$. 
The new coupling constraint is relaxed and penalized in the objective function by $\sigma_n\psi_n(x_{a(n)}-z_n)$ for some $\sigma_n>0$. 
Then the DP recursion with penalization becomes
\begin{equation}\label{eq:RegularizedNodalProblem}
Q^\Reg_{n}(x_{a(n)})\coloneqq\min_{\substack{(x_n,y_n)\in\calF_n,\\z_n\in\calX_{a(n)}}} \biggl\{f_n(z_n,y_n,x_n)+\sigma_n\psi_n(x_{a(n)}-z_n)+\sum_{m\in\calC(n)}p_{nm}Q^\Reg_m(x_n)\biggr\},
\end{equation}
for all $n\in\calN$, and $Q_n^\Reg$ is referred to as the regularized value function. 
By convention, $\calX_{a(r)}=\{x_{a(r)}\}=\{0\}$ and therefore, penalization $\psi_r(x_{a(r)}-z_r)\equiv 0$ for any $z_r\in\calX_{a(r)}$.
Since the state spaces are compact, without loss of generality, we can scale the penalty functions $\psi_n$ such that the Lipschitz constant of $\psi_n$ on $\calX_{a(n)}-\calX_{a(n)}$ is 1. 
The following proposition shows that $Q_n^\Reg$ is a Lipschitz continuous envelope function of $Q_n$ for all nodes $n$.
\begin{proposition}\label{prop:InfConvolutionLipschitz}
    Suppose $\psi_n$ is a $1$-Lipschitz continuous penalty function on the compact set $\calX_{a(n)}-\calX_{a(n)}$ for all $n\in\calN$.
    Then $Q_n^\Reg(x)\le Q_n(x)$ for all $x\in\calX_{a(n)}$ and $Q_n^\Reg(x)$ is $\sigma_n$-Lipschitz continuous on $\calX_{a(n)}$.
    Moreover, if the original problem~\eqref{eq:NodalProblem} and the penalty functions $\psi_n$ are convex, then $Q_n^\Reg(x)$ is also convex.
\end{proposition}
The key idea is that by adding a Lipschitz function $\psi_n$ into the nodal problem, we can make $Q_n^R(x)$ Lipschitz continuous even when $Q_n(x)$ is not. 
The proof is given in Section~\ref{app:proof:prop:InfConvolutionLipschitz}. 
The optimal value of the regularized root nodal problem
\begin{align}\label{eq:DefinitionRootNodeProblem}
    v^\reg\coloneqq\min_{(x_r,y_r)\in\calF_r}\biggl\{f_r(x_{a(r)},y_r,x_r)+\sum_{m\in\calC(r)}p_{rm}Q_m^\Reg(x_r)\biggr\}
\end{align}
is thus an underestimation of $v^\primal$, i.e. $v^\reg \le v^\primal$. 
For notational convenience, we also define the regularized expected cost-to-go function for each node $n$ as:
\begin{equation}\label{eq:DefinitionPenalCostToGo}
	\calQ_n^\Reg(x_n)\coloneqq\sum_{m\in\calC(n)}p_{nm}Q_{m}^\Reg(x_n).
\end{equation}
\begin{definition}
For any $\epsilon>0$, a feasible root node solution $(x_r,y_r)\in\calF_r$ is said to be \emph{$\epsilon$-optimal} to the regularized problem~\eqref{eq:RegularizedNodalProblem} if it satisfies $f_r(x_{a(r)},y_r,x_r)+\calQ_r^\Reg(x_r)\le v^\reg+\epsilon$.
\end{definition}

Next we discuss conditions under which $v^\reg=v^\primal$ and any optimal solution $(x_n,y_n)_{n\in\calN}$ to the regularized problem~\eqref{eq:RegularizedNodalProblem} is feasible and hence optimal to the original problem~\eqref{eq:NodalProblem}.
Note that by expanding $Q^\Reg_m$ in the regularized problem~\eqref{eq:RegularizedNodalProblem} for all nodes, we obtain the extensive formulation for the regularized problem:
\begin{equation}\label{eq:PenalizedExtensiveForm}
v^{\reg}=\min_{\substack{(x_n,y_n)\in\calF_n,n\in\calN\\z_n\in\calX_{a(n)}}}\sum_{n\in\calN}p_{n}\bigl(f_{n}(z_n,y_n,x_n)+\sigma_{n}\psi_{n}(x_{a(n)}-z_{n})\bigr).
\end{equation}
We refer to problem~\eqref{eq:PenalizedExtensiveForm} as the \textit{penalty reformulation} and make the following assumption on its exactness for the rest of the paper.
\begin{assumption}\label{assum:ExactPenalty}
	We assume that the penalty reformulation \eqref{eq:PenalizedExtensiveForm} is exact for the given penalty parameters $\sigma_n>0$, $n\in\calN$, i.e., any optimal solution of  \eqref{eq:PenalizedExtensiveForm} satisfies $z_n=x_{a(n)}$ for all $n\in\calN$.
\end{assumption}
Assumption \ref{assum:ExactPenalty} guarantees the solution of the regularized extensive formulation~\eqref{eq:PenalizedExtensiveForm} is feasible for the original problem \eqref{eq:ExtensiveForm}, then by the fact that $v^\reg\le v^\primal$, is also optimal to the original problem, we have $v^\reg=v^\primal$. 
Thus regularized value functions serve as a surrogate of the original value function, without compromise of feasibility of its optimal solutions.
A consequence of Assumption~\ref{assum:ExactPenalty} is that the original and regularized value functions coincide at all optimal solutions, the proof of which is given in Section~\ref{app:proof:lemma:RegOptimalValueFunction}.
\begin{lemma}
	\label{lemma:RegOptimalValueFunction}
	Under Assumption \ref{assum:ExactPenalty}, any optimal solution $(x_n,y_n)_{n\in\calN}$ to problem~\eqref{eq:ExtensiveForm} satisfies $Q_n^\Reg(x_{a(n)})=Q_n(x_{a(n)})$ for all $n\neq r$.
\end{lemma}

We illustrate the regularization on the examples through Figures~\ref{fig:ConvexNonLipschitz} and \ref{fig:MixedIntegerNonLipschitz}.
\begin{figure}[htbp]
\centering
\begin{subfigure}[b]{0.48\textwidth}
	\includegraphics[width=\textwidth]{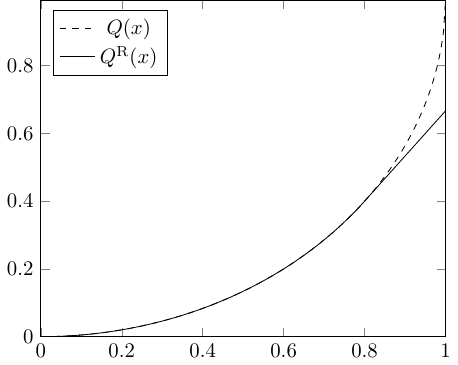}
	\caption{Value Functions in Example~\ref{ex:ConvexNonLipschitz}}
	\label{fig:ConvexNonLipschitz}
\end{subfigure}
\hfill
\begin{subfigure}[b]{0.48\textwidth}
	\includegraphics[width=\textwidth]{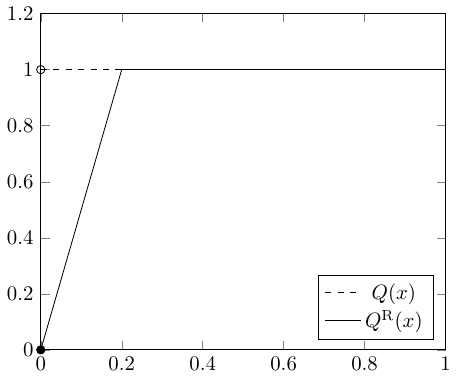}
	\caption{Value Functions in Example~\ref{ex:MixedIntegerNonLipschitz}}
	\label{fig:MixedIntegerNonLipschitz}
\end{subfigure}
\caption{Value functions in Examples~\ref{ex:ConvexNonLipschitz} and \ref{ex:MixedIntegerNonLipschitz}.}
\label{fig:ExamplesConvexNonconvex}
\end{figure}
In Figure~\ref{fig:ConvexNonLipschitz}, the value function $Q(x)$ derived in Example~\ref{ex:ConvexNonLipschitz} is not Lipschitz continuous at $x=1$.
With $\psi(x)=\norm{x}$ and $\sigma=4/3$, we obtain the regularized value function, which coincides with the original one on $[0, 0.8]$ and is Lipschitz continuous on the entire interval $[0,1]$.
In Figure~\ref{fig:MixedIntegerNonLipschitz}, the value function $Q(x)$ derived in Example~\ref{ex:ConvexNonLipschitz} is not continuous at $x=0$. 
With $\psi(x)=\norm{x}$, $\sigma=5$, we obtain the regularized value function, which coincides with the primal one on $\{0\}\cup[0.2,1]$ and is Lipschitz continuous on the entire interval $[0,1]$.
In both examples, it can be easily verified that the penalty reformulation is exact and thus preserves optimal solution.

We comment that Assumption~\ref{assum:ExactPenalty} can hold for appropriately chosen penalty factors in various mixed-integer nonlinear optimization problems, including
\begin{itemize}
	\item convex problems with interior feasible solutions,
	\item problems with finite state spaces,
	\item problems defined by mixed-integer linear functions, and
	\item problems defined by continuously differentiable functions,
\end{itemize}
if certain constraint qualification is satisfied and proper penalty functions are chosen.
We refer the readers to Section~\ref{sec:ExactPenaltyProblemClasses} in the Appendix for detailed discussions.
    We emphasize that Assumption~\ref{assum:ExactPenalty} should be interpreted as a restriction on the MS-MINLP problem class studied in this paper.
    Namely all problem instances in our discussion must satisfy Assumption~\ref{assum:ExactPenalty} with a given set of penalty functions \(\psi_n\) and penalty parameters \(\sigma_n\), while they can have other varying problem data such as the numbers of stages \(T\) and characteristics of the state spaces \(\calX_n\).
    In general, it is possible that a uniform choice of $\sigma_n$ needs to grow with $T$ to satisfy the assumption.

\subsection{Generalized Conjugacy Cuts and Value Function Approximation.}
\label{sec:GeneralizedConjugacy}
In this part, we first introduce generalized conjugacy cuts for nonconvex functions and then apply it to the under-approximation of value functions of MS-MINLP.

\subsubsection{Generalized Conjugacy Cuts.}
Let $Q:\calX\to\R_+\cup\{+\infty\}$ be a proper, l.s.c.\ function defined on a compact set $\calX\subseteq\R^d$. Let $\calU$ be a nonempty set for dual variables.
Given a continuous function $\Phi:\calX\times \calU\to \R$, the $\Phi$-conjugate of $Q$ (see e.g., Chapter 11-L in \cite{rockafellar2009variational}) is defined as 
\begin{equation}
Q^\Phi(u)=\max_{x\in \calX}\left\{\Phi(x,u)-Q(x)\right\}.\label{eq:DefPhiConjugate}
\end{equation}
% The $\Phi$-biconjugate of $Q$ is then defined as
% \begin{equation*}
% Q^{\Phi\Phi}(x)=\sup_{u\in \calU}\biggl\{\Phi(x,u)-Q^\Phi(u)\biggr\}=\sup_{u\in \calU}\min_{z\in \calX}\biggl\{\Phi(x,u)-\Phi(z,u)+Q(z)\biggr\}.
% \end{equation*}
The following generalized Fenchel-Young inequality holds by definition for any $x\in \calX$ and $u\in \calU$, 
\begin{equation*}
Q(x)+Q^\Phi(u)\ge\Phi(x,u).
\end{equation*}
For any $\hat{u}\in \calU$ and an associated maximizer $\hat{x}$ in \eqref{eq:DefPhiConjugate}, we define
\begin{align}\label{eq:DefGeneralizedCut}
C^{\Phi}(x\given\hat{u},\hat{v}) :=\hat{v}+\Phi(x,\hat{u})
\end{align}
where $\hat{v}\coloneqq-Q^\Phi(\hat{u})$. Then, the following inequality, derived from the generalized Fenchel-Young inequality, is valid for any $x\in\calX$,
\begin{align}\label{eq:GeneralizedConjugacyCutValidness}
Q(x) \ge  C^{\Phi}(x\given\hat{u},\hat{v}),
\end{align}
which we call a \emph{generalized conjugacy cut} for the target function $Q$. 

\subsubsection{Value Function Approximation.} \label{sec:ValueFunctionApprox}
The generalized conjugacy cuts can be used in the setting of an augmented Lagrangian dual~\cite{ahmed2019stochastic} with bounded dual variables.
For a nodal problem $n\in\calN, n\neq r$ and a point $\bar{x}\in\calX_{a(n)}$, define $\Phi_n^{\bar{x}}(x,u)\coloneqq-\innerprod{\lambda}{\bar{x}-x}-\rho\psi_n(\bar{x}-x)$, where $u:=(\lambda,\rho)\in\R^{d_n+1}$ are parameters.
Consider a compact set of parameters $\calU_n=\{(\lambda,\rho):\norm{\lambda}_*\le l_{n,\lambda},0\le\rho\le l_{n,\rho}\}$ with nonnegative bounds $l_{n,\lambda}$ and $l_{n,\rho}$, where $\norm{\cdot}_*$ is the dual norm of $\norm{\cdot}$. Consider the following dual problem 
\begin{align}
    \hat{v}_n:=\max_{(\lambda,\rho)\in\calU_n}\biggl\{\min_{z\in\calX_{a(n)}}\big[Q_n(z)+\innerprod{\lambda}{\bar{x}-z}+\rho\psi_n(\bar{x}-z)\big]\biggr\}.\label{eq:CutGenerationDualProblem}
\end{align}
Denote $\hat{z}_n$ and $(\hat{\lambda}_n,\hat{\rho}_n)$ as an optimal primal-dual solution of \eqref{eq:CutGenerationDualProblem}. 
The dual problem \eqref{eq:CutGenerationDualProblem} can be viewed as choosing $(\hat{\lambda}_n,\hat{\rho}_n)$ as the value of $\hat{u}$ in \eqref{eq:DefGeneralizedCut}, which makes the constant term $-Q^{\Phi}(\hat{u})$ as large as possible, thus makes the generalized conjugacy cut \eqref{eq:GeneralizedConjugacyCutValidness} as tight as possible. 
With this choice of the parameters, a generalized conjugacy cut for $Q_n$ at $\bar{x}$ is given by 
\begin{align}\label{eq:CutGeneration}
    Q_n(x) & \ge C_n^{\Phi_n^{\bar{x}}}(x\given\hat{\lambda}_n,\hat\rho_n,\hat{v}_n) \\ & = -\langle{\hat\lambda_n},{\bar{x}-x}\rangle-\hat\rho_n\psi_n(\bar{x}-x)+\hat{v}_n,\quad\forall\,x\in\calX_{a(n)}.\notag
\end{align}\vspace{-8mm}

\begin{proposition}\label{prop:GeneralizedCutTightness}
Given the above definition of \eqref{eq:CutGenerationDualProblem}-\eqref{eq:CutGeneration}, if $(\bar{x}_n,\bar{y}_n)_{n\in\calN}$ is an optimal solution to problem~\eqref{eq:ExtensiveForm} and the bound $l_{n,\rho}$ satisfies $l_{n,\rho}\ge\sigma_n$ for all nodes $n$, then for every node $n$, the generalized conjugacy cut \eqref{eq:CutGeneration} is tight at $\bar{x}_n$, i.e. 
$Q_n(\bar{x}_n) = C_n^{\Phi_n^{\bar{x}_n}}(\bar{x}_n\given\hat{\lambda}_n,\hat\rho_n,\hat{v}_n)$.
\end{proposition}
% This follows from Lemma \ref{lemma:RegOptimalValueFunction} that $Q_n(\bar{x}_n)=Q_n^\Reg(\bar{x}_n)$. 
The proof is given in Section~\ref{app:proof:prop:GeneralizedCutTightness}.
The proposition guarantees that, under Assumption~\ref{assum:ExactPenalty}, the generalized conjugacy cuts are able to approximate the value functions exactly at any state associated to an optimal solution.

In the special case where problem~\eqref{eq:NodalProblem} is convex and $\psi_n(x)=\norm{x}$ for all $n\in\calN$, the exactness of the generalized conjugacy cut holds even if we set 
$l_{n,\rho}=0$, i.e. the conjugacy cut is linear.
To be precise, we begin with the following lemma.
\begin{lemma}
	\label{lemma:ConvexCutTightness}
	Let $\calX\subset\R^d$ be a convex, compact set.
	Given a convex, proper, l.s.c.\ function $Q:\calX\to\R\cup\{+\infty\}$,  for any $x\in\calX$, the inf-convolution satisfies
	\begin{equation}
	    Q\square(\sigma\norm{\cdot})(x)\coloneqq\min_{z\in\calX}\{Q(z)+\sigma\norm{x-z}\}=\max_{\norm{\lambda}_*\le\sigma}\min_{z\in\calX}\{Q(z)+\innerprod{\lambda}{x-z}\}.\label{eq:ConvexCutInfConv}
	\end{equation}	
\end{lemma}
% This follows from the definition of the dual norm $\|\cdot\|_*$ and a standard saddle point theorem in convex optimization. 
The proof are given in Section~\ref{app:proof:lemma:ConvexCutTightness}.
Next we show the tightness in the convex case similar to Proposition~\ref{prop:GeneralizedCutTightness}. The proof is given in Section~\ref{app:proof:prop:ConvexCutTightness}.
\begin{proposition}
    \label{prop:ConvexCutTightness}
    Suppose \eqref{eq:NodalProblem} is convex and $\psi_n(x)=\norm{x}$ for all nodes $n$.
    Given the above definition of \eqref{eq:CutGenerationDualProblem}-\eqref{eq:CutGeneration}, if $(\bar{x}_n,\bar{y}_n)_{n\in\calN}$ is an optimal solution to problem~\eqref{eq:ExtensiveForm} and the bounds satisfy $l_{n,\lambda}\ge\sigma_n$, $l_{n,\rho}=0$ for all nodes $n$, then for every node $n$, the generalized conjugacy cut \eqref{eq:CutGeneration} is exact at $\bar{x}_n$, i.e. 
    $Q_n(\bar{x}_n) = C_n^{\Phi_n^{\bar{x}_n}}(\bar{x}_n\given\hat{\lambda}_n,\hat\rho_n,\hat{v}_n)$.
\end{proposition}

In this case, the generalized conjugacy reduces to the usual conjugacy for convex functions and the generalized conjugacy cut is indeed linear.
This enables approximation of the value function that preserves convexity.
\begin{remark}
    This proposition can be generalized to special nonconvex problems with $Q$ extensible to a Lipschitz continuous convex function defined on the convex hull $\conv\calX$.
    This is true if $\calX$ is the finite set of extreme points of a polytope, e.g., $\{0,1\}^d$.
    The above discussion provides an alternative explanation of the exactness of the Lagrangian cuts in SDDiP \cite{zou_stochastic_2018} assuming relatively complete recourse.
\end{remark}

\section{Nested Decomposition and Dual Dynamic Programming Algorithms}
\label{sec:DDPAlgorithms}

In this section, we introduce a nested decomposition algorithm for general scenario trees, and two dual dynamic programming algorithms for stagewise independent scenario trees.
Since the size of the scenario tree could be large, we focus our attention to finding an $\epsilon$-optimal root node solution $x_r^*$ (see definition \eqref{eq:DefinitionRootNodeProblem},) rather than an optimal solution $\{x^*_n\}_{n\in\calT}$ for the entire tree. 

\subsection{Subproblem Oracles.}

Before we propose the new algorithms, we first define subproblem oracles, which we will use to describe the algorithms and conduct complexity analysis.
A subproblem oracle is an oracle that takes subproblem information together with the current algorithm information to produce a solution to the subproblem.
With subproblem oracles, we can describe the algorithms consistently regardless of convexity.

We assume three subproblem oracles in this paper, corresponding to the forward steps and backward steps of non-root nodes, and the root node step in the algorithms.
For non-root nodes, we assume the following two subproblem oracles.
\begin{definition}[Forward Step Subproblem Oracle for Non-Root Nodes]
    \label{def:NonrootForwardStepOracle}
	Consider the following subproblem for a non-root node $n$, 
	\begin{equation}\label{eq:NonRootForwardOracleProblem}
		\min_{\substack{(x,y)\in\calF_n,\\z\in\calX_{a(n)}}}\left\{f_n(z,y,x)+\sigma_n\psi_n(x_{a(n)}-z)+\Theta_n(x)\right\}.\tag{F}
	\end{equation}
	where the parent node's state variable $x_{a(n)}\in\calX_{a(n)}$ is a given parameter and $\Theta_n:\calX_n\to\bar{\R}$ is a l.s.c.\ function, representing an under-approximation of the expected cost-to-go function $\calQ_n(x)$ defined in \eqref{eq:DefinitionCostToGo}. 
	The forward step subproblem oracle finds an optimal solution of \eqref{eq:NonRootForwardOracleProblem} given $x_{a(n)}$ and $\Theta_n$, that is, we denote this oracle as a mapping $\scrO^\Fwd_n$ that takes $(x_{a(n)},\Theta_n)$ as input and outputs an optimal solution $(x_n,y_n,z_n)$ of \eqref{eq:NonRootForwardOracleProblem} for $n\neq r$.
% 	More precisely, let $(x_n,y_n,z_n)$ denote an optimal solution of \eqref{eq:NonRootForwardOracleProblem}.
% 	Then, the forward step subproblem oracle is defined formally as the mapping $\scrO^\Fwd_n:(x_{a(n)},\Theta_n)\mapsto(x_n,y_n,z_n)$ for $n\neq r$.
\end{definition}
Recall that the values \(\sigma_n\) for all \(n\in\calN\) in~\eqref{eq:NonRootForwardOracleProblem} are the chosen penalty parameters that satisfy Assumption~\ref{assum:ExactPenalty}.
In view of Propositions~\ref{prop:GeneralizedCutTightness} and~\ref{prop:ConvexCutTightness}, we set \(l_{n,\lambda}\ge\sigma_n\) and \(l_{n,\rho}=0\) for the convex case with \(\psi_n=\norm{\cdot}\); or \(l_{n,\rho}\ge\sigma_n\) otherwise for the dual variable set \(\calU_n:=\{(\lambda,\rho):\norm{\lambda}_*\le l_{n,\lambda},0\le\rho\le l_{n,\rho}\}\) in the next definition.
\begin{definition}[Backward Step Subproblem Oracles for Non-Root Nodes]
    \label{def:NonrootBackwardStepOracle}
	Consider the following subproblem for a non-root node $n$,
	\begin{equation}\label{eq:NonRootBackwardOracleProblem}
		\max_{(\lambda,\rho)\in\calU_n}\min_{\substack{(x,y)\in\calF_n,\\z\in\calX_{a(n)}}}\left\{f_n(z,y,x)+\bangle{\lambda}{x_{a(n)}-z}+\rho\psi_n(x_{a(n)}-z)+\Theta_n(x)\right\},\tag{B}
	\end{equation}
	where the parent node's state variable $x_{a(n)}\in\calX_{a(n)}$ is a given parameter and $\Theta_n:\calX_n\to\bar{\R}$ is a l.s.c.\ function, representing an under-approximation of the expected cost-to-go function. 
	The backward step subproblem oracle finds an optimal solution of \eqref{eq:NonRootBackwardOracleProblem} for the given $x_{a(n)}$ and $\Theta_n$. 
	Similarly, we denote this oracle as a mapping $\scrO^\Bwd_n$ that takes $(x_{a(n)},\Theta_n)$ as input and outputs an optimal solution $(x_n,y_n,z_n;\lambda_n,\rho_n)$ of \eqref{eq:NonRootBackwardOracleProblem} for $n\neq r$.
% 	More precisely, let $(x_n,y_n,z_n;\lambda_n,\rho_n)$ be an optimal primal-dual solution pair of \eqref{eq:NonRootBackwardOracleProblem}, the backward step oracle is defined formally as the mapping $\scrO^\Bwd_n:(x_{a(n)},\Theta_n)\mapsto(x_n,y_n,z_n;\lambda_n,\rho_n)$ for $n\neq r$.
\end{definition}
For the root node, we assume the following subproblem oracle.
\begin{definition}[Subproblem Oracle for the Root Node]
    \label{def:RootNodeOracle}
	Consider the following subproblem for the root node $r\in\calN$, 
	\begin{equation}\label{eq:RootOracleProblem}
		\min_{(x,y)\in\calF_r}\left\{f_r(x_{a(r)},y,x)+\Theta_r(x)\right\},\tag{R}
	\end{equation}
	where $\Theta_r:\calX_r\to\bar{\R}$ is a l.s.c.\ function, representing an under-approximation of the expected cost-to-go function. 
	The subproblem oracle for the root node is denoted as $\scrO_r$ that takes $\Theta_r$ as input and outputs an optimal solution $(x_r,y_r)$ of \eqref{eq:RootOracleProblem} for the given function $\Theta_r$. 
\end{definition}
These subproblem oracles \(\scrO^\Fwd_n\), \(\scrO^\Bwd_n\), including the parameters \(\sigma_n\), \(l_{n,\lambda}\), and \(l_{n,\rho}\) for all \(n\neq r\), and \(\scrO_r\) will be given as inputs to the algorithms. 
They may return any optimal solution to the corresponding nodal subproblem. 
For numerical implementation, they are usually handled by subroutines or external solvers.

\subsection{Under- and Over-Approximations of Cost-to-go Functions.}
We first show how to iteratively construct under-approximation of expected cost-to-go functions using the generalized conjugacy cuts developed in Section~\ref{sec:GeneralizedConjugacy}.
The under-approximation serves as a surrogate of the true cost-to-go function in the algorithm.
Let $i\in\N$ be the iteration index of an algorithm.
Assume $(x_n^i,y_n^i)_{n\in\calN}$ are feasible solutions to the regularized nodal problem~\eqref{eq:RegularizedNodalProblem} in the $i$-th iteration.
Then the under-approximation of the expected cost-to-go function is defined recursively
from leaf nodes to the root node,
and inductively for $i\in\N$ as
\begin{equation}\label{eq:UnderApproximationCostToGo}
	\ulcQ_n^i(x)\coloneqq\max\left\{\ulcQ_n^{i-1}(x),\;\sum_{m\in\calC(n)}p_{nm}C_m^i(x\given\hat\lambda_m^i,\hat\rho_m^i,\ubar{v}_m^i)\right\},\quad \forall x\in\calX_{n},
\end{equation}
where $\ulcQ_n^0\equiv 0$ on $\calX_{n}$.
In the definition~\eqref{eq:UnderApproximationCostToGo}, $C_m^i$ is the generalized conjugacy cut for $Q_m$ at $i$-th iteration and $\Phi_m^{x_{n}^i}(x,\lambda,\rho)=-\bangle{\lambda}{x_n^i - x}-\rho\psi_n(x_{n}^i-x)$ (cf. \eqref{eq:CutGenerationDualProblem}-\eqref{eq:CutGeneration}), that is, 
\begin{equation}\label{eq:UnderApproximationCut}
	C_m^i(x\given\hat\lambda_m^i,\hat\rho_m^i,\ubar{v}_m^i)\coloneqq-\bangle{\hat\lambda_m^i}{x_{n}^i-x}-\hat\rho_m^i\psi_m(x_{n}^i-x)+\ubar{v}_m^i,
\end{equation}
where $(\hat{x}_m^i,\hat{y}_m^i\hat{z}_m^i;\hat\lambda_m^i,\hat\rho_m^i)=\scrO^\Bwd_m(x_{n}^i,\ulcQ_m^i)$, and $\ubar{v}_m^i$ satisfies
\begin{align}
	&\ubar{v}_m^i=f_m(\hat{z}_m^i,\hat{y}_m^i,\hat{x}_m^i)+\bangle{\hat\lambda_m^i}{x_{n}^i-\hat{z}_m^i}+\hat\rho_m^i\psi_m(x_{n}^i-\hat{z}_m^i)+\ulcQ_m^i(\hat{x}_m^i).\label{eq:UnderApproximationUpdateFormula}
\end{align}
The next proposition shows that $\ulcQ_n^i$ is indeed an under-approximation of $\calQ_n$, the proof of which is given in Section~\ref{app:proof:prop:UnderApproximationValidness}.
\begin{proposition}\label{prop:UnderApproximationValidness}
	For any $n\in\calN$, and $i\in\N$, $\ulcQ_n^i(x)$ is $(\sum_{m\in\calC(n)}p_{nm}(l_{m,\lambda}+l_{m,\rho}))$-Lipschitz continuous and
	\begin{align*}
	    \calQ_n(x) \ge \ulcQ_n^i(x), \quad\forall x\in\calX_{n}.
	\end{align*}
\end{proposition}

Now, we propose the following over-approximation of the regularized expected cost-to-go functions, which is used in the sampling and termination of the proposed nested decomposition and dual dynamic programming algorithms.
For $i\in\N$, at root node $r$, let $(x_r^i,y_r^i)=\scrO_r(\ulcQ_r^{i-1})$, and, at each non-root node $n$, let $(x_n^i,y_n^i,z_n^i)=\scrO^\Fwd_n(x_{a(n)}^i,\ulcQ_n^{i-1})$.
Then the over-approximation of the regularized expected cost-to-go function is defined recursively, from leaf nodes to the child nodes of the root node,
and inductively for $i\in\N$ by
\begin{equation}\label{eq:OverApproximationCostToGo}
	\olcQ_n^i(x)\coloneqq\begin{cases}
        \conv\bigg\{\olcQ_n^{i-1}(x),  \displaystyle\sum_{m\in\calC(n)}p_{nm}\left(\bar{v}_m^i+\sigma_m\nVert{x-x_{n}^i}\right)\bigg\},\,\text{if~\eqref{eq:RegularizedNodalProblem} is convex}\\[1mm]
        \min\bigg\{\olcQ_n^{i-1}(x), \displaystyle\sum_{m\in\calC(n)}p_{nm}\left(\bar{v}_m^i+\sigma_m\nVert{x-x_{n}^i}\right)\bigg\},\,\text{otherwise}
	\end{cases}
\end{equation}
where $\olcQ_n^0\equiv+\infty$ for any non-leaf node $n\in\calN$, $\olcQ_n^i\equiv0$ for any iteration $i\in\bbN$ and any leaf node $n$, and $\bar{v}_m^i$ satisfies
\begin{equation}\label{eq:OverApproximationUpdateFormula}
	\bar{v}_m^i=f_m({z}_m^i,{y}_m^i,{x}_m^i)+\sigma_m\psi_m(x_{n}^i-{z}_m^i)+\olcQ_m^{i}({x}_m^i).
\end{equation}
Here, the operation $\conv\{f,g\}$ forms the convex hull of the union of the epigraphs of any continuous functions $f$ and $g$ defined on the space $\bbR^d$. 
More precisely using convex conjugacy, we define
\begin{align}\label{eq:ConvexHullFunctionOperation}
	\conv\{f,g\}(x)&:=\left(\min\{f(x),g(x)\}\right)^{**}\\
	&=\sup_{\lambda\in\bbR^d}\inf_{z\in\bbR^d}\left\{\min\{f(z),g(z)\}+\bangle{\lambda}{x-z}\right\}	.\notag
\end{align}
The key idea behind the upper bound function \eqref{eq:OverApproximationCostToGo} is to exploit the Lipschitz continuity of the regularized value function $Q_m^R(x)$. 
In particular, it would follow from induction that $\bar{v}_m^i$ is an upper bound on $Q_m^R(x_n^i)$, and then, by the $\sigma_m$-Lipschitz continuity of $Q_m^R(x)$, we have $\bar{v}_m^i + \sigma_m\|x-x_n^i\|\ge Q_m^R(x_n^i) + \sigma_m\|x-x_n^i\|\ge Q_m^R(x)$ for all $x\in X_n$. 
The next proposition summarizes this property, with the proof given in Section~\ref{app:proof:prop:OverApproximationValidness}.
\begin{proposition}\label{prop:OverApproximationValidness}
	For any non-root node $n\in\calN$ and $i\ge1$, $\olcQ_n^i(x)$ is $(\sum_{m\in\calC(n)}p_{nm}\sigma_m)$-Lipschitz continuous.
	Moreover, we have $\bar{v}_m^i\ge Q_m^\Reg(x_{n}^i)$ for any node $m\in\calC(n)$ and thus
	\begin{align*}
	    \olcQ_n^i(x)\ge \calQ_n^\Reg(x),\quad\forall\,x\in\calX_{n}. 
	\end{align*}
\end{proposition}

\subsection{A Nested Decomposition Algorithm for General Trees.}
We first propose a nested decomposition algorithm in Algorithm~\ref{alg:DualDP} for a general scenario tree.
In each iteration $i$, Algorithm~\ref{alg:DualDP} carries out a forward step, a backward step, and a root node update step. 
In the forward step, the algorithm proceeds from $t=1$ to $T$ by solving all the nodal subproblems with the current under-approximation of their cost-to-go functions in stage $t$.
After all the state variables $x_n^i$ are obtained for nodes $n\in\calN$, the backward step goes from $t=T$ back to $1$.
At each node $n$ in stage $t$, it first updates the under-approximation of the expected cost-to-go function.
Next it solves the dual problem to obtain an optimal primal-dual solution pair $(\hat{x}_n^i,\hat{y}_n^i,\hat{z}_n^i;\hat\lambda_n^i,\hat\rho_n^i)$, which is used to construct a generalized conjugacy cut using~\eqref{eq:UnderApproximationCut}, together with values $\ubar{v}_n^i$ and $\bar{v}_n^i$ calculated with~\eqref{eq:UnderApproximationUpdateFormula} and \eqref{eq:OverApproximationUpdateFormula}.
Finally the algorithm updates the root node solution using the updated under-approximation of the cost-to-go function, and determines the new lower and upper bounds.
The incumbent solution $(x_r^*,y_r^*)$ may also be updated as the algorithm output at termination, although it is not used in the later iterations.

\begin{algorithm}[ht]
    \caption{A Nested Decomposition Algorithm for a General Tree}
	\label{alg:DualDP}
	\begin{algorithmic}[1]
		\Require{scenario tree $\calT=(\calN,\calE)$ with subproblem oracles $\scrO_r,\scrO_n^\Fwd,\scrO_n^\Bwd,n\neq r$, and $\epsilon>0$} 
		\Ensure{an $\epsilon$-optimal root node solution $(x_r^*,y_r^*)$ to the regularized problem~\eqref{eq:RegularizedNodalProblem}}
		\State{Initialize: $i\leftarrow1$; $\ulcQ_n^0\leftarrow0,\;\olcQ_n^0\leftarrow+\infty\,\forall\,n:\calC(n)\neq\varnothing$ and $\olcQ_n^0\leftarrow0\,\forall\,n:\calC(n)=\varnothing$}
		\State{Evaluate $(x_r^1,y_r^1)=\scrO_r(0)$}
		\State{Set $\LB\leftarrow f_r(x_{a(r)},y_r^1,x_r^1),\;\UB\leftarrow +\infty$}
		\While{$\UB-\LB>\epsilon$}
		\For{$t=1,\dots,T-1$}
		\Comment{$i$-th forward step}
			\For{$n\in\calN(t)$}
				\State{Evaluate $(x_n^i,y_n^i,z_n^i)=\scrO_n^\Fwd(x_{a(n)}^i,\ulcQ_n^{i-1})$}
			\EndFor
		\EndFor
		\For{$t=T,\dots,1$}
		\Comment{$i$-th backward step}
			\For{$n\in\calN(t)$}
				\State{Update $\ulcQ_n^i$ and $\olcQ_n^i$ using \eqref{eq:UnderApproximationCostToGo} and~\eqref{eq:OverApproximationCostToGo}}
				\State{Evaluate $(\hat{x}_n^i,\hat{y}_n^i,\hat{z}_n^i;\hat\lambda_n^i,\hat\rho_n^i)=\scrO_n^\Bwd(x_{a(n)}^i,\ulcQ_n^i)$}
				\State{Calculate $C_n^i,\ubar{v}_n^i,$ and $\bar{v}_n^i$ using ~\eqref{eq:UnderApproximationCut}, \eqref{eq:UnderApproximationUpdateFormula}, and \eqref{eq:OverApproximationUpdateFormula}}
			\EndFor
		\EndFor
		\State{Update $\ulcQ_r^i$ and $\olcQ_r^i$ using \eqref{eq:UnderApproximationCostToGo} and~\eqref{eq:OverApproximationCostToGo}}
		\Comment{root node update}
		\State{Evaluate $(x_r^{i+1},y_r^{i+1})=\scrO_r(\ulcQ_r^i)$}
		\State{Update $\LB\leftarrow f_r(x_{a(r)},y_r^{i+1},x_r^{i+1})+\ulcQ_r^i(x_r^{i+1})$}
		\If{$\UB> f_r(x_{a(r)},y_r^{i+1},x_r^{i+1})+\olcQ_r^i(x_r^{i+1})$}
		    \State{Update $\UB\leftarrow f_r(x_{a(r)},y_r^{i+1},x_r^{i+1})+\olcQ_r^i(x_r^{i+1})$}
		    \State{Set $(x_r^*,y_r^*)=(x_r^{i+1},y_r^{i+1})$}
		\EndIf
		\State{$i\leftarrow i+1$}
		\EndWhile
	\end{algorithmic}
\end{algorithm}

Algorithm~\ref{alg:DualDP} solves the regularized problem~\eqref{eq:RegularizedNodalProblem} for an $\epsilon$-optimal root node solution. To justify the $\epsilon$-optimality of the output of the algorithm, we have the following proposition, the proof of which is given in Section~\ref{app:proof:prop:RootNodeOptimality}.
\begin{proposition}\label{prop:RootNodeOptimality}
	Given any $\epsilon > 0$, if $\UB-\LB\le \epsilon$, then the returned solution $(x_r^*,y_r^*)$ is an $\epsilon$-optimal root node solution to the regularized problem~\eqref{eq:RegularizedNodalProblem}.
	In particular, if $\olcQ_r^i(x_r^{i+1})-\ulcQ_r^i(x_r^{i+1})\le\epsilon$ for some iteration index $i$, then $\UB-\LB\le\epsilon$ and Algorithm~\ref{alg:DualDP} terminates after the $i$-th iteration.
\end{proposition}

\subsection{A Deterministic Sampling Dual Dynamic Programming Algorithm.}
\label{subsec:StagewiseIndependentProblems}
Starting from this subsection, we turn our attention to stagewise independent stochastic problems, 
which is defined in the following assumption.
\begin{assumption}\label{assum:StagewiseIndependence}
	For any $t=1,\dots,T-1$ and any $n,n'\in\calN(t)$, the state space, the transition probabilities, as well as the data associated with the child nodes $\calC(n)$ and $\calC(n')$ are identical.
	In particular, this implies $\calQ_n(x)=\calQ_{n'}(x)\eqqcolon\calQ_t(x)$ for all $x\in\calX_n=\calX_{n'}\eqqcolon\calX_t\subseteq\R^{d_t}$.
\end{assumption}
We denote $n\sim n'$ for $n,n'\in\calN(t)$ for some $t=1,\dots,T-1$, if the nodes $n,n'$ are defined by identical data. 
We then use $\tcN(t):=\calN(t)/\sim$ to denote the set of nodes with size $N_t:=\vert{\tcN(t)}\vert$ that are defined by distinct data in stage $t$ for all $t=1,\dots,T-1$, i.e. $\tcN:=\cup_{t=0}^{T}\tcN(t)$
forms a recombining scenario tree \cite{zou_multistage_2019}.
For each node $m\in\tcN(t)$, we denote $p_{t-1,m}:=p_{nm}$ for any $n\in\tcN(t-1)$ since $p_{n,m}=p_{n',m}$ for any $n,n'\in\tcN(t-1)$.
Due to stagewise independence, it suffices to keep track of the state of each stage in the algorithm, instead of the state of each node.
To be consistent, we also denote the root node solution as $(x_0^i,y_0^i)$ for $i\in\N$.
We present the algorithm in Algorithm~\ref{alg:DeterministicDualDP}.

Similar to Algorithm~\ref{alg:DualDP}, each iteration in Algorithm~\ref{alg:DeterministicDualDP} consists of a forward step, a backward step, and a root node update step.
In particular, at a node $n\in\tcN(t)$ with $t<T$, the forward step proceeds to a child node $m\in\tcN(t+1)$, where the \textit{approximation gap} $\gamma_m^i:=\olcQ_t^{i-1}(x_m^i)-\ulcQ_t^{i-1}(x_m^i)$
is among the largest of all the approximation gaps of states $x_{m'}^i$ of nodes $m'\in\tcN(t+1)$.
Then the state variable of node $m$ is considered the state variable of stage $t(m)$ in the iteration $i$.
Due to stagewise independence, the backward step at each stage $t$ only need to generate cuts for the nodes in the recombining tree $\tcN$.
The optimality of the returned solution $(x_0^*,y_0^*)$ is guaranteed by Proposition~\ref{prop:RootNodeOptimality}.

\begin{algorithm}[ht]
	\caption{Deterministic Sampling Dual Dynamic Programming Algorithm}
	\label{alg:DeterministicDualDP}
	\begin{algorithmic}[1]
		\Require{recombining scenario tree $\tcN$ with subproblem oracles $\scrO_r,\scrO_n^\Fwd,\scrO_n^\Bwd,n\neq r$, and $\epsilon>0$}
		\Ensure{an $\epsilon$-optimal root node solution $(x_0^*,y_0^*)$ to the regularized problem~\eqref{eq:RegularizedNodalProblem}}
		\State{Initialize: $i\leftarrow1$; $\ulcQ_t^0\leftarrow0,\,\forall\,t,\;\olcQ_t^0,\leftarrow+\infty\,\forall\,t\le T-1$ and $\olcQ_T^0\leftarrow0$}
		\State{Evaluate $(x_0^1,y_0^1)=\scrO_r(0)$}
		\State{Set $\LB\leftarrow f_r(x_{a(r)},y_0^1,x_0^1),\;\UB\leftarrow +\infty$}
		\While{$\UB-\LB>\epsilon$}
		\For{$t=1,\dots,T-1$}
		\Comment{$i$-th forward step}
			\For{$n\in\tcN(t)$}
				\State{Evaluate $(x_n^i,y_n^i,z_n^i)=\scrO_n^\Fwd(x_{t-1}^i,\ulcQ_t^{i-1})$}\label{alg:DeterministicDualDP:ForwardStep}
				\State{Calculate the gap $\gamma_n^i=\olcQ_t^{i-1}(x_{n}^i)-\ulcQ_t^{i-1}(x_{n}^i)$}
			\EndFor
			\State{Select any $n^*(t)\in\{n\in\calN(t):\gamma_n^i\ge\gamma_{n'}^i,\ \forall\,n'\in\tcN(t)\}$, and let $x_t^i\leftarrow x_{n^*(t)}^i$}\label{alg:DeterministicDualDP:MaxGap}
		\EndFor
		\For{$t=T,\dots,1$}
		\Comment{$i$-th backward step}
			\State{Update $\ulcQ_t^i$ and $\olcQ_t^i$ using \eqref{eq:UnderApproximationCostToGo} and \eqref{eq:OverApproximationCostToGo}}
			\For{$n\in\tcN(t)$}
				\State{Evaluate $(\hat{x}_n^i,\hat{y}_n^i,\hat{z}_n^i;\hat\lambda_n^i,\hat\rho_n^i)=\scrO_n^\Bwd(x_{t-1}^i,\ulcQ_t^i)$}
				\State{Calculate $C_n^i,\ubar{v}_n^i,\bar{v}_n^i$ using \eqref{eq:UnderApproximationCut}, \eqref{eq:UnderApproximationUpdateFormula}, and \eqref{eq:OverApproximationUpdateFormula}}\label{alg:DeterministicDualDP:cut}
			\EndFor
		\EndFor
		\State{Update $\ulcQ_0^i$ and $\olcQ_0^i$ using \eqref{eq:UnderApproximationCostToGo} and \eqref{eq:OverApproximationCostToGo}}
		\Comment{root node update}
		\State{Evaluate $(x_0^{i+1},y_0^{i+1})=\scrO_r(\ulcQ_0^i)$}
		\State{Update $\LB\leftarrow f_r(x_{a(r)},y_0^{i+1},x_0^{i+1})+\ulcQ_0^i(x_0^{i+1})$}
		\If{$\UB> f_r(x_{a(r)},y_0^{i+1},x_0^{i+1})+\olcQ_0^i(x_0^{i+1})$} \label{alg:DeterministicDDP:UB1}
		    \State{Update $\UB\leftarrow f_r(x_{a(r)},y_0^{i+1},x_0^{i+1})+\olcQ_0^i(x_0^{i+1})$}\label{alg:DeterministicDDP:UB2}
		    \State{Set $(x_0^*,y_0^*)=(x_0^{i+1},y_0^{i+1})$}
		\EndIf
		\State{$i\leftarrow i+1$}
		\EndWhile
	\end{algorithmic}
\end{algorithm}

\subsection{A Stochastic Sampling Dual Dynamic Programming Algorithm.}
Now we present a stochastic dual dynamic programming algorithm, which uses stochastic sampling rather than deterministic sampling. 
So, instead of traversing the scenario tree and finding a path with the largest approximation gap, the stochastic sampling algorithm generates $M$ scenario paths before an iteration begins for some $M\ge 1$.
To be precise, we introduce the following notations.
Let $\calP=\prod_{t=1}^T\tcN(t)$ denote all possible scenario paths from stage 1 to stage $T$.
A scenario path is denoted as a $T$-element sequence $P=(n_1,\dots,n_T)\in\calP$, where $n_t\in\tcN(t)$ for each $t=1,\dots,T$.
In the $i$-th iteration, we sample $M$ independent scenario paths $\scrP^i=\{P^{i,1},\dots,P^{i,M}\}$, and we use $P^{i,j}_t$ to denote the $t$-th node in the scenario path $P^{i,j}$, i.e., the node in the $t$-th stage of the $j$-th scenario path in the $i$-th iteration, for $1\le j\le M$ and $1\le t\le T$.
Since in each iteration, the solutions and the approximations depend on the scenario path $P^{i,j}$, we use two superscripts $i$ and $j$ for solutions and cuts, where a single superscript $i$ is used in the deterministic sampling algorithm.
In addition, for every node $n\in\tcN(t)$ for some stage $t$, the under-approximation of the expected cost-to-go function is updated over all scenario path index $j=1,\dots,M$, with $M$ cuts in total,
i.e.,
\begin{equation}\label{eq:StochasticSamplingUnderApproximation}
	\ulcQ_t^{i}(x)\coloneqq\max\bigg\{\ulcQ_t^{i-1}(x),\sum_{m\in\tcN(t+1)}p_{tm}C_m^{i,j}(x\given\hat\lambda_m^{i,j},\hat\rho_m^{i,j},\ubar{v}_m^{i,j}),\,1\le j\le M\bigg\},
\end{equation}
where $C_m^{i,j}$ is the generalized conjugacy cut generated with $(\hat{x}_m^{i,j},\hat{y}_m^{i,j},\hat{z}_m^{i,j};\hat\lambda_m^{i,j},\hat\rho_m^{i,j})=\scrO_m^\Bwd(x_{n}^{i,j},\ulcQ_{t+1}^i)$ using formula~\eqref{eq:UnderApproximationCut}.
With these notations, the algorithm is displayed in Algorithm~\ref{alg:StochasticDualDP}.

\begin{algorithm}[ht]
	\caption{Stochastic Sampling Dual Dynamic Programming Algorithm}
	\label{alg:StochasticDualDP}
	\begin{algorithmic}[1]
		\Require{recombining scenario tree $\tcN$ with subproblem oracles $\scrO_r,\scrO_n^\Fwd,\scrO_n^\Bwd,n\neq r$}
		\State{Initialize: $i\leftarrow1$; $\ulcQ_t^0\leftarrow0,\;\forall\,t$}
		\State{Evaluate $(x_0^1,y_0^1)=\scrO_r(0)$}
		\While{some stopping criterion is not satisfied}
		\State{Sample $M$ scenario paths $\scrP^i=\{P^{i,1},\dots,P^{i,M}\}$}
		\For{$j=1,\dots,M$}
		\Comment{$i$-th forward step}
    		\For{$t=1,\dots,T-1$}
				\State{Evaluate $(x_t^{i,j},y_t^{i,j},z_t^{i,j})=\scrO_n^\Fwd(x_{t-1}^{i,j},\ulcQ_t^{i-1})$}
    		\EndFor
    	\EndFor
		\For{$t=T,\dots,1$}
		\Comment{$i$-th backward step}
			\State{Update $\ulcQ_t^i$ using \eqref{eq:StochasticSamplingUnderApproximation}}
		    \For{$j=1,\dots,M$}
    			\For{$n\in\calN(t)/\sim$}
					\State{Evaluate $(\hat{x}_n^{i,j},\hat{y}_n^{i,j},\hat{z}_n^{i,j};\hat\lambda_n^{i,j},\hat\rho_n^{i,j})=\scrO_n^\Bwd(x_{t-1}^{i,j},\ulcQ_t^i)$}
					\State{Calculate $C_n^{i,j}$ and $\ubar{v}_n^{i,j}$ using \eqref{eq:UnderApproximationCut} and \eqref{eq:UnderApproximationUpdateFormula}}
    			\EndFor
    		\EndFor
		\EndFor
		\State{Update $\ulcQ_0^i$ using \eqref{eq:StochasticSamplingUnderApproximation}}
		\Comment{root node update}
		\State{Evaluate $\scrO_r(\ulcQ_0^i)=(x_0^{i+1},y_0^{i+1})$}
		\State{$i\leftarrow i+1$}
		\EndWhile
	\end{algorithmic}
\end{algorithm}

Unlike the preceding two algorithms, Algorithm~\ref{alg:StochasticDualDP} does not need to construct the over-approximation of the regularized value functions for selecting the child node to proceed with.
Instead, it determines the scenario paths before the forward step starts.
In the forward step, each nodal problem in the sampled scenario path is solved.
Then in the backward step, the dual problems are solved at the nodes that are defined by distinct data, dependent on the parent node's state variable obtained in the forward step.
The termination criterion is flexible. 
In the existing literature~\cite{shapiro_analysis_2011,zou_stochastic_2018}, statistical upper bounds based on the sampled scenario paths are often used together with the lower bound for terminating the algorithm.
In particular for the convex problems, if we set $\sigma_n=+\infty$, which implies $l_{n,\lambda}=+\infty$ in the backward step subproblem oracles \(\scrO_n^\Bwd\) for all \(n\in\calN\), then Algorithm~\ref{alg:StochasticDualDP} reduces to the usual SDDP algorithm in the literature~\cite{shapiro_analysis_2011,girardeau_convergence_2015}.

\section{Upper Bounds on Iteration Complexity of Proposed Algorithms.}
\label{sec:ComplexityUpperBounds}
In this section, we derive upper bounds on the iteration complexity of the three proposed algorithms, i.e. the bound on the iteration index when the algorithm terminates.
These upper bounds on the iteration complexity imply convergence of the algorithm to an $\epsilon$-optimal root node solution for any $\epsilon>0$.

\subsection{Upper Bound Analysis on Iteration Complexity of Algorithm~\ref{alg:DualDP}.}
\label{sec:ComplexityUB:DualDP}
In this section, we discuss the iteration complexity of Algorithm~\ref{alg:DualDP}.
We begin with the definition of a set of parameters used in the convergence analysis. 
Let $\epsilon$ denote the desired root-node optimality gap $\epsilon$ in Algorithm~\ref{alg:DualDP}.
Let $\delta=(\delta_n)_{n\in\calN,\calC(n)\neq\varnothing}$ be a set of positive numbers such that $\epsilon=\sum_{n\in\calN,\calC(n)\neq\varnothing}p_n\delta_n$. Since $\epsilon>0$, such $\delta_n$'s clearly exist. Then, we define recursively for each non-leaf node $n$
\begin{equation}\label{eq:ApproximationGapVector}
	\gamma_n(\delta)\coloneqq
	\delta_n+\sum_{m\in\calC(n)}p_{nm}\gamma_m(\delta),
\end{equation}
and $\gamma_n(\delta)=0$ for leaf nodes $n$.
For $i\in\N$, recall the approximation gap $\gamma_n^i=\olcQ_n^{i-1}(x_n^i)-\ulcQ_n^{i-1}(x_n^i)$ for $n\in\calN$.
For leaf nodes, $\gamma_n^i\equiv0$ by definition for all $i\in\N$.
In addition, we define the sets of indices $\calI_n(\delta)$ for each $n\in\calN$ as
\begin{equation}\label{eq:ApproximationIndexSet}
	\calI_n(\delta)\coloneqq\left\{i\in\N:\gamma_n^i>\gamma_n(\delta)\text{ and }\gamma_m^i\le\gamma_m(\delta),\forall\,m\in\calC(n)\right\}.
\end{equation}
Intuitively, the index set $\calI_n(\delta)$ consists of the iteration indices when all the child nodes of $n$ have good approximations of the expected cost-to-go function at the forward step solution, while the node $n$ itself does not.
The next lemma shows that the backward step for node $n$ in the iteration $i\in\calI_n(\delta)$ will reduce the expected cost-to-go function approximation gap at node $n$ to be no more than $\gamma_n(\delta)$.

\begin{lemma}\label{lemma:NeighborhoodApproximation}
    If an iteration index $i\in\calI_n(\delta)$, i.e., $\olcQ_n^{i-1}(x_n^i)-\ulcQ_n^{i-1}(x_n^i)>\gamma_n(\delta)$ and $\olcQ_m^{i-1}(x_m^i)-\ulcQ_m^{i-1}(x_m^i)\le\gamma_m(\delta)$ for all $m\in\calC(n)$, then
    \begin{align}
    \olcQ_n^i(x)-\ulcQ_n^i(x)\le\gamma_n(\delta),\quad\forall\,x\in\calX_n,\,\Vert{x-x_n^i}\Vert\le\frac{\delta_n}{2L_n},\label{eq:lemma:NeighborhoodApproximation}
    \end{align}
    where $L_n\coloneqq\sum_{m\in\calC(n)}p_{nm}(l_{m,\lambda}+l_{m,\rho})$ is determined by the input parameters.
\end{lemma}
The proof is given in Section~\ref{app:proof:lemma:NeighborhoodApproximation}. 
% The key idea is that the over- and under-approximations $\olcQ_n^i(x)$ and $\ulcQ_n^i(x)$ are  monotone with respect to the iteration index. 
% So if they come close at some point, a certain neighborhood of that point is well-approximated in all subsequent iterations by the Lipschitz continuity.
Lemma~\ref{lemma:NeighborhoodApproximation} shows that an iteration being in the index set would imply an improvement of the approximation in a neighborhood of the current state. 
In other words, each $i\in \calI_n$ would carve out a ball of radius $\delta_n/(2L_n)$ in the state space $\calX_n$ such that no point in the ball can be the forward step solution of some iteration $i$ in $\calI_n$. 
This implies that we could bound the cardinality $|\calI_n|$ of $\calI_n$ by the size and shape of the corresponding state space $\calX_n$. 
%Since $\calX_n$ can be nonconvex, we consider finite covers of $\calX_n$ by norm balls and provide the bound in terms of the number and sizes of the balls. This is made more precise in the following lemma.
\begin{lemma}
	\label{lemma:IndexSetCardinality}
	Let $\scrB=\{\calB_{n,k}\subset\R^{d_n}\}_{1\le k\le K_n,n\in\calN}$ be a collection of balls, each with diameter $D_{n,k}\ge0$, such that $\calX_n\subseteq\bigcup_{k=1}^{K_n}\calB_{n,k}$.
	Then,
	\begin{equation*}
	\abs{\calI_n(\delta)}\le\sum_{k=1}^{K_n}\left(1+\frac{2L_nD_{n,k}}{\delta_n}\right)^{d_n}.
	\end{equation*} 
\end{lemma}
The proof is by a volume argument of the covering balls with details given in Section~\ref{app:proof:lemma:IndexSetCardinality}.
We have an upper bound on the iteration complexity of Algorithm~\ref{alg:DualDP}.
\begin{theorem}
	\label{thm:DDPComplexityUpperBound}
	Given $\epsilon>0$, choose values $\delta=(\delta_n)_{n\in\calN,\calC(n)\neq\varnothing}$ such that $\delta_n>0$ and $\sum_{n\in\calN,\calC(n)\neq\varnothing}p_n\delta_n=\epsilon$.
	Let $\scrB=\{\calB_{n,k}\}_{1\le k\le K_n,n\in\calN}$ be a collection of balls, each with diameter $D_{n,k}\ge0$, such that $\calX_n\subseteq\bigcup_{k=1}^{K_n}\calB_{n,k}$ for $n\in\calN$.
	If Algorithm~\ref{alg:DualDP} terminates with an $\epsilon$-optimal root node solution $(x_r^*,y_r^*)$ at the end of $i$-th iteration, then 
	$$i\le\sum_{\substack{n\in\calN,\\\calC(n)\neq\varnothing}}\sum_{k=1}^{K_n}\left(1+\frac{2L_nD_{n,k}}{\delta_n}\right)^{d_n}.$$ 
\end{theorem}
\begin{proof}
	After the $i$-th iteration, at least one of the following two situations must happen:
	\begin{enumerate}
		\item[i.] At the root node, it holds that $\olcQ_r^i(x_r^{i+1})-\ulcQ_r^i(x_r^{i+1})\le\gamma_r(\delta)$, where $\gamma_r$ is defined in \eqref{eq:ApproximationGapVector}.
		\item[ii.] There exists a node $n\in\calN$ such that $\olcQ_n^i(x_n^{i+1})-\ulcQ_n^i(x_n^{i+1})>\gamma_n(\delta)$, but all of its child nodes satisfy $\olcQ_m^i(x_m^{i+1})-\ulcQ_m^i(x_m^{i+1})\le\gamma_m(\delta)$, $\forall\,m\in\calC(n)$. 
		In other words, $i+1\in \calI_n(\delta)$.
	\end{enumerate}
	Note that $\gamma_r(\delta)=\delta_r+\sum_{m\in\calC(r)}p_{rm}\gamma_m(\delta)=\cdots=\sum_{n\in\calN,\calC(n)\neq\varnothing}p_n\delta_n$.
	If case i happens, then by Proposition~\ref{prop:RootNodeOptimality}, $(x_r^{i+1},y_r^{i+1})$ is an $\epsilon$-optimal root node solution.
	Note that case ii can only happen at most $\sum_{n\in\calN}\abs{\calI_n(\delta)}$ times by Lemma~\ref{lemma:IndexSetCardinality}.
	Therefore, we have that 
	$$i\le\sum_{\substack{n\in\calN\\\calC(n)\neq\varnothing}}\sum_{k=1}^{K_n}\left(1+\frac{2L_nD_{n,k}}{\delta_n}\right)^{d_n},$$
	when the algorithm terminates.
	\qed
\end{proof}
Theorem~\ref{thm:DDPComplexityUpperBound} implies the $\epsilon$-convergence of the algorithm for any $\epsilon>0$.
We remark that the form of the upper bound depends on the values $\delta$ and the covering balls $\calB_{n,k}$, and therefore the right-hand-side can be tightened to the infimum over all possible choices.
While it may be difficult to find the best bound in general, in the next section we take some specific choices of $\delta$ and $\scrB$ and simplify the complexity upper bound, based on the stagewise independence assumption.

%%%%%%%%%%%%%%%%%%%%%%%%%%%%%%%%%%%%%%%%%%%%%%%%%%%%%%
%%%% Upper Bound for Deterministic Sampling DDP %%%%%%
%%%%%%%%%%%%%%%%%%%%%%%%%%%%%%%%%%%%%%%%%%%%%%%%%%%%%%
\subsection{Upper Bound Analysis on Iteration Complexity of Algorithm~\ref{alg:DeterministicDualDP}.}
\label{sec:ComplexityUB:DeterministicDualDP}
Before giving the iteration complexity bound for Algorithm~\ref{alg:DeterministicDualDP}, we slightly adapt the notations in the previous section to the stagewise independent scenario tree.
We take the values $\delta=(\delta_n)_{n\in\tcN,\calC(n)\neq\varnothing}$ such that $\delta_n=\delta_{n'}$ for all $n,n'\in\tcN(t)$ for some $t=1,\dots,T$. 
Thus we denote $\delta_t=\delta_n$ for any $n\in\tcN(t)$, and $\delta_0=\delta_r$.
The vector of $\gamma_t(\delta)$ is defined recursively for non-leaf nodes as
\begin{equation}
	\gamma_t(\delta):=\gamma_{t+1}(\delta)+\delta_t,\quad\text{if }t\le T-1,
\end{equation}
and $\gamma_T(\delta)=0$.
Let $\gamma_t^i\coloneqq\olcQ_t^{i-1}(x_t^i)-\ulcQ_t^{i-1}(x_t^i)$ and recall that $\gamma_0^i\coloneqq\gamma_r^i$ for each index $i$.
The sets of indices $\calI_t(\delta)$ are defined for $t=0,\dots,T-1$ as
\begin{equation}\label{eq:StagewiseIndependentIndexSet}
	\calI_t(\delta)\coloneqq\left\{i\in\bbN:\gamma_t^i>\gamma_t(\delta)\text{ and }\gamma_{t+1}^i\le\gamma_{t+1}(\delta)\right\}.
\end{equation}
Note that $\gamma_t^i=\max_{n\in\tcN(t)}\gamma_n^i$ (line~\ref{alg:DeterministicDualDP:MaxGap} in Algorithm~\ref{alg:DeterministicDualDP}).
By Lemma~\ref{lemma:NeighborhoodApproximation}, an iteration $i\in\calI_t(\delta)$ implies $\olcQ_t^i(x)-\ulcQ_t^i(x)\le\gamma_t(\delta)$ for all $x\in\calX_n$ with $\|x-x_t^i\|\le\delta_t/(2L_t)$, where $L_t=L_n$ for any $n\in\tcN(t)$.
Moreover, since $\calX_n=\calX_t$ for $n\in\tcN(t)$, 
for any covering balls $\calB_{t,k}\subset\R^{d_t}$ with diameters $D_{t,k}\ge0$, such that $\calX_t\subseteq\cup_{k=1}^{K_t}\calB_{t,k}$, by the same argument of Lemma~\ref{lemma:IndexSetCardinality}, we know that 
\begin{equation}\label{eq:IndexSetCardinality}
	\abs{\calI_t(\delta)}\le\sum_{k=1}^{K_t}\left(1+\frac{2L_tD_{t,k}}{\delta_t}\right)^{d_t}.
\end{equation}
We summarize the upper bound on the iteration complexity of Algorithm \ref{alg:DeterministicDualDP} in the next theorem, and omit the proof since it is almost a word-for-word repetition with the notation adapted as above.
\begin{theorem}
	\label{thm:DeterministicDDPComplexityUpperBound}
	Given any $\epsilon>0$, choose values $\delta=(\delta_t)_{t=0}^{T-1}$ such that $\delta_t>0$ and  $\sum_{t=0}^{T-1}\delta_t=\epsilon$. Let  $\scrB=\{\calB_{t,k}\subset\R^{d_t}\}_{1\le k\le K_t,0\le t\le T-1}$ be a collection of balls, each with diameter $D_{t,k}\ge0$, such that $\calX_t\subseteq\bigcup_{k=1}^{K_t}\calB_{t,k}$ for $0\le t\le T-1$. If
	Algorithm~\ref{alg:DeterministicDualDP} terminates with an $\epsilon$-optimal root node solution $(x_0^*,y_0^*)$ in $i$ iterations, then 
	$$i\le\sum_{t=0}^{T-1}\sum_{k=1}^{K_t}\left(1+\frac{2L_t D_{t,k}}{\delta_t}\right)^{d_t}.$$ 
\end{theorem}

We next discuss some special choices of the values $\delta$ and the covering ball collections $\scrB$.
First, since $\calX_t$ are compact, suppose $\calB_t$ is the smallest ball containing $\calX_t$.
Then we have $\diam\calX_t\le D_t\le 2\diam\calX_t$ where $D_t=\diam\calB_t$.
Moreover, suppose $L_t\le L$ for some $L>0$ and $d_t\le d$ for some $d>0$.
Then by taking $\delta_t=\epsilon/T$ for all $0\le t\le T-1$, we have the following bound.
\begin{corollary}\label{cor:AbsoluteOptimalityGapComplexityBound}
	If Algorithm~\ref{alg:DeterministicDualDP} terminates with an $\epsilon$-optimal root node solution $(x_0^*,y_0^*)$, then the iteration index is bounded by
	$$i\le T\left(1+\frac{2LDT}{\epsilon}\right)^d,$$
	where $L,d,D$ are the upper bounds for $L_t,d_t,$ and $D_t$, $0\le t\le T-1$, respectively.
\end{corollary}
\begin{proof}
	Take $\delta_t=\epsilon/T$ for all $0\le t\le T-1$ and apply Theorem~\ref{thm:DeterministicDDPComplexityUpperBound}.
	\qed
\end{proof}
Note that the iteration complexity bound in Corollary~\ref{cor:AbsoluteOptimalityGapComplexityBound} grows asymptotically $\calO(T^{d+1})$ as $T\to\infty$.
Naturally such bound is not satisfactory since it is nonlinear in $T$ with possibly very high degree $d$.
However, by changing the optimality criterion, we next derive an iteration complexity bound that grows linearly in $T$, while all other parameters, $L, D, \epsilon, d$, are independent of $T$.
\begin{corollary}\label{cor:RelativeOptimalityGapComplexityBound}
	If Algorithm~\ref{alg:DeterministicDualDP} terminates with a $(T\epsilon)$-optimal root node solution $(x_0^*,y_0^*)$, then the iteration index is bounded by
	$$i\le T\left(1+\frac{2LD}{\epsilon}\right)^d,$$
	where $L,d,D$ are the upper bounds for $L_t,d_t,$ and $D_t$, $0\le t\le T-1$, respectively.
\end{corollary}
\begin{proof}
	Take $\delta_t=\epsilon$ for all $0\le t\le T-1$ and apply Theorem~\ref{thm:DeterministicDDPComplexityUpperBound}.
	\qed
\end{proof}
The termination criterion in Corollary~\ref{cor:RelativeOptimalityGapComplexityBound} corresponds to the usual relative optimality gap, if the total objective is known to grow at least linearly with \(T\), as is the case for many practical problems.
Last, we consider a special case where $\calX_t$ are finite for all $0\le t\le T-1$.
\begin{corollary}\label{cor:FiniteStateSpaceComplexityBound}
	Suppose the cardinality $\abs{\calX_t}\le K<\infty$ for all $0\le t\le T-1$, for some positive integer $K$.
	In this case, if Algorithm~\ref{alg:DeterministicDualDP} terminates with an $\epsilon$-optimal root node solution $(x_0^*,y_0^*)$, then the iteration index is bounded by
	$$i\le TK.$$
\end{corollary}
\begin{proof}
	Note that when $\calX_t$ is finite, it can be covered by degenerate balls $B_0(x)$, $x\in\calX_t$.
	Thus $D_{t,k}=0$ for $k=1,\dots,K_t$ and $K_t\le K$ by assumption.
	Apply Theorem~\ref{thm:DeterministicDDPComplexityUpperBound}, we get
	$i\le\sum_{t=0}^{T-1}\sum_{k=1}^{K_t}1\le TK.$
	\qed
\end{proof}
The bound in Corollary~\ref{cor:FiniteStateSpaceComplexityBound} grows linearly in $T$ and does not depend on the value of $\epsilon$.
In other words, we are able to obtain exact solutions to the regularized problem~\eqref{eq:RegularizedNodalProblem} assuming the subproblem oracles.

\begin{remark}
	All the iteration complexity bounds in Theorem~\ref{thm:DeterministicDDPComplexityUpperBound}, Corollary~\ref{cor:AbsoluteOptimalityGapComplexityBound}, Corollary~\ref{cor:RelativeOptimalityGapComplexityBound}, and Corollary~\ref{cor:FiniteStateSpaceComplexityBound} are independent of the size of the scenario tree in each stage $N_t$, $1\le t\le T$.
	This can be explained by the fact that Algorithm~\ref{alg:DeterministicDualDP} evaluates $1+N_T+2\sum_{t=1}^{T-1}N_t$ times of the subproblem oracles in each iteration.
	% We later exploit this fact for the discussion on the iteration complexity of Algorithm~\ref{alg:DeterministicDualDP}. 
\end{remark}

%%%%%%%%%%%%%%%%%%%%%%%%%%%%%%%%%%%%%%%%%%%%%%%%%%%%%%
%%%% Upper Bound for SDDP %%%%%%
%%%%%%%%%%%%%%%%%%%%%%%%%%%%%%%%%%%%%%%%%%%%%%%%%%%%%%

\subsection{Upper Bound Analysis on Iteration Complexity of Algorithm \ref{alg:StochasticDualDP}}
\label{sec:ComplexityUB:StochasticDualDP}
In the following we study the iteration complexity of Algorithm~\ref{alg:StochasticDualDP}.
For clarity, we model the subproblem oracles $\scrO^\Fwd_n$ and $\scrO^\Bwd_n$ as random functions, that are $\Sigma_i^\oracle$-measurable in each iteration $i\in\bbN$, for any node $n\neq r$, 
where $\{\Sigma_i^\oracle\}_{i=0}^\infty$ is a filtration of \(\sigma\)-algebras in the probability space.
Intuitively, this model says that the information given by $\Sigma_i^\oracle$ could be used to predict the outcome of the subproblem oracles. 
We now make the following assumption on the sampling step.
\begin{assumption}\label{assum:StochasticSampling}
    In each iteration $i$, the $M$ scenario paths are sampled uniformly with replacement, independent from each other and the outcomes of the subproblem oracles.
    That is, the conditional probability of the $j$-th sample $P^{i,j}$ taking any scenario $n_t\in\tcN(t)$ in stage $t$ is almost surely
\begin{equation}
    \Prob(P^{i,j}_t=n_t\mid \Sigma_{\infty}^\oracle,\sigma\{P^{i',j'}_{t'}\}_{(i',j',t')\neq (i,j,t)})=\frac{1}{N_t},
\end{equation}
where $\Sigma_\infty^\oracle:=\cup_{i=1}^\infty\Sigma_i^\oracle$, and $\sigma\{P^{i',j'}_{t'}\}_{(i',j',t')\neq (i,j,t)}$ is the $\sigma$-algebra generated by scenario samples other than the $j$-th sample in stage $t$ of iteration $i$.
\end{assumption}

In the sampling step in the $i$-th iteration, let $\gamma_t^{i,j}\coloneqq\calQ_t^\Reg(x_t^{i,j})-\ulcQ_t^{i-1}(x_t^{i,j})$ for any $t\le T-1$, which is well defined by Assumption~\ref{assum:StagewiseIndependence}, and let $\tilde\gamma_t^{i,j}\coloneqq\max\{\calQ_t^\Reg(x_n)-\ulcQ_t^{i-1}(x_n):(x_n,y_n,z_n)=\scrO_n^\Fwd(x_{t-1}^{i,j},\ulcQ_n^{i-1}),\,n\in\tcN(t)\}$ for each scenario path index $1\le j\le M$.
Note that by definition, we have $\gamma_{t}^{i,j}\le\tilde\gamma_t^{i,j}$ for any $t=1,\dots,T-1$, everywhere in the probability space.
We define the sets of indices $\calI_t(\delta)$ for each $t=0,\dots,T-1$, similar to those in the deterministic sampling case, as 
\begin{equation}\label{eq:StochasticSamplingIndexSet}
	\calI_t(\delta)\coloneqq\bigcup_{j=1}^M\left\{i\in\bbN:\gamma_t^{i,j}>\gamma_t(\delta)\text{ and }\tilde\gamma_{t+1}^{i,j}\le\gamma_{t+1}(\delta)\right\}.
\end{equation}
With the same argument, we know that the upper bound~\eqref{eq:IndexSetCardinality} on the sizes of $\calI_t(\delta)$ holds everywhere for each $t=0,\dots,T-1$.
However, since the nodes in the forward steps are sampled randomly, we do not necessarily have $i\in\cup_{t=0}^{T-1}\calI_t(\delta)$ for each iteration index $i\in\N$ before Algorithm~\ref{alg:StochasticDualDP} first finds an $\epsilon$-optimal root node solution.
Instead, we define an event $A_i(\delta):=\{i\in\cup_{t=0}^{T-1}\calI_t(\delta)\}\bigcup\cup_{j=1}^M\{\gamma_0^{i-1,j}\le\gamma_0(\delta)=\epsilon\}$ for each iteration $i$, that means either some approximation is improved in iteration $i$ or the algorithm has found an $\epsilon$-optimal root node solution in iteration $i-1$. 
The next lemma estimates the conditional probability of $A_i(\delta)$ given any oracles outcomes and samplings up to iteration $i$.
For simplicity, we define two $\sigma$-algebras $\Sigma_i^\sample:=\sigma\{P^{i',j'}\}_{i'\le i,j'=1,\dots,M}$ and $\Sigma_i:=\sigma(\Sigma_i^\oracle,\Sigma_i^\sample)$ for each $i$.
\begin{lemma}\label{lemma:StochasticDDPEventProbability}
    Fix any $\epsilon=\sum_{t=0}^{T-1}\delta_t$.
    Then the conditional probability inequality
    $$\Prob(A_i(\delta)\given \Sigma_{i-1})\ge \nu:=1-(1-1/N)^M,$$ 
    holds almost surely, 
    where $N\coloneqq \prod_{t=1}^{T-1}N_t$ if $T\ge 2$ and $N\coloneqq 1$ otherwise.
\end{lemma}
The proof is given in Section~\ref{app:proof:lemma:StochasticDDPEventProbability}.
Now we are ready to present the probabilistic complexity bound of Algorithm~\ref{alg:StochasticDualDP}, the proof of which is given in Section~\ref{app:proof:thm:StochasticDDPComplexityUpperBound}.

\begin{theorem}\label{thm:StochasticDDPComplexityUpperBound}
	Let $I=I(\delta,\scrB)$ denote the iteration complexity bound in Theorem~\ref{thm:DeterministicDDPComplexityUpperBound}, determined by the vector $\delta$ and the collection of state space covering balls $\scrB$, and $\nu$ denote the probability bound proposed in Lemma~\ref{lemma:StochasticDDPEventProbability}.
	Moreover, let $\iota$ be the random variable of the smallest index such that the root node solution $(x_0^{\iota+1},y_0^{\iota+1})$ is $\epsilon$-optimal in Algorithm~\ref{alg:StochasticDualDP}.
	Then for any real number $\kappa>1$, the probability 
	$$\Prob\left(\iota\ge 1+\frac{\kappa I}{\nu}\right)\le \exp\left(\frac{-I\nu(\kappa-1)^2}{16\kappa}\right).$$
\end{theorem}

\begin{remark}\label{remark:StochasticSamplingComplexityBound}
    Theorem~\ref{thm:StochasticDDPComplexityUpperBound} shows that for a fixed problem (such that $I=I(\delta,\scrB)$ and $N=N_1\cdots N_{T-1}$ are fixed), given any probability threshold $q\in(0,1)$, the number of iterations needed for Algorithm~\ref{alg:StochasticDualDP} to find an $\epsilon$-optimal root node solution with probability greater than $1-q$ is $\calO(-\ln{q}/\nu^2)$, which does not depend on $I$.
    In particular, if we set $M=1$, then the number of iterations needed is $\calO(-N^2\ln{q})$, which is exponential in the number of stage $T$ if $N_t\ge 2$ for all $t=1,\dots,T-1$.
    It remains unknown to us whether there exists a complexity bound for Algorithm~\ref{alg:StochasticDualDP} that is polynomial in $T$ in general.
\end{remark}

\section{Lower Bounds on Iteration Complexity of Proposed Algorithms.}\label{sec:ComplexityLowerBounds}

In this section, we discuss the sharpness of the iteration complexity bound of Algorithm~\ref{alg:DeterministicDualDP} given in Section~\ref{sec:ComplexityUpperBounds}. 
In particular, we are interested in the question whether it is possible that the iteration needed for Algorithm~\ref{alg:DeterministicDualDP} to find an $\epsilon$-optimal root node solution grows linearly in $T$ when the state spaces are infinite sets.
We will see that in general it is not possible, with or without the assumption of convexity.
The following lemma simplifies the discussion in this section.

\begin{lemma}\label{lemma:LipschitzOptimalValueFunctionApproximation}
	Suppose $f_n(z,y,x)$ is $l_n$-Lipschitz continuous in $z$ for each $n\in\calN$.
	If we choose $\psi_n(x)=\norm{x}$ and $\sigma_n\ge l_n$, then $Q_n^\Reg(x)=Q_n(x)$ on $\calX_{a(n)}$ for all non-root nodes $n\in\calN$.
\end{lemma}
The proof exploits the Lipschitz continuity of $f_n$ and the fact $Q_n^R(x)$ is an under-approximation of $Q_n(x)$ in an inductive argument. 
The details are given in Section~\ref{app:proof:lemma:LipschitzOptimalValueFunctionApproximation}.
In other words, for problems that already have Lipschitz continuous value functions, the regularization does not change the function value at any point.
Thus the examples in the rest of this section serve the discussion not only for Algorithm~\ref{alg:DeterministicDualDP}, but for more general algorithms including SDDP and SDDiP.

\subsection{General Lipschitz Continuous Problems.}
We discuss the general Lipschitz continuous case, i.e., the nodal objective functions $f_n(z,y,x)$ are $l_n$-Lipschitz continuous in $z$ but not necessarily convex.
In this case we choose to approximate the value function using $\psi_n(x)=\norm{x}$ and assume that $l_{n,\rho}\ge l_n$.
We can set $l_{n,\lambda}=0$ for all $n\in\calN$, without loss of exactness of the approximation by the proof of Proposition~\ref{prop:GeneralizedCutTightness}.
We begin with the following lemma on the complexity of such approximation.
\begin{lemma}\label{lemma:GeneralLipschitzFunction}
	Consider a norm ball $\calX=\{x\in\R^d:\norm{x}\le D/2\}$ and a finite set of points $\calW=\{w_k\}_{k=1}^K\subset\calX$.
	Suppose that there is $\beta>0$ and an $L$-Lipschitz continuous function $f:\calX\to\R_+$ such that $\beta<f(w_k)<2\beta$ for $k=1,\dots,K$. 
	Define 
	\begin{itemize}
		\item $\displaystyle\ulQ(x)\coloneqq\max_{k=1,\dots,K}\{0,f(w_k)-L\norm{x-w_k}\}$ and 
		\item $\displaystyle\olQ(x)\coloneqq\min_{k=1,\dots,K}\{f(w_k)+L\norm{x-w_k}\}$.
	\end{itemize}
	 If $K<\left(\frac{DL}{4\beta}\right)^d$, then $\displaystyle\min_{x\in\calX}\ulQ(x)=0$ and $\displaystyle\min_{x\in\calX}\olQ(x)>\beta$.
\end{lemma}
The proof is given in Section~\ref{app:proof:lemma:GeneralLipschitzFunction}. 
The lemma shows that if the number of points in $\calW$ is too small, i.e. $K<(DL/2\beta)^d$, then the difference between the upper and lower bounds could be big, i.e.  $\olQ(\bar{x})-\ulQ(\bar{x})>\beta$ for some $\bar{x}$. In other words, in order to have a small gap between the upper and lower bounds, we need sufficient number of sample points. This lemma is directly used to provide a lower bound on the complexity of Algorithm~\ref{alg:DeterministicDualDP}.

Now we construct a Lipschitz continuous multistage problem defined on a chain, i.e., a scenario tree, where each stage has a single node, $N(t)=1$ for $t=1,\dots,T$. The problem is given by the value functions in each stage as,
\begin{align}
\begin{dcases}
 Q_r = \min_{x_0\in\calX_r}Q_1(x_0),\\
 Q_{t}(x_{t-1}) =\min_{x_t\in\calX_t}\left\{f_{t}(x_{t-1})+Q_{t+1}(x_t)\right\}, & 1\le t\le T-1,\\
Q_{T}(x_{T-1}) =f_{T}(x_{T-1}).
\end{dcases}\label{eq:LipschitzBoxExample}
\end{align}
Here for all $t=1,\dots,T$, $f_t : \calX_t\rightarrow\R_+$ is an $L$-Lipschitz continuous function that satisfies 
$\beta<f_t(x)<2\beta$ for all $x\in\calX_t$ with $\beta:=\epsilon/T$, the number of stages $T\ge 1$, and $\epsilon>0$ is a fixed constant. The state space $\calX_t:=\calB^d(D/2)\subset\R^d$ is a ball with radius $D/2>0$. We remark that $\epsilon$ will be the optimality gap in Theorem \ref{thm:GeneralLipschitzComplexity}. So for a fixed optimality gap $\epsilon$, we construct an instance of multistage problem \eqref{eq:LipschitzBoxExample} that will prove to be difficult for Algorithm \ref{alg:DeterministicDualDP} to solve.
Also \eqref{eq:LipschitzBoxExample} is constructed such that there is no constraint coupling the state variables $x_t$ in different stages. 

By Lemma~\ref{lemma:LipschitzOptimalValueFunctionApproximation}, if we choose $\psi_n(x)=\norm{x}$ for all $n\in\calN$ 
and $l_{n,\rho}= L$ for the problem~\eqref{eq:LipschitzBoxExample}, then we have $Q^\Reg_t(x)=Q_t(x)$ for all $t=1,\dots,T$.
The next theorem shows a lower bound on the iteration complexity of problem~\eqref{eq:LipschitzBoxExample} with this choice of penalty functions.
\begin{theorem}\label{thm:GeneralLipschitzComplexity}
	For any optimality gap $\epsilon>0$, there exists a problem of the form~\eqref{eq:LipschitzBoxExample} with subproblem oracles $\scrO_n^\Fwd,\scrO_n^\Bwd$, $n\in\calN$, and $\scrO_r$, such that if Algorithm~\ref{alg:DeterministicDualDP} gives $\UB-\LB\le\epsilon$ in the $i$-th iteration, then
	$$i\ge\left(\frac{DLT}{4\epsilon}\right)^d.$$
\end{theorem}
The proof is given in Section~\ref{app:proof:thm:GeneralLipschitzComplexity}.
The theorem shows that in general Algorithm~\ref{alg:DeterministicDualDP} needs at least $\calO(T^d)$ iterations before termination.
We comment that this is due to the fact that the approximation using generalized conjugacy is tight only locally.
Without convexity, one may need to visit many states to cover the state space to achieve tight approximations of the value functions before the algorithm is guaranteed to find an $\epsilon$-optimal solution.

\subsection{Convex Lipschitz Continuous Problems.}
In the above example for general Lipschitz continuous problem, we see that the complexity of Algorithm~\ref{alg:DeterministicDualDP} grows at a rate of $\calO(T^d)$.
It remains to answer whether convexity could help us avoid this possibly undesirable growth rate in terms of $d$. 
We show that even by using linear cuts, rather than generalized conjugacy cuts, for convex value functions, the complexity lower bound of the proposed algorithms could not be substantially improved.
We begin our discussion with a definition.
\begin{definition}
	Given a $d$-sphere $\calS^d(R)=\{x\in\R^{d+1}:\norm{x}_2=R\}$ with radius $R>0$, a spherical cap with depth $\beta>0$ centered at a point $x\in \calS^d(R)$ is the set
	$$\calS^d_\beta(R,x)\coloneqq\{y\in \calS^d(R):\innerprod{y-x}{x}\ge -\beta R\}.$$
\end{definition}
The next lemma shows that we can put many spherical caps on a sphere, the center of each is not contained in any other spherical cap, the proof of which is given in Section~\ref{app:proof:lemma:SphericalCap}.
\begin{lemma}\label{lemma:SphericalCap}
	Given a $d$-sphere $\calS^d(R),d\ge 2$ and depth $\beta<(1-\frac{\sqrt{2}}{2})R$, there exists a finite set of points $\calW$ with 
	$$\abs{\calW}\ge\frac{(d^2-1)\sqrt{\pi}}{d}\frac{\Gamma(d/2+1)}{\Gamma(d/2+3/2)}\left(\frac{R}{2\beta}\right)^{(d-1)/2},$$ 
	such that, for any $w\in \calW$, $\calS^d_\beta(R,w)\cap \calW=\{w\}$.
\end{lemma}

Hereafter, we denote a set of points that satisfies Lemma~\ref{lemma:SphericalCap} as $\calW_\beta^d(R)\subset \calS^d(R)$.
Next we construct an $L$-Lipschitz convex function for any $L>0$, $\epsilon>0$ that satisfies certain properties on $\calW_{\epsilon/L}^d(R)$.
The proof is given in Section~\ref{app:proof:lemma:SphericalCapConvexFunction}.
\begin{lemma}\label{lemma:SphericalCapConvexFunction}
	Given positive constants $\epsilon>0, L>0$ and a set $\calW_{\epsilon/L}^d(R)$. Let $K\coloneqq\vert{\calW^d_{\epsilon/L}(R)}\vert$.
	For any values $v_k\in(\epsilon/2,\epsilon)$, $k=1,\dots,K$, 
	define a function $F:\calB^{d+1}(R)\to\R$ as $F(x)=\max_{k=1,\dots,K}\{0,v_k+\frac{L}{R}\innerprod{w_k}{x-w_k}\}$. Then $F$ satisfies the following properties:
	\begin{enumerate}
		\item $F$ is an $L$-Lipschitz convex function;
		\item $F(w_k)=v_k$ for all $w_k\in \calW^d_{\epsilon/L}(R)$;
		\item $F$ is differentiable at all $w_k$, with $v_k+\innerprod{\nabla F(w_k)}{w_l-w_k}< 0$ for all $l\neq k$; 
		\item For any $w_l\in\calW^d_{\epsilon/L}(R)$, $\ulQ_l(x)\coloneqq\max_{k\neq l}\{0,v_k+\innerprod{\nabla F(w_k)}{x-w_k}\}$ and $\olQ_l(x)\coloneqq\conv_{k\neq l}\{v_k+L\norm{x-w_k}\}$ satisfy
		$$\olQ_l(w_l)-\ulQ_l(w_l)>\frac{3\epsilon}{2}.$$
	\end{enumerate}
\end{lemma}

Now we present the multistage convex dual dynamic programming example based on the following parameters:
$T\ge 2$ (number of stages), $L>0$ (Lipschitz constant), $d\ge3$ (state space dimension), $D=2R>0$ (state space diameter), and $\epsilon>0$ (optimality gap).
Choose any $L_1,\dots,L_T$ such that $L/2\le L_T< L_{T-1}<\cdots<L_1\le L$, and then construct finite sets $\calW_t\coloneqq\calW^{d-1}_{\epsilon/((T-1)L_{t+1})}(R)=\{w_{t,k}\}_{k=1}^{K_t}$, $K_t=\abs{\calW_t}$ as defined in Lemma~\ref{lemma:SphericalCap} for $t=1,\dots,T-1$.
Moreover, define convex $L_{t+1}$-Lipschitz continuous functions $F_t$ for some values $v_{t,k}\in(\epsilon/(2T-2),\epsilon/(T-1))$, $k=1,\dots,K_t$, and the finite sets $\calW_t$.
By Assumption~\ref{assum:StagewiseIndependence}, we define the stagewise independent scenario tree as follows.
There are $K_t$ distinct nodes in each stage $t=1,\dots,T-1$, which can be denoted by an index pair $n=(t,k)$ for $k=1,\dots,K_t$, and all nodes are defined by the same data in the last stage $T$.
Then we define our problem by specifying the nodal cost functions $f_{r}\equiv0$, $f_{1,k}(x_0,y_1,x_1):=L_1\nVert{x_1-w_{1,k}}$ for $k=1,\dots,K_1$, $f_{t,k}(x_{t-1},y_t,x_t):=F_{t-1}(x_{t-1})+L_t\nVert{x_t-w_{t,k}}$ for $k=1,\dots,K_t$ and $t=2,\dots,T-1$, and $f_{T,1}(x_{T-1},y_T,x_T):=F_{T-1}(x_{T-1})$, and state spaces $\calX_t=\calX=\calB^{d+1}(R)$.
Alternatively, the value functions can be written as
\begin{empheq}[left={\empheqlbrace}]{align}
& Q_{1,k} = \min_{x_1\in\calX}\left\{L_1\norm{x_1-w_{1,k}}+\calQ_1(x_1)\right\},\ \forall k\le K_1,\label{eq:SphericalCapExample}\\
& Q_{t,k}(x_{t-1}) = \min_{x_t\in\calX}\left\{F_{t-1}(x_{t-1})+L_t\norm{x_t-w_{t,k}}+\calQ_t(x_t)\right\},k\le K_t,\notag \\
& Q_{T,1}(x_{T-1}) = F_{T-1}(x_{T-1}),\notag
\end{empheq}
where the second equation is defined for all $2\le t\le T-1$, and the expected cost-to-go functions as
$$\calQ_t(x_t)\coloneqq\frac{1}{K_t}\sum_{k=1}^{K_t} Q_{t+1,k}(x_t),\quad t=0,\dots,T-1.$$
By Lemma~\ref{lemma:SphericalCap}, 
\begin{align*}
	K_t&\ge\frac{((d-1)^2-1)\sqrt{\pi}}{d-1}\frac{\Gamma((d-1)/2+1)}{\Gamma((d-1)/2+3/2)}\left(\frac{RL_t(T-1)}{2\epsilon}\right)^{(d-2)/2},\\
	&\ge\frac{d(d-2)\sqrt{\pi}}{d-1}\frac{\Gamma((d/2+1/2)}{\Gamma(d/2+1)}\left(\frac{DL(T-1)}{8\epsilon}\right)^{(d-2)/2}.
\end{align*}

Since for each value function $Q_{t,k}$ is $L_t$-Lipschitz continuous, we choose $\sigma_n=L_t$ with $\psi_n(x)=\norm{x}$ for any $n=(t,k)\in\tcN(t)$ and $t=1,\dots,T$ such that by Lemma~\ref{lemma:LipschitzOptimalValueFunctionApproximation} we have $Q_{t,k}(x)=Q_{t,k}^\Reg(x)$ for all $x\in\calX$.
Moreover, due to convexity, we set $l_{n,\rho}=0$ for all $n\in\calN$ and $l_{n,\lambda}=L_t$ for each $n\in\tcN(t)$ and $t=1,\dots,T$, i.e., the cuts are linear.
Following the argument of Proposition~\ref{prop:ConvexCutTightness}, we know that such linear cuts are capable of tight approximations. 
With such a choice of regularization we have the following theorem on the complexity of Algorithm~\ref{alg:DeterministicDualDP}.

\begin{theorem}\label{thm:ConvexLipschitzComplexity}
	For any optimality gap $\epsilon>0$, there exists a multistage stochastic convex  problem of the form~\eqref{eq:SphericalCapExample} such that, if Algorithm~\ref{alg:DeterministicDualDP} gives $\UB-\LB<\epsilon$ at $i$-th iteration, then
	\begin{align*}
	    i &> \frac{1}{3}\frac{d(d-2)\sqrt{\pi}}{d-1}\frac{\Gamma(d/2+1/2)}{\Gamma(d/2+1)}\left(\frac{DL(T-1)}{8\epsilon}\right)^{(d-2)/2}.
	\end{align*}
\end{theorem}
The proof is given in Section~\ref{app:proof:thm:ConvexLipschitzComplexity}.
The theorem implies that, even if problem~\eqref{eq:NodalProblem} is convex and has Lipschitz continuous value functions, the minimum iteration for Algorithm~\ref{alg:DeterministicDualDP} to get a guaranteed $\epsilon$-optimal root node solution grows as a polynomial of the ratio $T/\epsilon$, with the degree being $d/2-1$.

We remark that Theorems~\ref{thm:GeneralLipschitzComplexity} and \ref{thm:ConvexLipschitzComplexity} correspond to two different challenges of the SDDP type algorithms.
The first challenge is that the backward step subproblem oracle may not give cuts that provide the desired approximation, which could happen when the value functions are nonconvex or nonsmooth.
Theorem~\ref{thm:GeneralLipschitzComplexity} results from the worst case that the backward step subproblem oracle leads to approximations of the value function in the smallest neighborhood.

The second challenge is that different nodes, or more generally, different scenario paths give different states in each stage, so sampling and solving the nodal problem on one scenario path provides little information to the nodal problem on another scenario path.
In example~\eqref{eq:SphericalCapExample}, the linear cut obtained in each iteration does not provide any information on the subsequent iteration states (unless the same node is sampled again).
From this perspective, we believe that unless some special structure of the problem is exploited, any algorithm that relies on local approximation of value functions will face the ``curse of dimensionality,''
i.e., the exponential growth rate of the iteration complexity in the state space dimensions.

\section{Conclusions.}
\label{sec:Concluding}
In this paper, we propose three algorithms in a unified framework of dual dynamic programming for solving multistage stochastic mixed-integer nonlinear programs. The first algorithm is a generalization of the classic nested Benders decomposition algorithm, which deals with general scenario trees without the stagewise independence property. The second and third algorithms generalize SDDP with sampling procedures on a stagewise independent scenario tree, where the second algorithm uses a deterministic sampling approach, and the third one uses a randomized sampling approach. 
The proposed algorithms are built on regularization of value functions, which enables them to handle problems with value functions that are non-Lipschitzian or discontinuous. We show that the regularized problem preserves the feasibility and optimality of the original multistage program, when the corresponding penalty reformulation satisfies exact penalization. 
The key ingredient of the proposed algorithms is a new class of cuts based on generalized conjugacy for approximating nonconvex cost-to-go functions of the regularized problems.

We obtain upper and lower bounds on the iteration complexity of the proposed algorithms on MS-MINLP problem classes that allow exact Lipschitz regularization with predetermined penalty functions and parameters.
The complexity analysis is new and deepens our understanding of the behavior of SDDP. For example, it is the first time to prove that the iteration complexity of SDDP depends polynomially on the number of stages, not exponentially, for both convex and nonconvex multistage stochastic programs, and this complexity dependence can be reduced to linear if the optimality gap is allowed to scale linearly with the number of stages, or if all the state spaces are finite sets. These findings resolve a conjecture of the late Prof. Shabbir Ahmed, who inspired us to work on this problem.

% BibTeX users please use one of
%\bibliographystyle{spbasic}      % basic style, author-year citations
\bibliographystyle{spmpsci}      % mathematics and physical sciences
%\bibliographystyle{spphys}       % APS-like style for physics
%\bibliography{}   % name your BibTeX data base
\bibliography{ref_dp}

\begin{thebibliography}{10}
\providecommand{\url}[1]{{#1}}
\providecommand{\urlprefix}{URL }
\expandafter\ifx\csname urlstyle\endcsname\relax
  \providecommand{\doi}[1]{DOI~\discretionary{}{}{}#1}\else
  \providecommand{\doi}{DOI~\discretionary{}{}{}\begingroup
  \urlstyle{rm}\Url}\fi

\bibitem{ahmed2019stochastic}
Ahmed, S., Cabral, F.G., da~Costa, B.F.P.: Stochastic lipschitz dynamic
  programming (2019)

\bibitem{ahmed_multi-stage_2000}
Ahmed, S., King, A.J., Parija, G.: A {Multi}-stage {Stochastic} {Integer}
  {Programming} {Approach} for {Capacity} {Expansion} under {Uncertainty}.
\newblock Journal of Global Optimization p.~23 (2000)

\bibitem{baringo_risk-constrained_2013}
Baringo, L., Conejo, A.J.: Risk-{Constrained} {Multi}-{Stage} {Wind} {Power}
  {Investment}.
\newblock IEEE Transactions on Power Systems \textbf{28}(1), 401--411 (2013).
\newblock \doi{10.1109/TPWRS.2012.2205411}.
\newblock \urlprefix\url{http://ieeexplore.ieee.org/document/6247489/}

\bibitem{basciftci_adaptive_2019}
Basciftci, B., Ahmed, S., Gebraeel, N.: Adaptive {Two}-stage {Stochastic}
  {Programming} with an {Application} to {Capacity} {Expansion} {Planning}.
\newblock arXiv:1906.03513 [math]  (2019).
\newblock \urlprefix\url{http://arxiv.org/abs/1906.03513}.
\newblock ArXiv: 1906.03513

\bibitem{Baucke_Downward_Zakeri}
Baucke, R., Downward, A., Zakeri, G.: A deterministic algorithm for solving
  multistage stochastic programming problems.
\newblock Optimization Online p.~25 (2017)

\bibitem{benders_partitioning_1962}
Benders, J.F.: Partitioning procedures for solving mixed-variables programming
  problems.
\newblock Numerische Mathematik p.~15 (1962)

\bibitem{birge_decomposition_1985}
Birge, J.R.: Decomposition and {Partitioning} {Methods} for {Multistage}
  {Stochastic} {Linear} {Programs}.
\newblock Operations Research \textbf{33}(5), 989--1007 (1985).
\newblock \doi{10.1287/opre.33.5.989}.
\newblock
  \urlprefix\url{http://pubsonline.informs.org/doi/abs/10.1287/opre.33.5.989}

\bibitem{bradley_dynamic_1972}
Bradley, S.P., Crane, D.B.: A {Dynamic} {Model} for {Bond} {Portfolio}
  {Management}.
\newblock Management Science \textbf{19}(2), 139--151 (1972).
\newblock \doi{10.1287/mnsc.19.2.139}.
\newblock
  \urlprefix\url{http://pubsonline.informs.org/doi/abs/10.1287/mnsc.19.2.139}

\bibitem{chen_scenario-based_2002}
Chen, Z.L., Li, S., Tirupati, D.: A scenario-based stochastic programming
  approach for technology and capacity planning.
\newblock Computers \& Operations Research \textbf{29}(7), 781--806 (2002).
\newblock \doi{10.1016/S0305-0548(00)00076-9}.
\newblock
  \urlprefix\url{https://linkinghub.elsevier.com/retrieve/pii/S0305054800000769}

\bibitem{chen_convergent_1999}
Chen, Z.L., Powell, W.B.: Convergent {Cutting}-{Plane} and {Partial}-{Sampling}
  {Algorithm} for {Multistage} {Stochastic} {Linear} {Programs} with
  {Recourse}.
\newblock Journal of Optimization Theory and Applications \textbf{102}(3),
  497--524 (1999).
\newblock \doi{10.1023/A:1022641805263}.
\newblock \urlprefix\url{http://link.springer.com/10.1023/A:1022641805263}

\bibitem{dantzig_decomposition_1960}
Dantzig, G.B., Wolfe, P.: Decomposition {Principle} for {Linear} {Programs}.
\newblock Operations Research \textbf{8}(1), 101--111 (1960).
\newblock \doi{10.1287/opre.8.1.101}.
\newblock
  \urlprefix\url{http://pubsonline.informs.org/doi/abs/10.1287/opre.8.1.101}

\bibitem{escudero_production_nodate}
Escudero, L.F., Kamesam, P.V., King, A.J., Wets, R.J.B.: Production planning
  via scenario modelling.
\newblock Annals of Operations Research p.~27 (1993)

\bibitem{feizollahi2017exact}
Feizollahi, M.J., Ahmed, S., Sun, A.: Exact augmented lagrangian duality for
  mixed integer linear programming.
\newblock Mathematical Programming \textbf{161}(1-2), 365--387 (2017)

\bibitem{flach_long-term_2010}
Flach, B., Barroso, L., Pereira, M.: Long-term optimal allocation of hydro
  generation for a price-maker company in a competitive market: latest
  developments and a stochastic dual dynamic programming approach.
\newblock IET Generation, Transmission \& Distribution \textbf{4}(2), 299
  (2010).
\newblock \doi{10.1049/iet-gtd.2009.0107}.
\newblock
  \urlprefix\url{https://digital-library.theiet.org/content/journals/10.1049/iet-gtd.2009.0107}

\bibitem{girardeau_convergence_2015}
Girardeau, P., Leclere, V., Philpott, A.B.: On the {Convergence} of
  {Decomposition} {Methods} for {Multistage} {Stochastic} {Convex} {Programs}.
\newblock Mathematics of Operations Research \textbf{40}(1), 130--145 (2015).
\newblock \doi{10.1287/moor.2014.0664}.
\newblock
  \urlprefix\url{http://pubsonline.informs.org/doi/10.1287/moor.2014.0664}

\bibitem{glassey_nested_1973}
Glassey, C.R.: Nested {Decomposition} and {Multi}-{Stage} {Linear} {Programs}.
\newblock Management Science \textbf{20}(3), 282--292 (1973).
\newblock \doi{10.1287/mnsc.20.3.282}.
\newblock
  \urlprefix\url{http://pubsonline.informs.org/doi/abs/10.1287/mnsc.20.3.282}

\bibitem{guigues2016convergence}
Guigues, V.: Convergence analysis of sampling-based decomposition methods for
  risk-averse multistage stochastic convex programs.
\newblock SIAM Journal on Optimization \textbf{26}(4), 2468--2494 (2016)

\bibitem{guigues2017dual}
Guigues, V.: Dual dynamic programing with cut selection: Convergence proof and
  numerical experiments.
\newblock European Journal of Operational Research \textbf{258}(1), 47--57
  (2017)

\bibitem{hjelmeland_nonconvex_2019}
Hjelmeland, M.N., Zou, J., Helseth, A., Ahmed, S.: Nonconvex {Medium}-{Term}
  {Hydropower} {Scheduling} by {Stochastic} {Dual} {Dynamic} {Integer}
  {Programming}.
\newblock IEEE Transactions on Sustainable Energy \textbf{10}(1), 481--490
  (2019).
\newblock \doi{10.1109/TSTE.2018.2805164}.
\newblock \urlprefix\url{https://ieeexplore.ieee.org/document/8289405/}

\bibitem{ho_nested_1974}
Ho, J.K., Manne, A.S.: Nested decomposition for dynamic models.
\newblock Mathematical Programming \textbf{6}(1), 121--140 (1974).
\newblock \doi{10.1007/BF01580231}.
\newblock \urlprefix\url{http://link.springer.com/10.1007/BF01580231}

\bibitem{kusy_bank_1986}
Kusy, M.I., Ziemba, W.T.: A {Bank} {Asset} and {Liability} {Management}
  {Model}.
\newblock Operations Research \textbf{34}(3), 356--376 (1986).
\newblock \doi{10.1287/opre.34.3.356}.
\newblock
  \urlprefix\url{http://pubsonline.informs.org/doi/abs/10.1287/opre.34.3.356}

\bibitem{lan2020complexity}
Lan, G.: Complexity of stochastic dual dynamic programming.
\newblock Mathematical Programming pp. 1--38 (2020)

\bibitem{lara_deterministic_2018}
Lara, C.L., Mallapragada, D.S., Papageorgiou, D.J., Venkatesh, A., Grossmann,
  I.E.: Deterministic electric power infrastructure planning: {Mixed}-integer
  programming model and nested decomposition algorithm.
\newblock European Journal of Operational Research \textbf{271}(3), 1037--1054
  (2018).
\newblock \doi{10.1016/j.ejor.2018.05.039}.
\newblock
  \urlprefix\url{https://linkinghub.elsevier.com/retrieve/pii/S0377221718304466}

\bibitem{linowsky_convergence_2005}
Linowsky, K., Philpott, A.B.: On the {Convergence} of {Sampling}-{Based}
  {Decomposition} {Algorithms} for {Multistage} {Stochastic} {Programs}.
\newblock Journal of Optimization Theory and Applications \textbf{125}(2),
  349--366 (2005).
\newblock \doi{10.1007/s10957-004-1842-z}.
\newblock \urlprefix\url{http://link.springer.com/10.1007/s10957-004-1842-z}

\bibitem{louveaux_solution_1980}
Louveaux, F.V.: A {Solution} {Method} for {Multistage} {Stochastic} {Programs}
  with {Recourse} with {Application} to an {Energy} {Investment} {Problem}.
\newblock Operations Research \textbf{28}(4), 889--902 (1980).
\newblock \doi{10.1287/opre.28.4.889}.
\newblock
  \urlprefix\url{http://pubsonline.informs.org/doi/abs/10.1287/opre.28.4.889}

\bibitem{mulvey_stochastic_1992}
Mulvey, J.M., Vladimirou, H.: Stochastic {Network} {Programming} for
  {Financial} {Planning} {Problems}.
\newblock Management Science \textbf{38}(11), 1642--1664 (1992).
\newblock \doi{10.1287/mnsc.38.11.1642}.
\newblock
  \urlprefix\url{http://pubsonline.informs.org/doi/abs/10.1287/mnsc.38.11.1642}

\bibitem{nesterov2018lectures}
Nesterov, Y.: Lectures on convex optimization, vol. 137.
\newblock Springer (2018)

\bibitem{pereira_stochastic_1985}
Pereira, M.V.F., Pinto, L.M.V.G.: Stochastic {Optimization} of a
  {Multireservoir} {Hydroelectric} {System}: {A} {Decomposition} {Approach}.
\newblock Water Resources Research \textbf{21}(6), 779--792 (1985).
\newblock \doi{10.1029/WR021i006p00779}.
\newblock \urlprefix\url{http://doi.wiley.com/10.1029/WR021i006p00779}

\bibitem{pereira_multi-stage_1991}
Pereira, M.V.F., Pinto, L.M.V.G.: Multi-stage stochastic optimization applied
  to energy planning.
\newblock Mathematical Programming \textbf{52}(1-3), 359--375 (1991).
\newblock \doi{10.1007/BF01582895}.
\newblock \urlprefix\url{http://link.springer.com/10.1007/BF01582895}

\bibitem{philpott_convergence_2008}
Philpott, A., Guan, Z.: On the convergence of stochastic dual dynamic
  programming and related methods.
\newblock Operations Research Letters \textbf{36}(4), 450--455 (2008).
\newblock \doi{10.1016/j.orl.2008.01.013}.
\newblock
  \urlprefix\url{https://linkinghub.elsevier.com/retrieve/pii/S0167637708000308}

\bibitem{philpott_midas_2020}
Philpott, A., Wahid, F., Bonnans, F.: {MIDAS}: {A} {Mixed} {Integer} {Dynamic}
  {Approximation} {Scheme}.
\newblock Mathematical Programming p.~16 (2020)

\bibitem{rockafellar2009variational}
Rockafellar, R.T., Wets, R.J.B.: Variational analysis, vol. 317.
\newblock Springer Science \& Business Media (2009)

\bibitem{shapiro_analysis_2011}
Shapiro, A.: Analysis of stochastic dual dynamic programming method.
\newblock European Journal of Operational Research \textbf{209}(1), 63--72
  (2011).
\newblock \doi{10.1016/j.ejor.2010.08.007}.
\newblock
  \urlprefix\url{https://linkinghub.elsevier.com/retrieve/pii/S0377221710005448}

\bibitem{shapiro_risk_2013}
Shapiro, A., Tekaya, W., da~Costa, J.P., Soares, M.P.: Risk neutral and risk
  averse {Stochastic} {Dual} {Dynamic} {Programming} method.
\newblock European Journal of Operational Research \textbf{224}(2), 375--391
  (2013).
\newblock \doi{10.1016/j.ejor.2012.08.022}.
\newblock
  \urlprefix\url{https://linkinghub.elsevier.com/retrieve/pii/S0377221712006455}

\bibitem{slyke_l-shaped_1969}
Slyke, R.M.V., Wets, R.: L-{Shaped} {Linear} {Programs} with {Applications} to
  {Optimal} {Control} and {Stochastic} {Programming}.
\newblock SIAM Journal on Applied Mathematics \textbf{17}(4), 638--663 (1969).
\newblock \urlprefix\url{http://www.jstor.org/stable/2099310}

\bibitem{takriti_incorporating_2000}
Takriti, S., Krasenbrink, B., Wu, L.S.Y.: Incorporating {Fuel} {Constraints}
  and {Electricity} {Spot} {Prices} into the {Stochastic} {Unit} {Commitment}
  {Problem}.
\newblock Operations Research \textbf{48}(2), 268--280 (2000).
\newblock \doi{10.1287/opre.48.2.268.12379}.
\newblock
  \urlprefix\url{http://pubsonline.informs.org/doi/abs/10.1287/opre.48.2.268.12379}

\bibitem{zou_partially_2018}
Zou, J., Ahmed, S., Sun, X.A.: Partially {Adaptive} {Stochastic} {Optimization}
  for {Electric} {Power} {Generation} {Expansion} {Planning}.
\newblock INFORMS Journal on Computing \textbf{30}(2), 388--401 (2018).
\newblock \doi{10.1287/ijoc.2017.0782}.
\newblock
  \urlprefix\url{http://pubsonline.informs.org/doi/10.1287/ijoc.2017.0782}

\bibitem{zou_stochastic_2018}
Zou, J., Ahmed, S., Sun, X.A.: Stochastic dual dynamic integer programming.
\newblock Mathematical Programming  (2018).
\newblock \doi{10.1007/s10107-018-1249-5}.
\newblock \urlprefix\url{http://link.springer.com/10.1007/s10107-018-1249-5}

\bibitem{zou_multistage_2019}
Zou, J., Ahmed, S., Sun, X.A.: Multistage {Stochastic} {Unit} {Commitment}
  {Using} {Stochastic} {Dual} {Dynamic} {Integer} {Programming}.
\newblock IEEE Transactions on Power Systems \textbf{34}(3), 1814--1823 (2019).
\newblock \doi{10.1109/TPWRS.2018.2880996}.
\newblock \urlprefix\url{https://ieeexplore.ieee.org/document/8532315/}

\end{thebibliography}

\newpage

\appendix
\section{Proofs.}

In this section, we present the proofs to the theorems, propositions, and lemmas that are not displayed in the main text.
%The statements are copied before the proofs as for the convenience of readers.

\subsection{Proofs for Statements in Section~\ref{sec:ProblemFormulations}}

\subsubsection{Proof for Proposition~\ref{prop:OptimalValueFunctionProperties}}
\label{app:proof:prop:OptimalValueFunctionProperties}
% \begin{statement}
% Under Assumption \ref{assum:MinCondition}, the value function $Q_n$ is lower semicontinuous (l.s.c.) for all $n\in\calN$. Moreover, for any node $n\in\calN$,
% 	\begin{enumerate}
% 		\item if $f_n(z,y,x)$ is Lipschitz continuous in the first variable $x$ with constant $l_n$, i.e. $|f_n(z,y,x)-f_n(z',y,x)|\le l_n\|z-z'\|$ for any $z,z'\in\calX_{a(n)}$ and any $(x,y)\in\calF_n$, then $Q_n$ is also Lipschitz continuous with constant $l_n$;
% 		\item if $\calX_{a(n)}$ and $\calF_n$ are convex sets, and $f_n$ and $\calQ_n$ are convex functions, then $Q_n$ is also convex.
% 	\end{enumerate}
% \end{statement}
\begin{proof}
	We show that $Q_n$ is l.s.c.\ by showing the lower level sets $\lev_a (Q_n)=\{z\in\calX_{a(n)}:Q_n(z)\le a\}$ are closed for all $a\in\R$. 
	At any leaf node $n$, the expected cost-to-go function $\cQ_n(x_n)$ is zero, thus $z$ is in $\lev_a(Q_n)$ if and only if $z$ is in the projection of the following set $\{(z,y,x) : (x,y)\in\cF_n, \, f_n(z,y,x)\le a\}$.
	Since $f_n$ is defined on a compact set $\{(z,y,x):z\in\calX_{a(n)},(x,y)\in\calF_n\}$ and l.s.c.\ by Assumption \ref{assum:MinCondition}, we know that the set $\{(z,y,x):f_n(z,y,x)\le a\}$ is compact.
	Moreover, since the projection $(z,y,x)\mapsto z$ is continuous, the image $\lev_a(Q_n)$ is still compact, hence closed.\\
	At any non-leaf node $n$, suppose $Q_m$ is l.s.c.\ for all its child nodes $m\in\calC(n)$.
	Then, $\calQ_n$ is l.s.c.\ since $\calQ_n$ is defined in \eqref{eq:DefinitionCostToGo} and $p_{nm}>0$ for all $m$.
	A point $z\in\lev_a(Q_n)$ if and only if $z$ is in the projection of the set $\{(z,y,x) : (y,x)\in\calF_n, \, f_n(z,y,x)+\calQ_n(x)\le a\}$.
	Similarly, this shows $\lev_a Q_n$ is closed since $f_n, \calQ_n$ are l.s.c.\ and the projection $(z,y,x)\mapsto z$ is continuous.
	We thus conclude $Q_n$ is l.s.c.\ for every node $n$ in the scenario tree.
	
	To show claims 1 and 2 in the proposition, take any two points $z_1,z_2\in \calX_{a(n)}$. 
	Suppose $(x_1,y_1),\,(x_2,y_2)\in\calF_n$ are the corresponding minimizers in the definition~\eqref{eq:NodalProblem}.
	Therefore, $Q_n(z_1)=f_n(z_1,y_1,x_1)+\calQ_n(x_1)$ and $Q_n(z_2)=f_n(z_2,y_2,x_2)+\calQ_n(x_2)$.
	If $f_n$ is Lipschitz continuous in the first variable, then we have
	\begin{align*}
		Q_{n}(z_1)-Q_{n}(z_2)&=f_n(z_1,y_1,x_1)+\calQ_n(x_1)-f_n(z_2,y_2,x_2)-\calQ_n(x_2)\\
		&\le f_n(z_1,y_2,x_2)+\calQ_n(x_2)-f_n(z_2,y_2,x_2)-\calQ_n(x_2)\\
		&\le f_n(z_1,y_2,x_2)-f_n(z_2,y_2,x_2)\le l_n\norm{z_1-z_2}.
	\end{align*}
	Likewise, by exchanging $z_1$ and $z_2$, we know that $Q_n(z_2)-Q_n(z_1)\le l_n\norm{z_1-z_2}$.
	This proves that $Q_n$ is Lipschitz continuous with the constant $l_n$.\\
	To show that $Q_n$ is convex, take any $t\in[0,1]$.
	Since $\calX_{a(n)}$ is convex, $Q_n$ is defined at $t z_1+(1-t)z_2$.
	Thus,
	\begin{align*}
		&Q_n(t z_1+(1-t)z_2)\\
		&\le f_n(t z_1+(1-t)z_2,t y_1+(1-t)y_2,t x_1+(1-t)x_2)+\calQ_n(t x_1+(1-t)x_2)\\
		&\le t f_n(z_1,y_1,x_1)+(1-t)f_n(z_2,y_2,x_2)+t\calQ_n(x_1)+(1-t)\calQ_n(x_2)\\
		&=t Q_n(z_1)+(1-t)Q_n(z_2).
	\end{align*}
	The first inequality follows from the definition~\eqref{eq:NodalProblem}, while the second inequality follows from the convexity of $f_n$ and $\calQ_n$.
	This shows $Q_n$ is convex.
	\qed
\end{proof}

\subsubsection{Proof for Proposition~\ref{prop:InfConvolutionLipschitz}}
\label{app:proof:prop:InfConvolutionLipschitz}
% \begin{statement}
%     Suppose $\psi_n$ is a $1$-Lipschitz continuous penalty function on the compact set $\calX_{a(n)}-\calX_{a(n)}$ for all $n\in\calN$.
%     Then $Q_n^\Reg(x)\le Q_n(x)$ for all $x\in\calX_{a(n)}$ and $Q_n^\Reg(x)$ is $\sigma_n$-Lipschitz continuous on $\calX_{a(n)}$.
%     Moreover, if the original problem~\eqref{eq:NodalProblem} and the penalty functions $\psi_n$ are convex, then $Q_n^\Reg(x)$ is also convex.
% \end{statement}
\begin{proof}
	First we show that the partial inf-convolution 
	$$f_n\square(\sigma_n\psi_n)(x_{a(n)},y_n,x_n)\coloneqq\min_{z\in\calX_{a(n)}}f_n(z_n,y_n,x_n)+\sigma_n\psi_n(x_{a(n)}-z_n)$$ 
	is  $\sigma_n$-Lipschitz continuous in the first variable $x_{a(n)}$.
    Note that the minimum is well-defined since $\calX_{a(n)}$ is compact and the functions $f_n,\sigma_n\psi_n$ are l.s.c..
    Besides, since $z=x_{a(n)}$ is a feasible solution in the minimization, we know that $f_n\square(\sigma_n\psi_n)(x_{a(n)},y_n,x_n)\le f_n(x_{a(n)},y_n,x_n)$ for all $x_{a(n)}\in\calX_{a(n)}$ and $(x_n,y_n)\in\calF_n$.
    Pick any $x_1,x_2\in\calX_{a(n)}$, $(x,y)\in\calF_n$, and let $z_1,z_2\in\calX_{a(n)}$ be the corresponding minimizers in the definition of $f_n\square(\sigma_n\psi_n)(x_1,y,x)$ and $f_n\square(\sigma_n\psi_n)(x_2,y,x)$, respectively.
	By definition,
	\begin{align*}
		& f_n\square(\sigma_n\psi_n)(x_1,y,x)-f_n\square(\sigma_n\psi_n)(x_2,y,x)\\
		&=f_n(z_1,y,x)+\psi(x_1-z_1)-f_n(z_2,y,x)-\psi(x_2-z_2)\\
		&\le f_n(z_2,y,x)+\psi(x_1-z_2)-f_n(z_2,y,x)-\psi(x_2-z_2)
		\le \sigma_n\norm{x_1-x_2}.
	\end{align*}
	Similarly, we can get $f_n\square(\sigma_n\psi_n)(x_2)-f_n\square(\sigma_n\psi_n)(x_1)\le \sigma_n\norm{x_1-x_2}$ by exchanging $x_1,x_2$ and $z_1,z_2$ in the above inequality.
	Therefore, $f_n\square(\sigma_n\psi_n)$ is $\sigma_n$-Lipschitz continuous in the first variable $x_{a(n)}$.\\
	The regularized problem~\eqref{eq:RegularizedNodalProblem} can be viewed as replacing the nodal objective function $f_n$ with the inf-convolution $f_n\square(\sigma_n\psi_n)$.
	Then by Proposition~\ref{prop:OptimalValueFunctionProperties}, $Q_n^\Reg(x)$ is $\sigma_n$-Lipschitz continuous on $\calX_{a(n)}$.
	Moreover, if the original problem~\eqref{eq:NodalProblem} is convex and $\psi_n$ are convex penalty functions, then $f_n\square(\sigma_n\psi_n)$ is also convex.
	Proposition~\ref{prop:OptimalValueFunctionProperties} ensures $Q_n^\Reg(x)$ is also convex on $\calX_{a(n)}$.
	\qed
\end{proof}

\subsubsection{Proof for Lemma~\ref{lemma:RegOptimalValueFunction}}
\label{app:proof:lemma:RegOptimalValueFunction}
% \begin{statement}
% 	Under Assumption \ref{assum:ExactPenalty}, any optimal solution $(x_n,y_n)_{n\in\calN}$ to problem~\eqref{eq:ExtensiveForm} satisfies $Q_n^\Reg(x_{a(n)})=Q_n(x_{a(n)})$ for all $n\neq r$.
% \end{statement}
\begin{proof}
	By definition, we have $Q_n^\Reg(x_n)\le Q_n(x_n)$ for all $n\in\calN$, $n\neq r$.
	We show the other direction by contradiction.
	Suppose there exists a node $n\in\calN$ such that $Q_n^\Reg(x_{a(n)})<Q_n(x_{a(n)})$.
	By definition, there exist $z'_m\in\calX_{a(m)}$ and $(x'_m,y'_m)\in\calF_m$ for all nodes in the subtree $m\in\calT(n)$, such that 
	$$Q_n^\Reg(x_{a(n)})=\frac{1}{p_n}\sum_{m\in\calT(n)}p_m\left[f_m(z'_m,y'_m,x'_m)+\sigma_m\psi_m({x'_{a(m)}-z'_m})\right].$$
	We can extend $(x'_m,y'_m,z'_m)_{m\in\calT(n)}$ to a feasible solution $(z'_m,y'_m,x'_m)_{m\in\calN}$ of the regularized problem by setting $z'_m=x_{a(m)}$, $y'_m=y_m$, and $x'_m=x_m$ for all $m\notin\calT(n)$.
	Thus
	\begin{align*}
		v^\reg&\le \sum_{m\in\calT(n)}p_mf_m(z'_m,y'_m,x'_m)+\sum_{m\notin\calT(n)}p_mf_m(z'_m,y'_m,x'_m)\\
		&= p_nQ_n^\Reg(x_{a(n)})+\sum_{m\notin\calT(n)}p_mf_m(x_{a(m)},y_m,x_m)\\
		&<p_nQ_n(x_{a(n)})+\sum_{m\notin\calT(n)}p_mf_m(x_{a(m)},y_m,x_m)\\
		&=\sum_{m\in\calT(n)}p_mf_m(x_{a(m)},y_m,x_m)+\sum_{m\notin\calT(n)}p_mf_m(x_{a(m)},y_m,x_m)=v^\primal.
	\end{align*}	
	This leads to a contradiction with the assumption that $v^\reg=v^\primal$.
	Therefore, we conclude that $Q_n^\Reg(x_{a(n)})=Q_n(x_{a(n)})$ for all $n\in\calN$, $n\neq r$.
	\qed
\end{proof}

\subsubsection{Proof for Proposition~\ref{prop:GeneralizedCutTightness}}
\label{app:proof:prop:GeneralizedCutTightness}
% \begin{statement}
% Given the above definition of \eqref{eq:CutGenerationDualProblem}-\eqref{eq:CutGeneration}, if $(\bar{x}_n,\bar{y}_n)_{n\in\calN}$ is an optimal solution to problem~\eqref{eq:ExtensiveForm} and the bound $l_{n,\rho}$ satisfies $l_{n,\rho}\ge\sigma_n$ for all nodes $n$, then for every node $n$, the generalized conjugacy cut \eqref{eq:CutGeneration} is tight at $\bar{x}_n$, i.e. 
% $Q_n(\bar{x}_n) = C_n^{\Phi_n^{\bar{x}_n}}(\bar{x}_n\given\hat{\lambda}_n,\hat\rho_n,\hat{v}_n)$.
% \end{statement}
\begin{proof}
If $l_{n,\rho}\ge\sigma_n$, then $(\lambda,\rho)=(0,\sigma_n)$ is contained in $\calU_n$, and therefore, is a dual feasible solution for \eqref{eq:CutGenerationDualProblem}. Thus, we have
\begin{align*}
	Q_n(\bar{x}_n)\ge C_n^{\Phi_n^{\bar{x}_n}}(\bar{x}_n\given\hat{\lambda}_n,\hat\rho_n,\hat{v}_n)&=\hat{v}_n\ge\min_{z\in\calX_{a(n)}}\{Q_n(z)+\sigma_n\psi_n(\bar{x}_n-z)\}=Q_n^\Reg(\bar{x}_n)=Q_n(\bar{x}_n),
\end{align*}
where the first inequality is the validity of the generalized conjugacy cut \eqref{eq:GeneralizedConjugacyCutValidness} and the second and the last equality are due to  Lemma~\ref{lemma:RegOptimalValueFunction} for $(\bar{x}_n,\bar{y}_n)_{n\in\calN}$ being an optimal solution to problem~\eqref{eq:ExtensiveForm}. This completes the proof.
%Therefore, we conclude that $C_n^{\Phi_n^{\bar{x}}}(\bar{x}\given\hat\lambda_n,\hat\rho_n,\hat{v}_n)=Q_n(\bar{x})$.
\qed
\end{proof}

\subsubsection{Proof for Lemma~\ref{lemma:ConvexCutTightness}}
\label{app:proof:lemma:ConvexCutTightness}
% \begin{statement}
% 	Let $\calX\subset\R^d$ be a convex, compact set.
% 	Given a convex, proper, l.s.c.\ function $Q:\calX\to\R\cup\{+\infty\}$,  for any $x\in\calX$, the inf-convolution satisfies
% 	\begin{equation}
% 	    Q\square(\sigma\norm{\cdot})(x)\coloneqq\min_{z\in\calX}\{Q(z)+\sigma\norm{x-z}\}=\max_{\norm{\lambda}_*\le\sigma}\min_{z\in\calX}\{Q(z)+\innerprod{\lambda}{x-z}\}.\label{eq:ConvexCutInfConvProof}
% 	\end{equation}
% \end{statement}
\begin{proof}
	The minimums in \eqref{eq:ConvexCutInfConv} are well-defined because of the compactness of $\calX$ and lower semicontinuity of $Q$.
	Take any $x\in\calX$.
	Since both the primal set $\calX$ and the dual set $\{\lambda\in\R^d:\norm{\lambda}_*\le\sigma\}$ are bounded, by strong duality (cf. Theorem 3.1.30 in \cite{nesterov2018lectures}), we have
	\begin{align*}
		\max_{\norm{\lambda}_*\le\sigma}\min_{z\in\calX}\{Q(z)+\innerprod{\lambda}{x-z}\}&=\min_{z\in\calX}\max_{\norm{\lambda}_*\le\sigma}\{Q(z)+\innerprod{\lambda}{x-z}\}
		=\min_{z\in\calX}\{Q(z)+\sigma\norm{x-z}\},
	\end{align*}
	which completes the proof.
	\qed
\end{proof}

\subsubsection{Proof for Proposition~\ref{prop:ConvexCutTightness}}
\label{app:proof:prop:ConvexCutTightness}
% \begin{statement}
%     Suppose \eqref{eq:NodalProblem} is convex and $\psi_n(x)=\norm{x}$ for all nodes $n$.
%     Given the above definition of \eqref{eq:CutGenerationDualProblem}-\eqref{eq:CutGeneration}, if $(\bar{x}_n,\bar{y}_n)_{n\in\calN}$ is an optimal solution to problem~\eqref{eq:ExtensiveForm} and the bounds satisfy $l_{n,\lambda}\ge\sigma_n$, $l_{n,\rho}=0$ for all nodes $n$, then for every node $n$, the generalized conjugacy cut \eqref{eq:CutGeneration} is exact at $\bar{x}_n$, i.e. 
%     $Q_n(\bar{x}_n) = C_n^{\Phi_n^{\bar{x}_n}}(\bar{x}_n\given\hat{\lambda}_n,\hat\rho_n,\hat{v}_n)$.
% \end{statement}
\begin{proof}
    By definition, 
    $Q_n^\Reg(x)\le Q_n\square(\sigma_n\psi_n)(x)$.
    Since $\psi_n(x)=\norm{x}$ is convex, by Proposition~\ref{prop:InfConvolutionLipschitz}, $Q_n^\Reg(x)$ is convex.
	Then by Lemma~\ref{lemma:ConvexCutTightness}, we have
    $$Q_n^\Reg(x)=\max_{\norm{\lambda}_*\le\sigma_n}\min_{z\in\calX_{a(n)}}\left\{Q_n(z)+\innerprod{\lambda}{x-z}\right\}.$$
    Therefore,
    \begin{align*}
    	C_n^{\Phi_n^{\bar{x}_n}}(\bar{x}_n\given\hat{\lambda}_n,\hat\rho_n,\hat{v}_n)=\hat{v}_n &=\max_{\norm{\lambda}_*\le l_{n,\lambda}}\min_{z\in\calX_{a(n)}}\left\{Q_n(z)+\innerprod{\lambda}{\bar{x}_n-z}\right\}\\
    	&\ge \max_{\norm{\lambda}_*\le \sigma_n}\min_{z\in\calX_{a(n)}}\left\{Q_n(z)+\innerprod{\lambda}{\bar{x}_n-z}\right\} \ge Q_n^\Reg(\bar{x}_n).
    \end{align*}
    By Lemma~\ref{lemma:RegOptimalValueFunction}, $Q_n^\Reg(\bar{x}_n)=Q_n(\bar{x}_n)$ if $(\bar{x}_n,\bar{y}_n)_{n\in\calN}$ is an optimal solution to problem~\eqref{eq:ExtensiveForm}.
    Therefore, we conclude that $C_n^{\Phi_n^{\bar{x}_n}}(\bar{x}_n\given\hat\lambda_n,\hat\rho_n,\hat{v}_n)=Q_n(\bar{x}_n)$ due to the validness of $C_n^{\Phi_n^{\bar{x}_n}}$ by \eqref{eq:GeneralizedConjugacyCutValidness}.
    \qed
\end{proof}

\subsection{Proofs for Statements in Section~\ref{sec:DDPAlgorithms}}

\subsubsection{Proof for Proposition~\ref{prop:UnderApproximationValidness}}
\label{app:proof:prop:UnderApproximationValidness}
% \begin{statement}
% 	For any $n\in\calN$, and $i\in\N$, $\ulcQ_n^i(x)$ is $(\sum_{m\in\calC(n)}p_{nm}(l_{m,\lambda}+l_{m,\rho}))$-Lipschitz continuous and
% 	\begin{align*}
% 	    \calQ_n(x) \ge \ulcQ_n^i(x), \quad\forall x\in\calX_{n}.
% 	\end{align*}
% \end{statement}
\begin{proof}
	Let $L_n:=\sum_{m\in\calC(n)}p_{nm}(l_{m,\lambda}+l_{m,\rho})$ for simplicity.
	We prove the proposition recursively for nodes $n\in\calN$, and inductively for iteration indices $i\in\N$.
	For leaf nodes and the first iteration, it holds obviously because $\ulcQ_n^i(x)=0$ for any leaf node $n\in\calN$ with $\calC(n)=\varnothing$, and $\ulcQ_n^0(x)=0$ from the definition~\eqref{eq:UnderApproximationCostToGo}.
	Now suppose for some $n\in\calN$, and $i\in\N$, it holds for all $m\in\calC(n)$ that $\ulcQ_m^i(x)\le \calQ_m(x)$, $\ulcQ_n^{i-1}(x)\le \calQ_n(x)$, and that $\ulcQ_n^{i-1}(x)$ is $L_n$-Lipschitz continuous. 
	Then it follows from~\eqref{eq:CutGeneration}, \eqref{eq:UnderApproximationUpdateFormula}, and \eqref{eq:NonRootBackwardOracleProblem} that $C_m^i(x\given\hat\lambda_m^i,\hat\rho_m^i,\ubar{v}_m^i)\le Q_m(x)$ for all $m\in\calC(n)$.
	By~\eqref{eq:UnderApproximationCut}, $C_m^i$ is $(l_{m,\lambda}+l_{m,\rho})$-Lipschitz continuous so $\sum_{m\in\calC(n)}p_{nm}C_m^i$ is $L_n$-Lipschitz continuous.
	Thus the pointwise maximum of $\ulcQ_n^{i-1}(x)$ and $\sum_{m\in\calC(n)}p_{nm}C_m^i(x\given\hat\lambda_n^i,\hat\rho_n^i,\ubar{v}_n^i)$ (cf.~\eqref{eq:UnderApproximationCostToGo}) is still dominated by $\calQ_n(x)$ and $L_n$-Lipschitz continuous.
\qed
\end{proof}

\subsubsection{Proof for Proposition~\ref{prop:OverApproximationValidness}}
\label{app:proof:prop:OverApproximationValidness}
% \begin{statement}
% 	For any non-root node $n\in\calN$ and $i\ge1$, $\olcQ_n^i(x)$ is $(\sum_{m\in\calC(n)}p_{nm}\sigma_m)$-Lipschitz continuous.
% 	Moreover, we have $\bar{v}_m^i\ge Q_m^\Reg(x_{n}^i)$ for any node $m\in\calC(n)$ and thus
% 	\begin{align*}
% 	    \olcQ_n^i(x)\ge \calQ_n^\Reg(x),\quad\forall\,x\in\calX_{n}. 
% 	\end{align*}
% \end{statement}
\begin{proof}
	Let $L_n=\sum_{m\in\calC(n)}p_{nm}\sigma_m$ for simplicity in this proof.
	We prove the statement by induction on the number of iterations $i$. 
	When $i=1$, $\olcQ_n^1(x)=\min\{+\infty,\sum_{m\in\calC(n)}p_{nm}(\bar{v}_m^i+\sigma_m\norm{x-x_n^i})\}=\sum_{m\in\calC(n)}p_{nm}(\bar{v}_m^i+\sigma_m\norm{x-x_n^i})$ which is clearly $L_n$-Lipschitz continuous.
	For any leaf node $n$, $\olcQ_n^1\equiv0=\calQ_n^\Reg$ by definition.
	Going recursively from leaf nodes to the root node, suppose $\olcQ_m^1\ge \calQ_m^\Reg$ for all $m\in\calC(n)$ for some node $n$, then we have
	\begin{align}
	    \bar{v}_m^1 &= f_m({z}_m^1,{y}_m^1,{x}_m^1)+\sigma_m\psi_m(x_{n}^1-{z}_m^1)+\olcQ_m^{1}({x}_m^1)\label{eq:OverApproximationValidnessDerivation}\\
	    &\ge \min\{f_m(z,y,x)+\sigma_m\psi_m(x_{n}^1-z)+\olcQ_m^{1}(x):(x,y)\in\calF_m,z\in\calX_n\}\notag \\
	    &\ge \min\{f_m(z,y,x)+\sigma_m\psi_m(x_{n}^1-z)+\calQ_m^\Reg(x):(x,y)\in\calF_m,z\in\calX_n\}= Q_m^\Reg(x_n^1). \notag
	\end{align}
	Thus $\olcQ_n^1(x)=\sum_{m\in\calC(n)}p_{nm}(\bar{v}_m^1+\sigma_m\norm{x-x_n^1})\ge\calQ_n^\Reg(x)$ for all $x\in\calX_n$ by the $\sigma_m$-Lipschitz continuity of the regularized value functions $Q_m^\Reg(x)$ for all $m\in\calC(n)$ shown in Proposition~\ref{prop:InfConvolutionLipschitz}.
	
	Now assume that the statement holds for all iterations up to $i-1$.
	For any leaf node $n$, $\olcQ_n^i\equiv0=\calQ_n^\Reg$ still holds by definition.
	For any non-leaf node $n$, suppose $\olcQ_m^i\ge \calQ_m^\Reg$ for all $m\in\calC(n)$.
	Then by the same argument~\eqref{eq:OverApproximationValidnessDerivation}, we know that $\bar{v}_m^i\ge Q_m^\Reg(x_n^i)$.
	By induction hypothesis, $\olcQ_n^{i-1}(x)\ge\calQ_n^\Reg(x)$ for all $x\in\calX_n$. 
	So for the cases without convexity, $\olcQ_n^i(x)=\min\{\olcQ_n^{i-1}(x),\sum_{m\in\calC(n)}p_{nm}(\bar{v}_m^i+\sigma_m\norm{x-x_n^i})\}$ is $L_n$-Lipschitz continuous and satisfies $\olcQ_n^{i}(x)\ge\calQ_n^\Reg(x)$ since $\sum_{m\in\calC(n)}p_{nm}(\bar{v}_m^i+\sigma_m\norm{x-x_n^i})\ge\calQ_n^\Reg(x)$ for all $x\in\calX_n$ following Proposition~\ref{prop:InfConvolutionLipschitz}.
	
	It remains to show that in the convex case $\olcQ_n^i(x)$ is still $L_n$-Lipschitz continuous and satisfies $\olcQ_n^i(x)\ge\calQ_n^\Reg(x)$ for any $x\in\calX_n$. Note that $\calQ_n^\Reg(x)$ can be naturally extended to the entire space $\bbR^{d_n}\supset\calX_n$ since $\calQ_n^\Reg\square(L_n\norm{\cdot})(x)=\calQ_n^\Reg(x)$ for any $x\in\calX_n$ by the $L_n$-Lipschitz continuity of $\calQ_n^\Reg$. The above argument of the base case $i=1$ for the nonconvex case can be directly applied to the convex case over $x\in\R^{d_n}$. Now assume that $\olcQ_n^{i-1}(x)$ is $L_n$-Lipschitz continuous on $\R^{d_n}$ and $\olcQ_n^{i-1}(x)\ge \calQ_n^R(x)$ for $x\in\R^{d_n}$ up to $i-1$.
	Since $Q'_n(x):=\min\{\olcQ_n^{i-1}(x),\sum_{m\in\calC(n)}p_{nm}(\bar{v}_m^i+\sigma_m\norm{x-x_n^i})\}$ is $L_n$-Lipschitz continuous in $x\in\R^{d_n}$, we claim that the supremum in the definition~\eqref{eq:ConvexHullFunctionOperation} can be attained within the dual norm ball $\calB_*(L_n):=\{\lambda\in\bbR^{d_n}:\norm{\lambda}_*\le L_n\}$.
	In fact, for any $\lambda\notin\calB_*(L_n)$, the infimum 
	$$\inf_{z\in\bbR^{d_n}}\{Q'_n(z)+\innerprod{\lambda}{z-x}\}\le\inf_{z\in\bbR^{d_n}}\{Q'_n(x)+L_n\norm{z-x}+\innerprod{\lambda}{z-x}\}=-\infty.$$
	As a result, $\olcQ_n^i(x)$ is a supremum of $L_n$-Lipschitz linear functions (of the forms $l(x):=\olcQ_n^i(\hat{z})+\langle\hat{\lambda},\hat{z}-x\rangle$ where $\hat{\lambda}\in\calB_*(L_n)$ and $\hat{z}\in\calX_n$) and thus is also an $L_n$-Lipschitz continuous function.
	Therefore, $Q'_n(x)\ge\calQ_n^\Reg(x)$ for all $x\in\bbR^{d_n}$. By \eqref{eq:ConvexHullFunctionOperation},  $\olcQ_n^i(x)=(Q'_n)^{**}(x)\ge(\calQ_n^\Reg)^{**}(x)=\calQ_n^\Reg(x)$ for all $x\in\R^{d_n}$.
	This completes the proof.
	\qed
\end{proof}

\subsubsection{Proof for Proposition~\ref{prop:RootNodeOptimality}}
\label{app:proof:prop:RootNodeOptimality}
% \begin{statement}
% 	Given any $\epsilon > 0$, if $\UB-\LB\le \epsilon$, then the returned solution $(x_r^*,y_r^*)$ is an $\epsilon$-optimal root node solution to the regularized problem~\eqref{eq:RegularizedNodalProblem}.
% 	In particular, if $\olcQ_r^i(x_r^{i+1})-\ulcQ_r^i(x_r^{i+1})\le\epsilon$ for some iteration index $i$, then $\UB-\LB\le\epsilon$ and Algorithm~\ref{alg:DualDP} terminates after the $i$-th iteration.
% \end{statement}
\begin{proof}
	From the definition of $v^\reg$ and Proposition~\ref{prop:OverApproximationValidness},
	\begin{align*}
		v^\reg&\le f_r(x_{a(r)},y_r^*,x_r^*)+\calQ^\Reg_r(x_r^*)\le f_r(x_{a(r)},y_r^*,x_r^*)+\olcQ^i_r(x_r^*)\le\UB.
	\end{align*}
	Since $\UB-\LB\le\epsilon$, we have
	\begin{equation*}
		f_r(x_{a(r)},y_r^*,x_r^*)+\olcQ^i_r(x_r^*)\le f_r(x_{a(r)},y_r^{i+1},x_r^{i+1})+\ulcQ_r^i(x_r^{i+1})+\epsilon.
	\end{equation*}
	Then, using the optimality of $(x_r^{i+1},y_r^{i+1})$ given by $\scrO_r(\ulcQ_r^i)$ and the fact that $\ulcQ_r^i(x)\le \calQ_r(x)$, we see that
	\begin{equation*}
		f_r(x_{a(r)},y_r^{i+1},x_r^{i+1})+\ulcQ^i_r(x_r^{i+1})\le \min_{(x,y)\in\calF_r}\left\{f_r(x_{a(r)},y,x)+\calQ_r(x)\right\}=v^\primal.
	\end{equation*}
	Under Assumption~\ref{assum:ExactPenalty}, $v^\reg=v^\primal$.
	Therefore, combining all the above inequalities, we have shown that
	\begin{equation*}
		v^\reg\le f_r(x_{a(r)},y_r^*,x_r^*)+\olcQ^i_r(x_r^*)\le v^\reg+\epsilon,
	\end{equation*}
	which means $(x_r^*,y_r^*)$ is an $\epsilon$-optimal root node solution to the regularized problem~\eqref{eq:RegularizedNodalProblem}. 
	Now suppose $\olcQ_r^i(x_r^{i+1})-\ulcQ_r^i(x_r^{i+1})\le\epsilon$ for some iteration index $i$.
	Note that $\UB\le f_r(x_{a(r)},y_r^{i+1},x_r^{i+1})+\olcQ_r^i(x_r^{i+1})$, we have 
	\begin{align*}
		&\UB-\LB\\
		&\le f_r(x_{a(r)},y_r^{i+1},x_r^{i+1}) + \olcQ_r^i(x_r^{i+1})  - (f_r(x_{a(r)},y_r^{i+1},x_r^{i+1}) + \ulcQ_r^i(x_r^{i+1}))\\
		& = \olcQ_r^i(x_r^{i+1})-\ulcQ_r^i(x_r^{i+1})
		 \le\epsilon.
	\end{align*}
	Therefore the algorithm terminates after the $i$-th iteration.
	\qed
\end{proof}

\subsection{Proofs for Statements in Section~\ref{sec:ComplexityUpperBounds}}
\subsubsection{Proof for Lemma~\ref{lemma:NeighborhoodApproximation}}
\label{app:proof:lemma:NeighborhoodApproximation}
% \begin{statement}
% 	If an iteration index $i\in\calI_n(\delta)$, i.e., $\olcQ_n^{i-1}(x_n^i)-\ulcQ_n^{i-1}(x_n^i)>\gamma_n(\delta)$ and $\olcQ_m^{i-1}(x_m^i)-\ulcQ_m^{i-1}(x_m^i)\le\gamma_m(\delta)$ for all $m\in\calC(n)$, then
% 	$$\olcQ_n^i(x)-\ulcQ_n^i(x)\le\gamma_n(\delta),\quad\forall\,x\in\calX_n,\,\Vert{x-x_n^i}\Vert\le\frac{\delta_n}{2L_n},$$
% 	where $L_n\coloneqq\sum_{m\in\calC(n)}p_{nm}\max\{\sigma_m,l_{m,\lambda}+l_{m,\rho}\}$.
% \end{statement}
\begin{proof}
	By definition~\eqref{eq:UnderApproximationCostToGo}, $\ulcQ_m^i(x)\ge\ulcQ_m^{i-1}(x)$ on $\calX_m$ for all $m\in\calC(n)$.
        If the problem is convex with \(\psi_n=\norm{\cdot}\), then by Lemma~\ref{lemma:ConvexCutTightness}, we have
	\begin{align*}
	  \ubar{v}_m^i & = \max_{\norm{\lambda}_*\le l_{m,\lambda}}\min_{\substack{(x,y)\in\calF_m,\\z\in\calX_n}}\left\{f_m(z,y,x)+\bangle{\lambda}{x_n^i-z}+\ulcQ_m^i(x)\right\}\\
                       & = \min_{\substack{(x,y)\in\calF_m,\\z\in\calX_n}}\left\{f_m(z,y,x)+l_{n,\lambda}\nVert{x_n^i-z}+\ulcQ_m^i(x)\right\}\\
                       & \ge \min_{\substack{(x,y)\in\calF_m,\\z\in\calX_n}}\left\{f_m(z,y,x)+\sigma_m\nVert{x_n^i-z}+\ulcQ_m^i(x)\right\}
        \end{align*}
        as \(l_{m,\lambda}\ge \sigma_m\) in Definition~\ref{def:NonrootBackwardStepOracle}.
	Otherwise, by definition~\eqref{eq:UnderApproximationUpdateFormula} and the fact that $(\lambda,\rho)=(0,\sigma_n)$ is a dual feasible solution for the problem~\eqref{eq:NonRootBackwardOracleProblem},  we have
	\begin{align*}
	  \ubar{v}_m^i & = \max_{(\lambda,\rho)\in\calU_m}\min_{\substack{(x,y)\in\calF_m,\\z\in\calX_n}}\left\{f_m(z,y,x)+\bangle{\lambda}{x_n^i-z}+\rho\psi_n(x_n^i-z)+\ulcQ_m^i(x)\right\}\\
	  & \ge \min_{\substack{(x,y)\in\calF_m,\\z\in\calX_n}}\left\{f_m(z,y,x)+\sigma_n\psi_n(x_n^i-z)+\ulcQ_m^i(x)\right\}.
        \end{align*}
        Thus in both cases, we have
        \begin{align*}
	 \ubar{v}_m^i & \ge \min_{\substack{(x,y)\in\calF_m,\\z\in\calX_n}}\left\{f_m(z,y,x)+\sigma_n\psi_n(x_n^i-z)+\ulcQ_m^{i-1}(x)\right\}\\
	  &= f_m(z_m^i,y_m^i,x_m^i)+\sigma_n\psi_n(x_n^i-z_m^i)+\ulcQ_m^{i-1}(x_m^i)
	\end{align*}
	for all $m\in\calC(n)$. 
        The last equality is due to the forward step subproblem oracle $\scrO_m^\Fwd(x_n^i,\ulcQ_m^{i-1})$ in the algorithm.
	Meanwhile, note that $\olcQ_m^i(x)\le\olcQ_m^{i-1}(x)$ for $x\in\calX_m$. By definition~\eqref{eq:OverApproximationUpdateFormula}, we have
	\begin{align*}
	    \bar{v}_m^i & = f_m(z_m^i, y_m^i, x_m^i) + \sigma_n \psi_n(x_n^i - z_m^i) +  \olcQ_m^{i}(x_m^i)\\
	    &\le f_m(z_m^i, y_m^i, x_m^i) + \sigma_n \psi_n(x_n^i - z_m^i) +  \olcQ_m^{i-1}(x_m^i)
	\end{align*} 
	for all $m\in\calC(n)$.
	Note that by definition~\eqref{eq:OverApproximationCostToGo}, $\olcQ_n^i(x_n^i)\le \sum_{m\in\calC(n)}p_{nm}\bar{v}_m^i$ and by definitions~\eqref{eq:UnderApproximationCostToGo} and \eqref{eq:UnderApproximationCut}, $\ulcQ_n^i(x_n^i)\ge \sum_{m\in\calC(n)}p_{nm}C_m^i(x_n^i\given\hat\lambda_m^i,\hat\rho_m^i,\ubar{v}_m^i)=\sum_{m\in\calC(n)}p_{nm}\ubar{v}_m^i$.
	Therefore, 
	\begin{align*}
		\olcQ_n^i(x_n^i)-\ulcQ_n^i(x_n^i)&\le\sum_{m\in\calC(n)}p_{nm}(\bar{v}_m^i-\ubar{v}_m^i)\\
		&\le\sum_{m\in\calC(n)}p_{nm}[\olcQ_m^{i-1}(x_m^i)-\ulcQ_m^{i-1}(x_m^i)]
		\le\sum_{m\in\calC(n)}p_{nm}\gamma_m(\delta).
	\end{align*}
	Note that $\olcQ_n^i(x)$ is $\big(\sum_{m\in\calC(n)}p_{nm}\sigma_m\big)$-Lipschitz continuous by Proposition~\ref{prop:OverApproximationValidness}, and $\ulcQ_m^i(x)$ is $\big[\sum_{m\in\calC(n)}p_{nm}(l_{n,\lambda}+l_{n,\rho})\big]$-Lipschitz continuous on $\calX_n$ by Proposition~\ref{prop:UnderApproximationValidness}.
        Since we have $l_{m,\lambda}+l_{m,\rho}\ge\sigma_m$ regardless of convexity of the problem, it holds that $\olcQ_n^i(x)$ and $\ulcQ_n^i(x)$ are both $L_n$-Lipschitz continuous.
	Therefore, for any $x\in\calX_n$, $\|x-x_n^i\|\le{\delta_n}/{(2L_n)}$, we have
	\begin{align*}
		\olcQ_n^i(x)-\ulcQ_n^i(x)&\le\olcQ_n^i(x_n^i)-\ulcQ_n^i(x_n^i)+2L_n\|x-x_n^i\|
		\le\sum_{m\in\calC(n)}p_{nm}\gamma_m(\delta)+\delta_n
		=\gamma_n(\delta).
	\end{align*}
        This completes the proof.
        \qed
\end{proof}

\subsubsection{Proof for Lemma~\ref{lemma:IndexSetCardinality}}
\label{app:proof:lemma:IndexSetCardinality}
% \begin{statement}
% 	Let $\scrB=\{\calB_{n,k}\subset\R^{d_n}\}_{1\le k\le K_n,n\in\calN}$ be a collection of balls, each with diameter $D_{n,k}\ge0$, such that $\calX_n\subseteq\bigcup_{k=1}^{K_n}\calB_{n,k}$.
% 	Then,
% 	\begin{equation}
% 	\abs{\calI_n(\delta)}\le\sum_{k=1}^{K_n}\left(1+\frac{2L_nD_{n,k}}{\delta_n}\right)^{d_n}.
% 	\end{equation} 
% \end{statement}
\begin{proof}
	We claim that for any $i,j\in\calI_n$, $i\neq j$, then $\|x_n^i-x_n^j\|>\delta_n/(2L_n)$.
	Assume for contradiction that $\|x_n^i-x_n^j\|\le{\delta_n}/(2L_n)$ for some $i<j$ and $i,j\in \calI_n(\delta)$. 
	By the definition of $\calI_n(\delta)$, $\gamma_m^i\le\gamma_m(\delta)$ for all $m\in\calC(n)$. 
	By Lemma~\ref{lemma:NeighborhoodApproximation}, $\olcQ_n^i(x)-\ulcQ_n^i(x)\le\gamma_n(\delta)$ for all $x\in\calX_n$, $\|x-x_n^i\|\le\delta_n/(2L_n)$.
	Since $j>i$ and $\|x_n^i-x_n^j\|\le{\delta_n}/(2L_n)$, this implies $\gamma_n^j=\olcQ_n^j(x_n^j)-\ulcQ_n^j(x_n^j)\le \gamma_n(\delta)$, which is a contradiction with $j\in\calI_n(\delta)$.
	Hence we prove the claim.\\
	Let $\calB(R),\calB(R,x)\subseteq \R^d$ denote the closed balls with radius $R\ge0$, centered at 0 and $x$, respectively.
	It follows from the claim that the closed balls $\calB(\delta_n/(4L_n),x_n^i)$ are non-overlapping for all $i\in \calI_n(\delta)$, each with the volume $\Vol\calB(\delta_n/(4L_n))$.
	Thus the sum of the volumes of these balls is $\abs{\calI_n(\delta)}\Vol\calB(\delta_n/(4L_n))$. 
	Note that for each index $i\in\calI_n(\delta)$, $x_n^i\in\calX_n$ and hence $x_n^i\in\calB_{n,k}$ for some $k$.
	The closed ball $\calB(\delta_n/(4L_n),x_n^i)\subseteq\calB_{n,k}+\calB(\delta_n/(4L_n))$, and therefore
	$$\bigcup_{i\in\calI_n(\delta)}\calB(\delta_n/(4L_n),x_n^i)\subseteq\bigcup_{k=1}^{K_n}(\calB_{n,k}+\calB(\delta_n/(4L_n))).$$
	It follows that
	\begin{align*}
            \Vol\bigg[&\bigcup_{i\in\calI_n(\delta)}\calB(\delta_n/(4L_n),x_n^i)\bigg]=\abs{\calI_n(\delta)}\cdot\Vol\calB(\delta_n/(4L_n))\\
		&\le\Vol\left[\bigcup_{k=1}^{K_n}(\calB_{n,k}+\calB(\delta_n/(4L_n))\right]
		\le \sum_{k=1}^{K_n}\Vol\left(\calB_{n,k}+\calB(\delta_n/(4L_n))\right).
	\end{align*}
	Therefore, 
	$$\abs{\calI_n(\delta)}\le\sum_{k=1}^{K_n}\frac{\Vol\left(\calB_{n,k}+\calB(\delta_n/(4L_n))\right)}{\Vol\calB(\delta_n/(4L_n))}=\sum_{k=1}^{K_n}\left(1+\frac{2L_nD_{n,k}}{\delta_n}\right)^{d_n}.$$
        This completes the proof.
	\qed
\end{proof}

\subsubsection{Proof for Lemma~\ref{lemma:StochasticDDPEventProbability}}
\label{app:proof:lemma:StochasticDDPEventProbability}
% \begin{statement}
    % Fix any $\epsilon=\sum_{t=0}^{T-1}\delta_t$.
    % Then the conditional probability inequality
    % $$\Prob(A_i(\delta)\given \Sigma_{i-1})\ge \nu:=1-(1-1/N)^M,$$ 
    % holds almost surely, 
    % where $N\coloneqq \prod_{t=2}^{T-1}N_t$ if $T\ge 2$ and $N\coloneqq 1$ otherwise.
% \end{statement}
\begin{proof}
    For each iteration $i\in\bbN$, the event $\cup_{j=1}^M\{\gamma_0^{i-1,j}\le \gamma_0(\delta)=\epsilon\}$ is $\Sigma_{i-1}$-measurable, so it suffices to prove this inequality for its complement in $A_i(\delta)$.
    Note that 
    $$\Prob\{\gamma_t^{i,j}=\tilde\gamma_t^{i,j}\given \Sigma_{i-1}\}\ge
    \Prob\{P^{i,j}_t=n(\tilde{\gamma}^{i,j}_t)\given \Sigma_{i-1}\},$$
    where $n(\tilde{\gamma}^{i,j}_t)$ is the smallest node index $n\in\tcN(t)$ such that $\calQ_t^\Reg(x_n)-\ulcQ_t^{i-1}(x_n)=\tilde{\gamma}^{i,j}_t$ for $(x_n,y_n,z_n)=\scrO_n^\Fwd(x_{t-1}^{i,j},\ulcQ_n^{i-1})$, which is determined given $\Sigma_{i-1}$.
    Using the same argument as in the proof of Theorem~\ref{thm:DDPComplexityUpperBound}, Lemma~\ref{lemma:NeighborhoodApproximation} shows that the event $\cap_{t=1}^{T-1}\{\gamma_t^{i,j}=\tilde{\gamma}^{i,j}_t\}$ implies the event $\{i\in\cup_{t=0}^{T-1}\calI_t(\delta)\}$ and hence the event $A_i(\delta)$ for each $j=1,\dots,M$.
    Therefore, since $\Sigma_{i-1}$ is contained in $\sigma(\Sigma^\oracle_\infty,\sigma\{P^{i',j'}\}_{(i',j')\neq (i,j)})$, by the independent, uniform sampling (Assumption~\ref{assum:StochasticSampling}), we have
    \begin{align*}
    &\Prob(A_i(\delta)\given\Sigma_{i-1}) \\
    &\ge\Prob\biggl(\bigcup_{j=1}^M\bigcap_{t=1}^{T-1}\{\gamma_t^{i,j}=\tilde\gamma_t^{i,j}\}\biggm\vert \Sigma_{i-1}\biggr)\\
    &\ge\Prob\biggl(\bigcup_{j=1}^M\bigcap_{t=1}^{T-1}\{\gamma_t^{i,j}=n(\tilde{\gamma}_t^{i,j})\}\biggm\vert \Sigma_{i-1}\biggr)\\
    &=1-\biggl(1-\Prob\biggl(\bigcap_{t=1}^{T-1}\{\gamma_t^{i,j}=n(\tilde{\gamma}_t^{i,j})\}\biggm\vert \Sigma_{i-1}\biggr)\biggr)^M
    = 1-(1-1/N)^M.
    \end{align*}
    Here, the last step follows from $\Prob(\bigcap_{t=1}^{T-1}\{\gamma_t^{i,j}=n(\tilde{\gamma}_t^{i,j})\}\mid \Sigma_{i-1})=\prod_{t=1}^{T-1}\Prob(\{\gamma_t^{i,j}=n(\tilde{\gamma}_t^{i,j})\}\mid \Sigma_{i-1})=\prod_{t=1}^{T-1}(1/N_t)=N$.
    \qed
\end{proof}

\subsubsection{Proof for Theorem~\ref{thm:StochasticDDPComplexityUpperBound}}
\label{app:proof:thm:StochasticDDPComplexityUpperBound}
% \begin{statement}
% 	Let $I=I(\delta,\scrB)$ denote the iteration complexity bound in Theorem~\ref{thm:DeterministicDDPComplexityUpperBound}, determined by the vector $\delta$ and the collection of state space covering balls $\scrB$, and $\nu$ denote the probability bound proposed in Lemma~\ref{lemma:StochasticDDPEventProbability}.
% 	Moreover, let $\iota$ be the random variable of the smallest index such that the root node solution $(x_0^{\iota+1},y_0^{\iota+1})$ is $\epsilon$-optimal in Algorithm~\ref{alg:StochasticDualDP}.
% 	Then for any real number $\kappa>1$, the probability 
% 	$$\Prob\left(\iota\ge 1+\frac{\kappa I}{\nu}\right)\le\exp\left(\frac{-2I\nu(\kappa-1)^2}{\kappa}\right).$$
% \end{statement}
\begin{proof}
    Let $a_i:=\indfunc_{A_i}$ denote the indicator of the event $A_i$ for $i\in\N$, and $S_i:=\sum_{i'=1}^i a_i$.
    Note that the event $\{\iota\ge i\}$ implies the event $\{S_i\le I\}$, so we want to bound probability of the latter for sufficiently large indices $i$.

    By Lemma~\ref{lemma:StochasticDDPEventProbability}, we see that the adapted sequence $\{S_i-i\nu\}_{i=1}^{\infty}$ is a submartingale with respect to the filtration $\{\Sigma_i\}_{i=1}^\infty$, because
    $$\bbE(S_i-i\nu\mid\Sigma_{i-1})=S_{i-1}-(i-1)\nu+(\bbE(a_i\mid\Sigma_{i-1})-\nu)\ge S_{i-1}-(i-1)\nu.$$
    Moreover, it has a bounded difference as $S_i+i\nu-(S_{i-1}+(i-1)\nu)=a_i+\nu\le 2$ almost surely.
    Now apply the one-sided Azuma-Hoeffding inequality and we get for any $k>0$ that
    $$\Prob(S_i\le i\nu-k)\le\exp\biggl(-\frac{k^2}{8i}\biggr).$$
    For any $\kappa>1$, take the smallest iteration index $i$ such that $i\nu\ge \kappa I$, and set $k:=(\kappa-1)I$.
    Since $I\ge\frac{i\nu}{2\kappa}$, the probability bound can be then written as
    $$\Prob(\iota\ge i)\le \Prob(S_i\le I)\le\exp\biggl(-\frac{(\kappa-1)^2I^2}{8i}\biggr)\le\exp\biggl(-\frac{(\kappa-1)^2I\nu}{16\kappa}\biggr)$$
    Substitute the left-hand-side with $\Prob(\iota\ge 1+\frac{\kappa I}{\nu})$ using the definition of $i$ and we have obtained the desired inequality.
    \qed
\end{proof}

\subsection{Proofs for Statements in Section~\ref{sec:ComplexityLowerBounds}}
\subsubsection{Proof for Lemma~\ref{lemma:LipschitzOptimalValueFunctionApproximation}}
\label{app:proof:lemma:LipschitzOptimalValueFunctionApproximation}
% \begin{statement}
% Suppose $f_n(z,y,x)$ is $l_n$-Lipschitz continuous in $z$ with for each $n\in\calN$.
% 	If we choose $\psi_n(x)=\norm{x}$ and $\sigma_n\ge l_n$, then $Q_n^\Reg(x)=Q_n(x)$ on $\calX_{a(n)}$ for all non-root nodes $n\in\calN$.
% \end{statement}
\begin{proof}
	We prove the lemma recursively starting from the leaf nodes.
	For leaf nodes $n\in\calN$, $\calC(n)=\varnothing$, $Q_n^\Reg(x)=\min_{z\in\calX_{a(n)}}Q_n(z)+\sigma_n\psi_n(x-z)\ge\min_{z\in\calX_{a(n)}}Q_n(z)+l_n\norm{x-z}$.
	Since $Q_n$ is $l_n$-Lipschitz continuous, $Q_n(z)\ge Q_n(x)-l_n\norm{x-z}$.
	Therefore, $Q_n^\Reg(x)\ge Q_n(x)$ and by Proposition~\ref{prop:InfConvolutionLipschitz}, we know $Q_n^\Reg(x)=Q_n(x)$ for all $x\in\calX_{a(n)}$.\\
	Now suppose for a node $n\in\calN$, we know that all of its child nodes satisfy $Q_m^\Reg(x)=Q_m(x),\forall\,x\in\calX_n$, for all $m\in\calC(n)$.
	Then by definition,
	\begin{align*}
	    Q_n^\Reg(x_{a(n)})=\min_{(x,y)\in\calF_n,z\in\calX_{a(n)}}f_n(z,y,x)+\sigma_n\psi_n(x_{a(n)}-z)+\calQ_n^\Reg(x).
	\end{align*}
	By assumption, we know that $\calQ_n^\Reg(x)=\calQ_n(x)$ for all $x\in\calX_n$.
	Therefore, $Q_n^\Reg(x_{a(n)})=\min_{z\in\calX_{a(n)}}Q_n(z)+\sigma_n\psi_n(x_{a(n)}-z)\ge\min_{z\in\calX_{a(n)}}Q_n(z)+l_n\|{x_{a(n)}-z}\|$.
	Then again by $l_n$-Lipschitz continuity of $f_n$, we conclude that $Q_n^\Reg(x)=Q_n(x)$ for all $x\in\calX_{a(n)}$.
	\qed
\end{proof}

\subsubsection{Proof for Lemma~\ref{lemma:GeneralLipschitzFunction}}
\label{app:proof:lemma:GeneralLipschitzFunction}
% \begin{statement}
% 	Consider a norm ball $\calX=\{x\in\R^d:\norm{x}\le D/2\}$ and a finite set of points $\calW=\{w_k\}_{k=1}^K\subset\calX$.
% 	Suppose that there is $\beta>0$ and an $L$-Lipschitz continuous function $f:\calX\to\R_+$ such that $\beta<f(w_k)<2\beta$ for $k=1,\dots,K$. 
% 	Define 
% 	\begin{itemize}
% 		\item $\displaystyle\ulQ(x)\coloneqq\max_{k=1,\dots,K}\{0,f(w_k)-L\norm{x-w_k}\}$ and 
% 		\item $\displaystyle\olQ(x)\coloneqq\min_{k=1,\dots,K}\{f(w_k)+L\norm{x-w_k}\}$.
% 	\end{itemize}
% 	 If $K<\left(\frac{DL}{4\beta}\right)^d$, then $\displaystyle\min_{x\in\calX}\ulQ(x)=0$ and $\displaystyle\min_{x\in\calX}\olQ(x)>\beta$.
% \end{statement}
\begin{proof}
	We claim that if $K<\left(\frac{DL}{4\beta}\right)^d$, then there exists a point $\hat{x}\in\calX$ such that $\norm{\hat{x}-w_k}\ge\frac{2\beta}{L}$ for all $k=1,\dots,K$.
	We prove the claim by contradiction.
	Suppose such a point does not exist, or equivalently, for any point $x\in\calX$, there exists $w_k\in\calW$ such that $\norm{x-w_k}<\frac{2\beta}{L}$.
	This implies that the balls $\calB(2\beta/L,w_k)$ cover the set $\calX$, which leads to
	\begin{align*}
		\Vol\calX\le\Vol\left(\bigcup_{k=1}^K\calB(2\beta/L,w_k)\right)
		\le\sum_{k=1}^K\Vol\calB(2\beta/L,w_k)
		=K\cdot\Vol\calB(2\beta/L).
	\end{align*}
	Therefore, it must hold that $K\ge\Vol\calX/\Vol\calB(2\beta/L)=\left(\frac{DL}{4\beta}\right)^d$, hence a contradiction.\\
	The existence of $\hat{x}$ guarantees that $f(w_k)-L\|\hat{x}-w_k\|\le f(w_k)-2\beta<0$ for each $k=1,2,\dots,K$. 
	Therefore, $0\le \min_{x\in\calX}\ulQ(x)\le\ulQ(\hat{x})=\max_{1\le k\le K}\{0, f(w_k)-L\|\hat{x}-w_k\|\}=0$.
	From compactness of $\calX$ and the continuity of $\olQ(x)$, we have the inequality $\min_{x\in\calX}\olQ(x)\ge\min_{1\le k\le K}\olQ(w_k)=\min_{1\le k\le K}f(w_k)>\beta$, which completes the proof.
	% Now pick any $\bar{x}\in\argmin_{x\in\calX}\ulQ(x)$.
	% From the above argument, $\ulQ(\bar{x})=0$, which implies that $f(w_k)-L\|\bar{x}-w_k\|\le 0$, for each $k$. Hence, $\min_{k=1,\dots,K}L\norm{\bar{x}-w_k}\ge f(w_k) >\beta/2$.
	% This again implies $\olQ(\bar{x})>\beta$ by the definition of $\olQ$.
	\qed
\end{proof}

\subsubsection{Proof for Theorem~\ref{thm:GeneralLipschitzComplexity}}
\label{app:proof:thm:GeneralLipschitzComplexity}
% \begin{statement}
% For any optimality gap $\epsilon>0$, there exists a problem of the form~\eqref{eq:LipschitzBoxExample} with subproblem oracles $\scrO_n^\Fwd,\scrO_n^\Bwd$, $n\in\calN$, and $\scrO_r$, such that if Algorithm~\ref{alg:DeterministicDualDP} gives $\UB-\LB\le\epsilon$ in the $i$-th iteration, then
% 	$$i\ge\left(\frac{DLT}{4\epsilon}\right)^d.$$
% \end{statement}
\begin{proof}
	Let us define the forward subproblem oracle $\scrO_n^\Fwd$ in iteration $i$ and stage $t$ as mapping $(x_{t-1}^i,\ulQ_{t+1}^{i-1})$ to an optimal solution $(x_t^i, z_t^i)$ of the forward subproblem
	\begin{align*}
	    \min_{x_t, z_t\in \calX_t}\left\{f_t(z_t)+L\|x_{t-1}^i - z_t\| + \ulQ_{t+1}^{i-1}(x_t)\right\},
	\end{align*}
	and the backward subproblem oracle $\scrO_n^\Bwd$ in iteration $i$ and stage $t$ as mapping $(x_{t-1}^i, \ulQ_{t+1}^i)$ to an optimal solution $(\hat{x}_t^i, \hat{z}_t^i; \hat{\lambda}_t^i=0, \hat{\rho}_t^i=L)$ of the backward subproblem
	\begin{equation}\label{eq:LipschitzBoxDualProblem}
        \max_{\substack{\lambda=0\\ 0\le\rho\le L}}\min_{x_t,z_t\in\calX_t}\left\{f_t(z_t)+\rho\|x_{t-1}^i-z_t\|+
	\ulQ_{t+1}^{i}(x_t)\right\}.
    \end{equation}
	Note that in the backward subproblem \eqref{eq:LipschitzBoxDualProblem}, we choose that $l_{n,\lambda}=0$ and $l_{n,\rho}=L$. It is observed that the objective function in \eqref{eq:LipschitzBoxDualProblem} is nondecreasing in $\rho$. Therefore, $\hat{\rho}_t^i=L$ is always an optimal solution for the outer maximization in \eqref{eq:LipschitzBoxDualProblem}. The root-node oracle $\scrO_r$ in iteration $i$ simply solves $\min_{x_0\in\calX}\ulQ_1^i(x_0)$ and outputs $x_0^{i+1}$.
	
	In the backward step (Algorithm \ref{alg:DeterministicDualDP}, step \ref{alg:DeterministicDualDP:cut}) and c.f. the definition~\eqref{eq:UnderApproximationCut}, the new generalized conjugacy cut in iteration $k\le i$ is generated by 
	\begin{align*}
	    C_t^k(x\given 0,L,\ubar{v}_t^k)=\ubar{v}_t^k-L\|x-x_{t-1}^k\|=\ubar{v}_t^k-L\|x-x_{t-1}^k\|,
	\end{align*}
	for node $t\ge 1$, where $\ubar{v}_t^k$ is computed and upper bounded as
	\begin{align*}
	    \ubar{v}_t^k &= f_t(\hat{z}_t^k)+L\|x_{t-1}^k - \hat{z}_t^k\|+\min_{x_t\in\calX_t}\ulQ_{t+1}^k(x_t),\\
	    &\le f_t(x_{t-1}^k)+\min_{x_t\in\calX_t}\ulQ_{t+1}^k(x_t),\\
	    &\le f_t(x_{t-1}^k)+\min_{x_t\in\calX_t}\ulQ_{t+1}^i(x_t),
	\end{align*}
	where the first inequality directly follows from \eqref{eq:LipschitzBoxDualProblem}, as $z=x_{t-1}^i$ is a feasible solution of the inner minimization problem, and the second inequality is due to the monotonicity $\ulQ_{t+1}^k(x)\le \ulQ_{t+1}^i(x)$ for $k\le i$.
	Therefore,
	\begin{align}
	    \ulQ_t^i(x)&=\max_{k=1,\dots,i}\left\{0,\ C_t^k(x\given 0,L,\ubar{v}_t^k)\right\}, \quad\text{(by \eqref{eq:UnderApproximationCostToGo}),}\notag\\
	    &=\max_{k=1,\dots,i}\left\{0,\ \ubar{v}_t^k-L\|x-x_{t-1}^k\|\right\},\notag\\
	    &\le \max_{k=1,\dots,i}\left\{0,\  f_t(x_{t-1}^k)-L\|x-x_{t-1}^k\|\right\}+\min_{x_t\in\calX_t}\ulQ_{t+1}^i(x_t).\label{eq:UpperBoundonUnderApproximation}
	\end{align}
	Similarly, by \eqref{eq:OverApproximationUpdateFormula}, the upper approximation of the value function is computed and lower bounded as
	\begin{align*}
	    \bar{v}_t^k &= f_t(z_t^k) + L\|z_t^k - x_{t-1}^k\| + \olQ_{t+1}^k(x_t^k), \\
	    &\ge f_t(x_{t-1}^k) + \olQ_{t+1}^k(x_t^k), \\
	    &\ge f_t(x_{t-1}^k) + \min_{x_t\in{\calX_t}}\olQ_{t+1}^k(x_t),\\
	    &\ge f_t(x_{t-1}^k) + \min_{x_t\in{\calX_t}}\olQ_{t+1}^i(x_t),
	\end{align*}
	where $\hat{\calX}_t^k := \argmin_{x_t\in\calX_t}\ulQ_{t+1}^{k-1}(x_t)$. 
	Therefore, the over-approximation satisfies
	\begin{align}
	    \olQ_t^i(x)&=\min_{k=1,\dots,i}\left\{\bar{v}_t^k+L\|x-x_{t-1}^k\|\right\}, \quad\text{(by \eqref{eq:OverApproximationCostToGo})},\notag\\
	    &\ge\min_{k=1,\dots,i}\left\{f_t(x_{t-1}^k)+L\norm{x-x_{t-1}^k}\right\}+\min_{x_t\in{\calX_t}}\olQ_{t+1}^i(x_t),\notag\\
	    &> \frac{\epsilon}{T} + \min_{x_t\in{\calX_t}}\olQ_{t+1}^i(x_t), \label{eq:LowerBoundonOverApproximation}
	\end{align}
	where \eqref{eq:LowerBoundonOverApproximation} follows from the construction that $f_t(x)>\epsilon/T$ for all $x\in\calX_t$. 

    Now using \eqref{eq:UpperBoundonUnderApproximation} and \eqref{eq:LowerBoundonOverApproximation}, we can prove the statement of the theorem.
	Suppose the iteration index $i<\left(\frac{DLT}{4\epsilon}\right)^d$. 
	Denote $w_k:=x_{t-1}^k$ for $k=1,\dots,i$. Since $\epsilon/T<f_t(w_k)<2\epsilon/T$ by construction, applying Lemma~\ref{lemma:GeneralLipschitzFunction}, we get $\min_{x_t\in\calX_t}\left\{\max_{k=1,\dots,i}\left\{0,\  f_t(x_{t-1}^k)-L\|x_t-x_{t-1}^k\|\right\}\right\}=0$. By \eqref{eq:UpperBoundonUnderApproximation}, $\min_{x_{t-1}\in\calX_{t-1}}\ulQ_t^i(x_{t-1})\le\min_{x_t\in\calX_t}\ulQ_{t+1}^i(x_t)$ for $t=1,\dots,T$. Note at stage $T$, $\ulQ_{T+1}^i\equiv 0$. Therefore, $\min_{x_{t-1}\in\calX_{t-1}}\ulQ_t^i(x_{t-1})\le0$ for all $t=1,\dots,T$. But since $\ulQ_t^i(x)\ge 0$ for all $x\in \calX_{t-1}$, we have $\min_{x\in\calX_{t-1}}\ulQ_t^i(x)=0$ for all $1\le t\le T$. Hence we see that $\LB=\min_{x_0\in \calX_0}\ulQ_1^i(x_0)=0$ in iteration $i$.
	
	Since $\calX_t$ is a norm ball, it is compact. So by \eqref{eq:LowerBoundonOverApproximation}, we have \begin{align*}
	    \min_{x_{t-1}\in\calX_{t-1}}\olQ_t^i(x_{t-1})> \epsilon/T+\min_{x_t\in\calX_t}\olQ_{t+1}^i(x_t),\quad \forall 1\le t\le T.
	\end{align*}
	This recursion implies that $\min_{x_0\in\calX_0}\olQ_1^i(x_0)> T({\epsilon}/{T})=\epsilon$. According to Algorithm \ref{alg:DeterministicDualDP}, Steps \ref{alg:DeterministicDDP:UB1}-\ref{alg:DeterministicDDP:UB2}, we have that in iteration $i$, $\displaystyle\UB = \min_{k=1,\dots,i}\{\olQ_1^k(x_0^{k+1})\}\ge\min_{k=1,\dots,i}\{\olQ_1^i(x_0^{k+1})\}\ge\min_{x_0\in\calX_0}\olQ_1^i(x_0)>\epsilon$. Combining the above analysis, we have $\UB-\LB>\epsilon$ in iteration $i$. Therefore, we conclude that if $\UB-\LB\le\epsilon$ at the $i$-th iteration, then we have $i\ge\left(\frac{DLT}{4\epsilon}\right)^d$.
	\qed
\end{proof}

\subsubsection{Proof for Lemma~\ref{lemma:SphericalCap}}
\label{app:proof:lemma:SphericalCap}
% \begin{statement}
% 	Given a $d$-sphere $\calS^d(R),d\ge 2$ and depth $\beta<(1-\frac{\sqrt{2}}{2})R$, there exists a finite set of points $\calW$ with 
% 	$$\abs{\calW}\ge\frac{(d^2-1)\sqrt{\pi}}{d}\frac{\Gamma(d/2+1)}{\Gamma(d/2+3/2)}\left(\frac{R}{2\beta}\right)^{(d-1)/2},$$ 
% 	such that for any $w\in \calW$, $\calS^d_\beta(R,w)\cap \calW=\{w\}$.
% \end{statement}
\begin{proof}
	Let $v_d$ denote the $d$-volume for a $d$-dimensional unit ball.
	Recall that the $d$-volume of $\calS^d(R)$ is given by $\Vol_d(\calS^d(R))=(d+1)v_{d+1}R^d=\displaystyle\frac{(d+1)\pi^{(d+1)/2}}{\Gamma(\frac{d+1}{2}+1)}R^d$.
	We next estimate the $d$-volume for the spherical cap $\calS^d_{\beta}(R,x)$.
	Let $\alpha\in(0,\pi/2)$ denote the central angle for the spherical cap, i.e., $\cos\alpha=1-\beta/R$.
	Since $\beta<(1-\frac{\sqrt{2}}{2})R$, we know that $\alpha<\pi/4$.
	Then for any $x\in \calS^d(R)$, the $d$-volume of the spherical cap can be calculated through
	$$\Vol_d(\calS^d_\beta(R,x))=\int_{0}^{\alpha}\Vol_{d-1}(\calS^{d-1}(R\sin\theta))R\dif\theta=dv_dR^d\int_{0}^{\alpha}(\sin\theta)^{d-1}\dif\theta.$$
	Note that when $\theta\in(0,\alpha)$, $\sin\theta>0$ and $\cos\theta/\sin\theta>1$.
	Therefore, since $d\ge2$,
	$$\Vol_d(\calS^d_\beta(R,x))\le dv_dR^d\int_{0}^{\alpha}(\sin\theta)^{d-1}\frac{\cos\theta}{\sin\theta}\dif\theta=dv_dR^d\cdot\frac{(\sin\alpha)^{d-1}}{d-1}.$$
	By substituting $\sin\alpha=\sqrt{1-(1-\beta/R)^2}$, we have that 
	\begin{align*}
		\frac{\Vol_d(\calS^d_\beta(R,x))}{\Vol_d(\calS^d(R))}&\le\frac{d}{d^2-1}\frac{v_d}{v_{d+1}}(\sin\alpha)^{d-1},\\
		&=\frac{d}{d^2-1}\frac{v_d}{v_{d+1}}\left(1-\left(1-\frac{\beta}{R}\right)^2\right)^{(d-1)/2}
		\le\frac{d}{d^2-1}\frac{v_d}{v_{d+1}}\left(\frac{2\beta}{R}\right)^{(d-1)/2}.
	\end{align*}
	Now suppose $\calW=\{w_i\}_{k=1}^K$ is a maximal set satisfying the assumption, that is, for any $w\in\calS^d(R),w\notin\calW$, there exists $w_k\in\calW$ such that $w\in\calS^d_\beta(R,w_k)$.
	Then, $\bigcup_{k=1}^K\calS^d_\beta(R,w_k)\supseteq\calS^d(R)$, therefore
	$$\Vol_d(\calS^d(R))\le\sum_{k=1}^K\Vol_d(\calS^d_\beta(R,w_k))=\abs{\calW}\Vol_d(\calS^d_\beta(R,w_1)).$$
	Therefore we have
	\begin{align*}
	\abs{\calW}\ge\frac{\Vol_d(\calS^d(R))}{\Vol_d(\calS^d_\beta(R,w_1))}&\ge\left[\frac{d}{d^2-1}\frac{v_d}{v_{d+1}}\left(\frac{2\beta}{R}\right)^{(d-1)/2}\right]^{-1}\\
	&=\frac{(d^2-1)\sqrt{\pi}}{d}\frac{\Gamma(d/2+1)}{\Gamma(d/2+3/2)}\left(\frac{R}{2\beta}\right)^{(d-1)/2}.    
	\end{align*}
        This completes the proof.
	\qed
\end{proof}

\subsubsection{Proof for Lemma~\ref{lemma:SphericalCapConvexFunction}}
\label{app:proof:lemma:SphericalCapConvexFunction}
% \begin{statement}
% Given positive constants $\epsilon>0, L>0$ and a set $\calW_{\epsilon/L}^d(R)$. Let $K\coloneqq\vert{\calW^d_{\epsilon/L}(R)}\vert$.
% 	For any values $v_k\in(\epsilon/2,\epsilon)$, $k=1,\dots,K$, define a function $F:\calB^{d+1}(R)\to\R$ as $F(x)=\max_{k=1,\dots,K}\{0,v_k+\frac{L}{R}\innerprod{w_k}{x-w_k}\}$. Then $F$ satisfies the following properties:
% 	\begin{enumerate}
% 		\item $F$ is an $L$-Lipschitz convex function;
% 		\item $F(w_k)=v_k$ for all $w_k\in \calW^d_{\epsilon/L}(R)$;
% 		\item $F$ is differentiable at all $w_k$, with $v_k+\innerprod{\nabla F(w_k)}{w_l-w_k}< 0$ for all $l\neq k$; 
% 		\item For any $w_l\in\calW^d_{\epsilon/L}(R)$, $\ulQ_l(x)\coloneqq\max_{k\neq l}\{0,v_k+\innerprod{\nabla F(w_k)}{x-w_k}\}$ and $\olQ_l(x)\coloneqq\conv_{k\neq l}\{v_k+L\norm{x-w_k}\}$ satisfy
% 		$$\olQ_l(w_l)-\ulQ_l(w_l)>\frac{3\epsilon}{2}.$$
% 	\end{enumerate}
% \end{statement}
\begin{proof}
	\begin{enumerate}
	    \item By construction, $F$ is a convex piecewise linear function. Since each linear piece has a Lipschitz constant $(L/R)\|w_k\|=L$, as $\|w_k\|=R$. Thus, $F$, as the maximum of these linear functions, is also an $L$-Lipschitz function. 
	    \item Since $w_k\notin\calS^d_{\epsilon/L}(w_l)$ for $l\ne k$, $\innerprod{w_l}{w_k-w_l}<-\epsilon R/L$. Hence, $v_l + L/R\innerprod{w_l}{w_k-w_l}<v_l - \epsilon<0$. Therefore, $F(w_k)=\max\{0,v_k\}=v_k$. 
	    \item Notice that the above maximum for $F(w_k)$ is achieved at a unique linear piece, which implies that $F$ is differentiable at $w_k$ for all $k$. The gradient $\nabla F(w_k)=(L/R) w_k$. This gives the inequality in property 3 from the proof of property 2.
	    \item The inequality of property 3 also implies that $\ulQ_l(w_l)=0$. Now we show $\olQ_l(w_l)>3\epsilon/2$. Since $w_l\notin\calS^d_{\epsilon/L}(w_k)$ for any $k\ne l$, then $\|w_l-w_k\|>\epsilon/L$ by the Pythagorean theorem. Also $v_k>\epsilon/2$. So $q_k:=v_k + L\|w_l-w_k\|>3\epsilon/2$. Since $\olQ_l(w_l)$ is the convex combination of $q_k$'s, we have $\olQ_l(w_l)>3\epsilon/2$. 
	\end{enumerate}
    This completes the proof.
	\qed
\end{proof}

\subsubsection{Proof for Theorem~\ref{thm:ConvexLipschitzComplexity}}
\label{app:proof:thm:ConvexLipschitzComplexity}
% \begin{statement}
% 	For problem~\eqref{eq:SphericalCapExample}, if Algorithm~\ref{alg:DeterministicDualDP} gives $\UB-\LB<\epsilon$ at $i$-th iteration, then
% 	\begin{align*}
% 	    i &> \frac{1}{3}\frac{d(d-2)\sqrt{\pi}}{d-1}\frac{\Gamma(d/2+1/2)}{\Gamma(d/2+1)}\left(\frac{DL(T-1)}{8\epsilon}\right)^{(d-2)/2}.
% 	\end{align*}
% \end{statement}
\begin{proof}
	First we claim that in any iteration $i$, for any nodal problem $k$ in stage $t$, the optimal solution in the forward step (Algorithm~\ref{alg:DeterministicDualDP} line~\ref{alg:DeterministicDualDP:ForwardStep}) must be $x_t^i=w_{t,k}$.
	To see this, recall that we set $l_{n,\lambda}=L_t$ and $l_{n,\rho}=0$ for all $n\in\tcN(t)$, so by Proposition~\ref{prop:UnderApproximationValidness}, the under-approximation of the cost-to-go function $\ulcQ_t^i(x)$ is $L_{t+1}$-Lipschitz continuous for all iteration $i\in\bbN$.
	So consider the forward step subproblem for node $n=(t,k)$ with $t\ge 2$ in iteration $i$
	\begin{equation}\label{eq:ConvexLipschitzComplexity:ForwardStep}
		\min_{x_t\in\calX}\{F_t(x_{t-1}^i)+L_t\norm{x_t-w_{t,k}}+\ulcQ_t^i(x_t)\}=F_t(x_{t-1}^i)+\min_{x_t\in\calX}\{L_t\norm{x_t-w_{t,k}}+\ulcQ_t^i(x_t)\}.
	\end{equation}
	Note that by the $L_{t+1}$-Lipschitz continuity of $\ulcQ_t^i$, 
	$$L_t\norm{x_t-w_{t,k}}+\ulcQ_t^i(x_t)\ge\ulcQ_t^i(w_{t,k})+(L_t-L_{t+1})\norm{x_t-w_{t,k}},$$
	which, alongside with the fact that $L_{t+1}<L_t$, implies that $x_t=w_{t,k}$ is the unique optimal solution to the forward step problem~\eqref{eq:ConvexLipschitzComplexity:ForwardStep}.
	The above argument also works for any node in the stage $t=1$ by simply removing the constant term $F_t(x_{t-1}^i)$ in the nodal problem~\eqref{eq:ConvexLipschitzComplexity:ForwardStep}.

	Now we define over- and under-approximations of the value functions for the purpose of this proof.
	For node $n=(t,k)$, let 
	\begin{equation*}
		\ulQ_{t,k}^i(x):=\max_{1\le j\le i}\{0,C_{t,k}^j(x\given \hat{\lambda}_{t,k}^j,0,\ubar{v}_{t,k}^j)\},
	\end{equation*}
	and
	\begin{equation*}
		\olQ_{t,k}^i(x):=\conv_{1\le j\le i}\{\bar{v}_{t,k}^j+L_t\nVert{x_{t-1}^j-x}\},
	\end{equation*}
	where for each $j$, by formula~\eqref{eq:UnderApproximationUpdateFormula},
	\begin{align*}
		\ubar{v}_{t,k}^j&:=\max_{\norm{\lambda}\le L_t}\min_{z,x\in\calX}\{F_{t-1}(z)+\bangle{\lambda}{x_{t-1}^j-z}+L_t\norm{x-w_{t,k}}+\ulcQ_t^j(x)\},\\
		&=\max_{\norm{\lambda}\le L_t}\min_{z\in\calX}\{F_{t-1}(z)+\bangle{\lambda}{x_{t-1}^j-z}\}+\min_{x\in\calX}\{L_t\norm{x-w_{t,k}}+\ulcQ_t^j(x)\}\\
		&=F_{t-1}(x_{t-1}^j)+\ulcQ_{t}^j(w_{t,k})
	\end{align*}
	with $-\hat{\lambda}_{t,k}^j\in\partial F_{t-1}(x_{t-1}^j)$ being an optimal solution to the outer maximization problem, and by formula~\eqref{eq:OverApproximationUpdateFormula}
	\begin{align*}
		\bar{v}_{t,k}^j&:=F_{t-1}(z_{t,k}^j)+L_t\nVert{x_{t-1}^j-z_{t,k}^j}+\olcQ_{t}^j(x_{t,k}^j)\\
		&=F_{t-1}(x_{t-1}^j)+\olcQ_t^j(x_{t,k}^j).
	\end{align*}
	The last equalities of $\ubar{v}_{t,k}^j$ and $\bar{v}_{t,k}^j$ are due to the $L_t$-Lipschitz continuity of $F_{t-1}$.
	So we have by the monotonicity of the under- and over-approximations that for each $j\le i$,
	\begin{align*}
		\ulQ_{t,k}^i(x)&=\max_{1\le j\le i}\{0,F_{t-1}(x_{t-1}^j)+\bangle{\hat{\lambda}_{t,k}^j}{x_{t-1}^j-x}+\ulcQ_t^j(w_{t,k})\}\\
		&\le\max_{1\le j\le i}\{0,F_{t-1}(x_{t-1}^j)+\bangle{\hat{\lambda}_{t,k}^j}{x_{t-1}^j-x}\}+\ulcQ_t^i(w_{t,k}),
	\end{align*}
	and
	\begin{align*}
		\olQ_{t,k}^i(x)&=\conv_{1\le j\le i}\{F_{t-1}(x_{t-1}^j)+L_t\nVert{x_{t-1}^j-x}+\olcQ_t^j(w_{t,k})\}\\
		&\ge\conv_{1\le j\le i}\{F_{t-1}(x_{t-1}^j)+L_t\nVert{x_{t-1}^j-x}\}+\olcQ_t^i(w_{t,k}).
	\end{align*}
	Therefore,
	\begin{align}\label{eq:ConvexLipschitzComplexity:Recursion}
		\olQ_{t,k}^i(x)-&\ulQ_{t,k}^i(x)\ge\olcQ_t^i(w_{t,k})-\ulcQ_t^i(w_{t,k})\\
		&+\conv_{1\le j\le i}\{F_{t-1}(x_{t-1}^j)+L_t\nVert{x_{t-1}^j-x}\}-\max_{1\le j\le i}\{0,F_{t-1}(x_{t-1}^j)+\bangle{\hat{\lambda}_{t,k}^j}{x_{t-1}^j-x}\}.\notag
	\end{align}
	Thus we have $\olQ_{t,k}^i(w_{t-1,k'})-\ulQ_{t,k}^i(w_{t-1,k'})\ge \olcQ_t^i(w_{t,k})-\ulcQ_t^i(w_{t,k})$ for any $k'=1,\dots,K_{t-1}$, since the last two terms on the right hand side of \eqref{eq:ConvexLipschitzComplexity:Recursion} are over- and under-approximations of the function $F_{t-1}$, respectively.
	Moreover, note that $x_{t-1}^j=w_{t-1,k'}$ for some $k'=1,\dots,K_{t-1}$ as it is the unique solution in the forward step.
	By Lemma~\ref{lemma:SphericalCapConvexFunction}, whenever the node $n'=(t-1,k')$ is never sampled up to iteration $i$, we further have
	$$\olQ_{t,k}^i(w_{t-1,k'})-\ulQ^i_{t,k}(w_{t-1,k'})>\frac{3\epsilon}{2(T-1)}+\olcQ^i_{t}(w_{t,k})-\ulcQ^i_t(w_{t,k}).$$ 
	Recall the definitions~\eqref{eq:UnderApproximationCostToGo} and~\eqref{eq:OverApproximationCostToGo}, for any $x\in\calX$,
	\begin{equation*}
		\ulcQ_{t-1}^i(x)=\max_{1\le j\le i}\left\{0,\frac{1}{K_t}\sum_{k=1}^{K_t}C_{t,k}^j(x\given \hat{\lambda}_{t,k}^j,0,\ubar{v}_{t,k}^j)\right\}\le\frac{1}{K_t}\sum_{k=1}^{K_t}\ulQ_{t,k}^i(x),
	\end{equation*}
	and 
	\begin{equation*}
		\olcQ_{t-1}^i(x)=\conv_{1\le j\le i}\left\{\frac{1}{K_t}\sum_{k=1}^{K_t}(\bar{v}_{t,k}^j+L_t\nVert{x_{t-1}^j-x})\right\}\ge\frac{1}{K_t}\sum_{k=1}^{K_t}\olQ_{t,k}^i(x).
	\end{equation*}
	Consequently, for any $k'=1,\dots,K_{t-1}$,
	\begin{equation*}
		\olcQ^i_{t-1}(w_{t-1,k'})-\ulcQ^i_{t-1}(w_{t-1,k'})\ge\frac{1}{K_{t}}\sum_{k=1}^{K_{t}}[\olcQ^i_t(w_{t,k})-\ulcQ^i_t(w_{t,k})],
	\end{equation*}
	and in addition, for any node $n'=(t-1,k')$ not sampled up to iteration $i$, 
	\begin{equation*}
	\olcQ^i_{t-1}(w_{t-1,k'})-\ulcQ^i_{t-1}(w_{t-1,k'})>\frac{3\epsilon}{2(T-1)}+\frac{1}{K_{t}}\sum_{k=1}^{K_{t}}[\olcQ^i_t(w_{t,k})-\ulcQ^i_t(w_{t,k})].
	\end{equation*}
	Therefore, for any iteration index $i\le\frac{1}{3}\abs{\calW_t}$, $t=1,\dots,T-1$, then there are $K_t-i\ge\frac{2}{3}\abs{\calW_t}$ nodes not sampled in stage $t$, which implies
	\begin{equation*}
		\frac{1}{K_{t-1}}\sum_{k'=1}^{K_{t-1}}[\olcQ^i_{t-1}(w_{t-1,k'})-\ulcQ^i_{t-1}(w_{t-1,k'})]>\frac{\epsilon}{T-1}+\frac{1}{K_{t}}\sum_{k=1}^{K_{t}}[\olcQ^i_t(w_{t,k})-\ulcQ^i_t(w_{t,k})].
	\end{equation*}
	Consequently, $\olcQ^i_r-\ulcQ^i_r>(T-1)\cdot\dfrac{\epsilon}{T-1}=\epsilon.$
	Therefore, if $\UB-\LB=\olcQ^i_r-\ulcQ^i_r\le\epsilon$ in the iteration $i$, then
	\begin{align*}
		i&>\frac{1}{3}\min_{t=1,\dots,T-1}\abs{\calW_t}
		\ge\frac{1}{3}\frac{d(d-2)\sqrt{\pi}}{d-1}\frac{\Gamma(d/2+1/2)}{\Gamma(d/2+1)}\left(\frac{DL(T-1)}{8\epsilon}\right)^{(d-2)/2}.
	\end{align*}
	This completes the proof.
	\qed
\end{proof}

\section{Problem Classes with Exact Penalization}
\label{sec:ExactPenaltyProblemClasses}
In this section, we discuss the problem classes that allows exact penalty reformulation, as stated in Assumption~\ref{assum:ExactPenalty}.
A penalty function $\psi:\R^d\to\R_+$ is said to be sharp, if $\psi(x)\ge c\norm{x}$ for all $x\in V\subset\R^d$, for some open neighborhood $V\ni 0$ and some positive scalar $c>0$.
% A simple observation about problems with Lipschitz value functions is made as follows.
% \begin{proposition}\label{prop:InfConvolutionExactLipschitzValueFunctions}
% 	Suppose $Q_n(x)$ is $L_n$-Lipschitz continuous for all $n\in\calN$.
% 	If the penalty functions $\psi_n$ are $L_n$-sharp for all nodes $n\in\calN$, then $Q_n(x)=Q_n^\Reg(x)$ for any state $x\in\calX_{a(n)}$, for all nodes $n\in\calN$.
% \end{proposition}

\subsection{Convex problems with interior points.}
For convex problems, the Slater condition implies that the intersection of the domain $\dom(\sum_{n\in\calN} f_n)$ and the feasible sets $\Pi_{n\in\calN}\calF_n$ has a non-empty interior.
We have the following proposition.
\begin{proposition}
	\label{prop:ExactPenalConvex}
	Suppose the problem~\eqref{eq:ExtensiveForm} is convex and satisfies the Slater condition
        For any sharp penalty functions $\psi_n$, there exist $\sigma_n>0$ such that the penalty reformulation is exact.
\end{proposition}
\begin{proof}
	Consider a perturbation vector $w=(w_n)_{n\in\calN}$ such that $w_n\in\calX_{a(n)}-\calX_{a(n)}$ for each $n\in\cN$, %$e=(a(n),n)\in\calE$,
	and define the perturbation function
	\begin{align*}
		&\tau(w)\coloneqq\min_{(z_n,x_n,y_n)\in\calX_{a(n)}\times\calF_n}\biggl\{\sum_{n\in\calN}p_{n}f_{n}(z_n,y_n,x_n)\biggm\vert w_n=x_{a(n)}-z_n,\,\forall\,n\in\calN\biggr\}.
	\end{align*}
	The function $\tau$ is convex  and $v^\primal=\tau(0)$ by definition.
	By the Slater condition, $0\in\mathrm{int}(\dom(\tau))$ so there exists a vector $\lambda\in\R^{\abs{\calN}}$ such that $\tau(w)\ge\tau(0)+\innerprod{\lambda}{w}$ for all perturbation $w$.
	Since $\psi_n$ are sharp, there exist $\sigma_n>0$ such that $\sum_{n\in\calN}\sigma_n\psi_n(w_n)+\innerprod{\lambda}{w}>0$ for all $w\neq 0$.
	Consequently the penalty reformulation is exact since $v^\reg=\min_{w}\tau(w)+\sum_{n\in\calN}\sigma_n\psi_n(w_n)$ and all optimal solutions must satisfy $w_n=x_{a(n)}-z_n=0$ for all $n\in\calN$.
	\qed
\end{proof}

\subsection{Problems with finite state spaces.}
We say a problem~\eqref{eq:NodalProblem} has finite state spaces if the state spaces $\calX_n$ are finite sets for all nodes $n$.
Such problems appear in multistage integer programming~\cite{zou_stochastic_2018}, or when the original state spaces can be approximated through finite ones~\cite{zou_multistage_2019,hjelmeland_nonconvex_2019}.
The following proposition shows the penalty reformulation is exact whenever the state spaces are finite.
\begin{proposition}
	\label{prop:ExactPenalFiniteState}
	For any penalty functions $\psi_n$, $n\in\calN$, if the state spaces are finite, then there exists a finite $\sigma_n>0$ such that the penalty reformulation \eqref{eq:PenalizedExtensiveForm} is exact.
\end{proposition}
\begin{proof}
	Let $d_n\coloneqq\min_{x\neq z\in\calX_{a(n)}}\abs{\psi_n(x-z)}$ for each $n\in\cN$.
	Since $\psi_n$ is a penalty function and the state space $\calX_n$ is finite, we know $d_n>0$.
	Define $c$ as
	\begin{align}
	    c:=\min_{(z_n,y_n,x_n)\in\cX_{a(n)}\times\cF_n}\;\sum_{n\in\cN} p_n f_n(z_n,y_n,x_n).\label{eq:PenaltyFinitespaceDefc}
	\end{align}
	Since \eqref{eq:PenaltyFinitespaceDefc} is a relaxation of the original problem \eqref{eq:ExtensiveForm} by ignoring coupling constraint $z_n = x_{a(n)}$, then $c\le v^{\primal}$.
	We choose
	$\sigma_n = 1 + (v^\primal-c)/(p_n d_n)$ for all $n\in\cN$.\\
	Now let $(x_n,y_n,z_n)_{n\in\calN}$ be an optimal solution to the regularized problem~\eqref{eq:RegularizedNodalProblem}.
	Then if there exists $x_{a(m)}\neq z_m$ for some $m\neq r$, then $p_m\sigma_m\psi_m(x_{a(m)}-z_m)>v^\primal-c$.
	Consequently,
	\begin{align*}
		v^\reg&\ge c+\sum_{n\in\calN}p_n\sigma_n\psi_n(x_{a(n)}-z_n)
		\ge c + p_m\sigma_m\psi_m(x_{a(m)}-z_m)
		> c+ v^\primal - c
		=v^\primal.
	\end{align*}
	This is a contradiction since $v^\reg\le v^\primal$.
	Thus, any optimal solution to the reformulation~\eqref{eq:PenalizedExtensiveForm} must have $x_{a(n)}=z_n$ for all $n\neq r$, which means the penalty reformulation is exact.
	\qed
\end{proof}

\subsection{Problems defined by mixed-integer linear functions.}
The problem~\eqref{eq:ExtensiveForm} is said to be defined by mixed-integer linear functions, if all the feasible sets $\calF_n$ and the epigraphs $\mathrm{epi}f_n$ are representable by mixed-integer variables and non-strict linear inequalities with rational coefficients.
Recall that by Assumption~\ref{assum:MinCondition}, the primal problem is feasible, $v^\primal>-\infty$.
We have the following proposition on the exact penalty reformulation.
\begin{proposition}[\cite{feizollahi2017exact}, Theorem 5]
	\label{prop:ExactPenalMILP}
	If problem~\eqref{eq:ExtensiveForm} is defined by mixed-integer linear functions and the penalty functions $\psi_n$ are sharp for all $n\in\calN$, then there exist $\sigma_n>0$, such that the penalty reformulation is exact.
\end{proposition}

\subsection{Problems defined by $C^1$-functions.}
The problem~\eqref{eq:ExtensiveForm} is said to be defined by $C^1$-functions if it is defined by functional constraints using indicator functions in each node $n\in\calN$:
$$f_n(x_{a(n)},y_n,x_n)=\begin{cases}
&f_{n,0}(x_{a(n)},y_n,x_n),\quad\text{if }g_{n,i}(x_{a(n)},y_n,x_n)\le0,i=1,\dots,I_n,\\
&+\infty\quad\text{otherwise}.\end{cases}$$
with all $f_{n,0},g_{n,i},i=1,\dots,I_n$ being continuously differentiable.
The Karush-Kuhn-Tucker condition at a feasible point $(x_n,y_n)_{n\in\calN}$ of \eqref{eq:ExtensiveForm} says that there exist multipliers $\mu_{n,i}\ge0,i=1,\dots,I_n$, such that
\begin{align*}
	&\nabla_{x_n,y_n}\left\{\sum_{n\in\calN}(f_{n,0}(x_{a(n)},y_n,x_n)-\mu_{n,i}g_{n,i}(x_{a(n)},y_n,x_n)\right\}=0,\\
	&\mu_{n,i}g_{n,i}(x_{a(n)},y_n,x_n)=0,\quad i=1,\dots,I_n.
\end{align*}
We have the following proposition on the exactness.
\begin{proposition}
	\label{prop:ExactPenalC1KKT}
	Suppose the problem~\eqref{eq:ExtensiveForm} is defined by $C^1$-functions and the Karush-Kuhn-Tucker condition holds for every local minimum solution of \eqref{eq:ExtensiveForm}.
	If the penalty functions $\psi_n$ are sharp for all $n\in\calN$, then there exist $\sigma_n>0$ such that the penalty reformulation is exact.
\end{proposition}
We give the proof of Proposition~\ref{prop:ExactPenalC1KKT} below.

\subsubsection{Proof for Proposition~\ref{prop:ExactPenalC1KKT}.}
\label{subsec:ProofExactPenalC1KKT}

We begin by stating a general exact penalization result for problems defined by $C^1$-functions.
Consider the following perturbation function
\begin{align}
	p(u)\coloneqq\min_{x\in\R^d}\quad& f(x,u)\label{app:eq:PenaltyExtensiveForm}\\
	\mathrm{s.t.}\quad&g_i(x,u)\le 0,\quad i=1,\dots,I,\notag\\
	& h_j(x,u)=0,\,\quad j=1,\dots,J,\notag
\end{align}
Here $u$ is the perturbation vector and $u=0$ corresponds to the original primal problem. 
Let $\psi$ be a penalty function on $\R^d$ and $\sigma>0$ a penalty factor.
A penalization of the original primal problem $p(0)$ is given by
\begin{align}
\min_{x\in\R^d}\quad& f(x,u)+\sigma\psi(u)\label{app:eq:PenaltyReformulation}\\
\mathrm{s.t.}\quad& g_i(x,u)\le0,\,\quad i=1,\dots,I,\notag\\
&h_j(x,u)= 0,\quad j=1,\dots,J.\notag
\end{align}
Naturally we could impose some bound on the perturbation as $\norm{u}\le R_u$. 
We assume that $f,g_i,h_j$ are continuously differentiable in $x$ and $u$ for all $i,j$.
Moreover, the compactness in Assumption~\ref{assum:MinCondition} implies that the feasible region prescribed by the inequality constraints $g_i(x,u)\le0$ are compact in $x$ for any $u$, i.e., $X=\{x\in\R^d:\exists u,\norm{u}\le R_u,\mathrm{s.t.}g_i(x,u)\le0,i=1,\dots,J\}$ is compact.
For example, some of the inequalities are bounds on the variables, $\norm{x}_\infty\le1$. 
We will show that there exists a penalty factor $\sigma>0$ such that any optimal solution to \eqref{app:eq:PenaltyReformulation} is feasible to \eqref{app:eq:PenaltyExtensiveForm}.
We next characterize the property of the perturbation function $p(u)$.
\begin{lemma}\label{app:lemma:PrimalFunctonalLSC}
	The perturbation function $p(u)$ is lower semicontinuous.
\end{lemma}
\begin{proof}
	Let $X(u)\subset X$ denote the feasible set in $x$ dependent on $u$.
	The minimum in the definition is well defined for every $u$ due to the compactness of $X(u)$.\\
	We show that $p(u)$ is lower semicontinuous (lsc) by showing $\liminf_{v\to u}p(v)\ge p(u)$ for any $u$. 
	Assume for contradiction that for any $\epsilon>0$, there exists a sequence $\{v_k\}_{k=1}^\infty$ such that $v_k\to u$ and $p(v_k)\le p(u)-\epsilon$. 
	Let $x_k\in\argmin f(x,v_k)$ and thus $p(v_k)=f(x,v_k)$. 
	Since $X$ is compact, there exists a subsequence $x_{k_j}$ and $z\in X$ such that $x_{k_j}\to z$ as $j\to\infty$.
	Then by continuity of $f$, $f(z,u)=\lim_{j\to\infty}f(x_{k_j},v_{k_j})\le p(u)-\epsilon$. 
	This contradicts with the definition of $p(u)$, since $p(u)=\min_{x\in X(u)}f(x,u)\le f(z,u)\le p(u)-\epsilon$. 
	Therefore $p(u)$ is lsc.
	\qed
\end{proof}

Now we give the theorem of exact penalization for problems defined by $C^1$-functions. 
\begin{proposition}
	\label{app:prop:ExactPenalty}
	If the Karush-Kuhn-Tucker condition is satisfied at every local minimum solution of \eqref{app:eq:PenaltyExtensiveForm}, then the penalty reformulation \eqref{app:eq:PenaltyReformulation} is exact for some finite $\sigma>0$.
\end{proposition}
\begin{proof}
	Let $X(u)$ denote the feasible region of $x$ defined by constraints $g_i(x,u)\le0,\ i=1,\dots,I$ and $h_j(x,u)=0,\ j=1,\dots,J$. 
	Then $X(u)$ is compact for any $u$ by the continuity of the constraint functions. 
	We show that for every optimal solution $x_0\in X(0)$, there exists a neighborhood $V(x_0)\ni x_0$ in the $x$ space, $U(x_0)\ni u=0$ in the $u$ space, and constant $L(x_0)>0$, such that for all $x\in V(x_0)$ and $u\in U(x_0)$, we have
	$f(x,u)\ge f(x_0,0)-L(x_0)\cdot\norm{u}$.
	Then we use this fact together with compactness of $X(0)$ to show the existence of exact penalization. In this proof, the little-$o$ is used to simplify notation, i.e., $o(\norm{a})$ denotes a function $b(a)$ such that
	 $\lim_{a\to0}\abs{b(a)}/\norm{a}=0.$
	
	Pick any optimal solution $x_0\in X(0)$. 
        By definition, it is also a local minimum solution. 
        By hypothesis, the KKT condition is satisfied at $x_0$, that is, there exist $\lambda_i\in\R,\ i=1,\dots,I,$ and $\mu_j\ge 0,\ j=1,\dots,J$ such that
	\begin{align}
		&\nabla_x f(x_0,0)+\sum_{i=1}^{I}\lambda_i \nabla_x g_i(x_0,0)+\sum_{j=1}^J \mu_j\nabla_x h_j(x_0,0)=0,\label{app:eq:PenaltyKKT}\\
		&h_j(x_0,0)=0,\ j=1,\dots,J,\notag\\
		&g_i(x_0,0)\le0,\quad \lambda_i\cdot g_i(x_0,0)=0,\ i=1,\dots,I.\notag
	\end{align}
	Since $h_j$'s are continuously differentiable and $h_j(x_0,0)=0$, we have
	\begin{equation}
		\innerprod{\nabla_x h_j(x_0,0)}{x-x_0}+\innerprod{\nabla_u h_j(x_0,0)}{u}+o(\norm{x-x_0}+\norm{u})=0,\ j=1,\dots,J.\label{app:eq:PenaltyFirstOrderEquality}
	\end{equation}
	Let $A\subset I$ denote the set of active inequality constraints. Then similarly we have
	\begin{equation}
		\innerprod{\nabla_x g_i(x_0,0)}{x-x_0}+\innerprod{\nabla_u g_i(x_0,0)}{u}+o(\norm{x-x_0}+\norm{u})\le0,\ i\in A.\label{app:eq:PenaltyFirstOrderInequality}
	\end{equation}
	For any $i\notin A$, by the continuity of $g_i$, there exist neighborhoods $W_i$ of $x_0$ and $U'_i$ of $u=0$ such that for any $(x,u)\in W_i\times U'_i$, $g_i(x,u)<0$ remains inactive. 
	Now, from \eqref{app:eq:PenaltyKKT}, \eqref{app:eq:PenaltyFirstOrderEquality}, \eqref{app:eq:PenaltyFirstOrderInequality}, and $f$ being continuously differentiable, we have
	\begin{align*}
		&f(x,u)-f(x_0,0)\notag\\
		&=\innerprod{\nabla_x f(x_0,0)}{x-x_0}+\innerprod{\nabla_u f(x_0,0)}{u}+o(\norm{x-x_0}+\norm{u}) \\
		% &=\bangle{-\sum_{i=1}^{I}\lambda_i \nabla_x g_i(x_0,0)-\sum_{j=1}^J \mu_j\nabla_x h_j(x_0,0)}{x-x_0}+\innerprod{\nabla_u f(x_0,0)}{u}+o(\norm{x-x_0}+\norm{u}) \notag\\
		&=\bangle{-\sum_{i\in A}\lambda_i \nabla_x g_i(x_0,0)-\sum_{j=1}^J\mu_j\nabla_x h_j(x_0,0)}{x-x_0}+\innerprod{\nabla_u f(x_0,0)}{u}+o(\norm{x-x_0}+\norm{u}) \notag\\
		&\ge\bangle{\nabla_u f(x_0,0)+\sum_{i\in A}\lambda_i \nabla_u g_i(x_0,0)+\sum_{j=1}^J\mu_j\nabla_u h_j(x_0,0)}{u}+o(\norm{x-x_0}+\norm{u}) \notag\\
		&>-L(x_0)\cdot\norm{u}+o(\norm{x-x_0}+\norm{u}),\notag
	\end{align*}
	where $L(x_0)\coloneqq\nVert{\nabla_u f(x_0,0)+\sum_{i\in A}\lambda_i \nabla_u g_i(x_0,0)+\sum_{j=1}^J\mu_j\nabla_u h_j(x_0,0)}+1>0$. 
	By the definition of the little-$o$ notation, there exists a neighborhood $V(x_0)\subset\cap_{i\notin A}W_i,\ x_0\in V(x_0)$ and $U(x_0)\subset\cap_{i\notin A}U_i,\ 0\in U(x_0)$ such that 
	\begin{equation*}
		f(x,u)-f(x_0,0)\ge -L(x_0)\cdot\norm{u},\quad\forall\,(x,u)\in V(x_0)\times U(x_0).\label{app:eq:PenaltyKeyInequality}
	\end{equation*}

	Now, let $X_{\opt}(0)$ denote the set of optimal solutions of $x$ when $u=0$. 
	Note that $X_{\opt}(0)\subset X(0)$ is closed due to the continuity of $f,h_i,g_j$, hence compact. 
	The collection of open sets $\{V(x)\}_{x\in X_{\opt}(0)}$ covers $X_{\opt}(0)$. 
	By compactness, there exists a finite subcollection $\{V(x_k)\}_{k=1}^K$ such that $X_{\opt}(0)\subset\cup_{k=1}^K V(x_k)\eqqcolon V$. 
	Let $L\coloneqq\max_{k=1,\dots,K}L(x_k)$ and $U=\cap_{k=1}^K U(x_k)$. Let $f^*$ denote the optimal value for $u=0$. Then we have
	\begin{equation*}
		f(x,u)\ge f^*- L\cdot\norm{u},\quad\forall\,(x,u)\in V\times U. 
	\end{equation*}
	To show the inequality for $x\notin V$, define
        $\tilde{p}(u):=\min_{x\in X(u)\setminus V}f(x,u)$.
	Note that $\tilde{p}(0)>f^*$ by the definition of $X_{\opt}(0)$.
	Then by Lemma~\ref{app:lemma:PrimalFunctonalLSC}, $p(u)$ is lower semicontinuous, and we know that there exists a neighborhood $U'$ of $0$ such that $\tilde{p}(u)>f^*$ for all $u\in U'$. 
	Therefore, for all $u\in U\cap U'$, we have
        $f(x,u)\ge f^*- L\cdot \norm{u}$.
	Finally, we can show that the penalization is exact. 
	Since $\psi$ is sharp, there exist an open set $\tilde{U}\subset U\cap U'$, and positive constants $c>0$ such that 
	$$\psi(u)\ge c\norm{u}\text{ on }\tilde{U}.$$
	Let $M=\min_{u\in \bar{B}_{R_u}(0)\setminus\tilde{U}}\tilde{p}(u)>f^*$, $m=\min_{u\in \bar{B}_{R_u}(0)\setminus\tilde{U}}\psi(u)>0$ because $\psi$ is a penalty function. 
	Let $\sigma=(M-f^*)/m+1$. 
	We have
        $f(x,u)\ge f^*- \sigma\cdot \norm{u},\quad\forall\,u\in\bar{B}_{R_u}(0)\setminus\{0\},\ x\in \cup_u X(u)$.
        As a result, any optimal solution to the penalization~\eqref{app:eq:PenaltyReformulation} would satisfy $u=0$.
	\qed
\end{proof}

Note that our problem~\eqref{eq:PenalizedExtensiveForm} can be written into the form~\eqref{app:eq:PenaltyReformulation} by letting $u=(x_{a(n)},z_n)_{n\in\calN}$, and including the duplicate constraints
$z_n-x_{a(n)}=0$ for any $n\neq r\in\calN$
in the equality constraints $h_j(x,u)=0$.
And other constraints $g_i(x,u)\le0$ correspond to the functional constraints in the problem~\eqref{eq:ExtensiveForm}.
Since $\psi_n$ are sharp, the aggregate penalty function defined by
$\psi(u)=\sum_{n\in\calN}p_n\psi_n(x_{a(n)}-z_n)$
is also sharp.
Let $\sigma$ denote the penalty factor in Proposition~\ref{app:prop:ExactPenalty}.
Proposition~\ref{prop:ExactPenalC1KKT} follows from this by letting $\sigma_n=\sigma/p_n$ for all $n\in\calN$.

\end{document}